\documentclass{article}
\usepackage{psfrag}
\usepackage[parfill]{parskip}
\usepackage{float, graphicx}
\usepackage[]{epsfig}
\usepackage{amsmath, amsthm, amssymb}
\usepackage{epsfig}
\usepackage{verbatim}
\usepackage{multicol}
\usepackage{url}
\usepackage{latexsym}
\usepackage{mathrsfs}
\usepackage[colorlinks, bookmarks=true]{hyperref}
\usepackage{graphicx}
\usepackage{amsmath}
\usepackage{enumerate}
\usepackage[normalem]{ulem}
\usepackage{bm}
\usepackage{dsfont}
\usepackage{stmaryrd}
\usepackage{tikz}
\usepackage[T1]{fontenc}
\usepackage{cite}

\usepackage{color}

\makeatletter
\def\thm@space@setup{%
  \thm@preskip=\parskip \thm@postskip=0pt
}
\makeatother

\numberwithin{equation}{section}

\renewcommand{\cal}{\mathcal}

\newcommand{\cC}{{\cal C}}
\newcommand{\cD}{\cal D}
\newcommand{\cE}{{\cal E}}
\newcommand{\cG}{{\cal G}}

\newcommand{\cM}{{\cal M}}

\newcommand{\cR}{{\mathcal R}}

\newcommand{\fa}{{\frak a}}
\newcommand{\fb}{{\frak b}}
\newcommand{\fc}{{\frak c}}
\newcommand{\fd}{{\frak d}}

\newcommand{\fm}{{\frak m}}

\newcommand{\rd}{{\rm d}}

\newcommand{\ri}{\mathrm{i}}

\newcommand{\bC}{{\mathbb C}}
\newcommand{\bE}{\mathbb{E}}

\newcommand{\bP}{\mathbb{P}}

\newcommand{\bR}{{\mathbb R}}

\newcommand{\bZ}{\mathbb{Z}}
\newcommand{\bW}{\mathbb{W}}
\newcommand{\bK}{\mathbb{K}}
\newcommand{\bM}{\mathbb{M}}

\newcommand{\al}{\alpha}

\newcommand{\la}{\lambda}

\newcommand{\veps}{\varepsilon}

\DeclareMathOperator{\supp}{supp}
\DeclareMathOperator{\dist}{dist}

\DeclareMathOperator{\OO}{O}
\DeclareMathOperator{\oo}{o}

\DeclareMathOperator{\argmin}{argmin}

\renewcommand{\Re}{\mathop{\mathrm{Re}}}
\renewcommand{\Im}{\mathop{\mathrm{Im}}}

\newcommand{\deq}{\mathrel{\mathop:}=} 
 
\renewcommand{\leq}{\leqslant}
\renewcommand{\geq}{\geqslant}

\newcommand{\td}{\tilde}
\newcommand{\del}{\partial}

\newcommand{\qq}[1]{[\![{#1}]\!]}

\newcommand{\beq}{\begin{equation}}
\newcommand{\eeq}{\end{equation}}

\theoremstyle{plain} 
\newtheorem{theorem}{Theorem}[section]
\newtheorem*{theorem*}{Theorem}
\newtheorem{lemma}[theorem]{Lemma}
\newtheorem*{lemma*}{Lemma}
\newtheorem{corollary}[theorem]{Corollary}
\newtheorem*{corollary*}{Corollary}
\newtheorem{proposition}[theorem]{Proposition}
\newtheorem*{proposition*}{Proposition}
\newtheorem{assumption}[theorem]{Assumption}
\newtheorem*{assumption*}{Assumption}
\newtheorem{claim}[theorem]{Claim}

\newtheorem{definition}[theorem]{Definition}
\newtheorem*{definition*}{Definition}

\newtheorem*{example*}{Example}
\newtheorem{remark}[theorem]{Remark}

\newtheorem*{remark*}{Remark}
\newtheorem*{remarks*}{Remarks}

\def\author#1{\par
    {\centering{\authorfont#1}\par\vspace*{0.05in}}
}

\def\titlefont{\fontsize{13}{15}\bfseries\boldmath\selectfont\centering{}}
\def\authorfont{\fontsize{13}{15}}

\let\affiliationfont\rhfont

\def\address#1{\par
    {\centering{\affiliationfont#1\par}}\par\vspace*{11pt}
}

\def\body{
\setcounter{footnote}{0}
\def\thefootnote{\alph{footnote}}
\def\@makefnmark{{$^{\rm \@thefnmark}$}}
}

\def\title#1{
    \thispagestyle{plain}
    \vspace*{-14pt}
    \vskip 79pt
    {\centering{\titlefont #1\par}}%
    \vskip 1em
}

\renewcommand{\i}{{\rm{i}}}

\newcommand{\ext}{{\rm{ext}}}
\newcommand{\dual}{{\rm dual}}

\newcommand{\loc}{{\rm{dis}}}
\newcommand{\cont}{{\rm cont}}

\begin{document}
\title{Rigidity and Edge Universality of Discrete $\beta$-Ensembles}

\vspace{1.2cm}

 \begin{minipage}[c]{0.5\textwidth}
 \author{Alice Guionnet}
\address{CNRS \& Universit\'e de Lyon, ENSL, UMPA, UMR 5669\\
   E-mail: alice.guionnet@ens-lyon.fr}
 \end{minipage}
 \begin{minipage}[c]{0.5\textwidth}
 \author{Jiaoyang Huang}
\address{Harvard University\\
   E-mail: jiaoyang@math.harvard.edu}
 \end{minipage}

~\vspace{0.3cm}

\begin{abstract}
We study discrete $\beta$-ensembles as introduced in \cite{Borodin2016}. We obtain rigidity estimates on the particle locations, i.e. with high probability, the particles are close to their classical locations with an optimal error estimate. We prove the edge universality of the discrete $\beta$-ensembles, i.e. for $\beta\geq 1$, the distribution of extreme particles converges to the Tracy-Widom $\beta$ distribution. As far as we know, this is the first proof of general Tracy-Widom $\beta$ distributions in the discrete setting.A special case of our main results implies that under the Jack deformation of the Plancherel measure, the distribution of the lengths of the first few rows in Young diagrams, converges to the Tracy-Widom $\beta$ distribution, which answers an open problem in \cite{MR3498866}.
Our proof relies on Nekrasov's (or loop) equations, a multiscale analysis  and a comparison argument with continuous $\beta$-ensembles.

\end{abstract}


\begingroup
\hypersetup{linkcolor=black}
 \tableofcontents
\endgroup

\let\thefootnote\relax\footnote{\noindent This work of A.G. is partially  supported by the LABEX MILYON (ANR-10-LABX-0070) of Universit\'e de Lyon. }

\date{\today}

\vspace{-0.7cm}

\section{Introduction}

During the eighteenth century, De Moivre was the first to show that the Gaussian (or Normal) distribution describes the fluctuations of the sums of independent binomial variables. Soon after, Laplace  generalized his result to more general independent variables. Gauss advertised this central limit theorem by showing that it allowes the evaluation of the errors in a large class of systems characterized by independence, and many other proofs of the central limit theorems were given, see e.g. \cite{lind,stein, bill, dur}. During the last twenty years, understanding the fluctuations of much more correlated systems became of central interest. Instead of the Gaussian distribution, some of these systems  kept producing another statistical curve, which had become known as the Tracy-Widom distribution. A typical example concerns the eigenvalues of random matrices with independent entries (up to the symmetry constraint). The eigenvalues are far from being independent and are in general ordered. The fluctuations of the extreme eigenvalues are described by the Tracy-Widom distribution with an exotic scaling exponent two-third. Such results were first proven for very specific examples that are integrable, namely the Gaussian ensembles \cite{TRW96,TW1,ME,forr}. It was much more recently proven that these results are universal, in the sense  that the local fluctuations do not depend too much on the microscopic details of the random matrices, i.e. single entry distribution, sparsity, see e.g. \cite{MR1727234, MR2669449, MR3161313, MR2964770,
MR3034787, Lee2016}.  It is worth noting that this universality is always proven by comparison to integrable models and there is not yet a direct approach to this problem for general models.
Besides matrix models, the Tracy-Widom distributions also appear in a large class of models which a priori are very different. In the beautiful article  \cite{MR1682248}, it was shown that the length of the longest increasing subsequence of a uniform random permutation is also asymptotically described by the Tracy-Widom distribution. Other highly correlated systems such as uniform random lozenge tiling \cite{JoArct}, interacting particle systems \cite{MR3340328,MR2520510,MR3265169}, polymers in disordered environments \cite{MR2917766, MR3116323} and the corner growth model \cite{MR1737991}, fluctuate following the Tracy-Widom distribution. Investigating the universality class of the Tracy-Widom distribution has become a key issue in probability theory. This article investigates the universality class of  the so-called Tracy-Widom $\beta$ distributions, which has been introduced and studied in \cite{RRV}, characterizing the fluctuations of the extreme particles of Gaussian $\beta$-ensembles.  We show that a natural family of discrete models, the discrete $\beta$-ensembles as introduced in \cite{Borodin2016}, and the Jack deformation of the Plancherel measure \cite{MR2227852} belong to the universality class of the Tracy-Widom $\beta$ distributions.

We fix the parameter $\theta=\beta/2$ and a sequence of positive real-valued functions $w(x;N)$. A discrete $\beta$-ensemble is a sequence of distributions 
\begin{align}{\label{eqn:distr}}
\bP_N(\ell_1, \ell_2, \cdots, \ell_N)=\frac{1}{Z_N} \prod_{1\leq i<j\leq N}\frac{\Gamma(\ell_j-\ell_i+1)\Gamma(\ell_j-\ell_i+\theta)}{\Gamma(\ell_j-\ell_i)\Gamma(\ell_j-\ell_i+1-\theta)}
\prod_{i=1}^N w(\ell_i; N),
\end{align}
on ordered $N$-tuples $\ell_1<\ell_2<\cdots<\ell_N$, such that $\ell_i=a(N)+\lambda_i+\theta i$ and $0< \lambda_1\leq \lambda_2\leq \cdots\leq \lambda_N<b(N)-a(N)-N\theta$ are integers. We refer to $w(x;N)$ as the weight of the discrete $\beta$-ensemble, and $\ell_i$ as the positions of particles. We denote the configuration space by $\bW_N^\theta$.  When $\theta$ equals one, the particles live on $\mathbb Z$, but in general  the possible positions of a particle depend on the position of the previous particle, and they do not live on a fixed lattice. Let us note that when $\theta=1$ or $1/2$  the above ratio of gamma functions is simply the usual Coulomb gas interaction $\prod |\ell_i-\ell_j|^{2\theta}$.

When $\theta$ equals one, the distribution \eqref{eqn:distr} arises in numerous problems of statistical mechanics and asymptotic representation theory, e.g., the distribution of
 uniformly random domino tilings of the
Aztec diamond (cf.\ \cite{Jo02}), the distribution of
 uniformly random lozenge tilings of an abc--hexagon (cf.\ \cite{Jo02}), last passage percolation (cf.\
\cite{Jo01}), stochastic systems of non--intersecting paths (cf.\
\cite{Jo02, KOR}), and the representation theory of the
infinite dimensional unitary group $U(\infty)$ (cf.\ \cite[Section 5]{BO1},
\cite[Section 4]{Bor}). One of the most heavily studied examples is the random lozenge tilings of an abc--hexagon. The distribution of the positions of horizontal lozenges on a vertical line is given by such a discrete $\beta$-ensemble with $\theta=1$ and $w(x;N)$ certain binomial weight \cite{Jo01,Jo02}; it is called  the Hahn ensemble. In fact, more general domains can be handled provided they can be  described in terms of Gelfand-Tsetlin patterns \cite{MR1641839, MR2454474, MR3178541,MR3278913}. Choosing  uniformly random lozenge tilings of those domains, the distributions of the positions of horizontal lozenges on a vertical line are given by discrete $\beta$-ensemble with $\theta=1$ and more complicated binomial weights, e.g., \cite[Section 9.2]{Borodin2016}. { For general $\theta>0$ and special weights, the distribution \eqref{eqn:distr} appears in 
\cite{MR2083231}, called the $zw$-measures. The normalization constants in the $zw$-measures are explicit, which could be interpreted as that the special form of the interaction in $\bP_N$, a ratio of Gamma functions, is a discrete analogue of the Selberg integral. Other families of related measures are the Jack deformation of the Plancherel measures on Young diagrams and, more generally, the $z$-measures \cite{MR2227852,MR2143199, MR2098845,MR2607329,MR3498866,MR2465773,MR2837720}. We remark that the Jack deformation of the Plancherel measures and $z$-measures on Young diagrams are slightly different from the discrete $\beta$-ensemble \eqref{eqn:distr}, where the number of particles is fixed.}

Moreover, the discrete $\beta$-ensemble \eqref{eqn:distr} is a natural discrete analogue of the continuous  $\beta$-ensemble given by
$$\rd\bP_{N}^{\cont}(\lambda_{1},\ldots,\lambda_{N})=\frac{1}{Z_{N}}\prod_{1\leq i<j\leq N}|\lambda_{i}-\lambda_{j}|^{\beta}e^{-N\sum_{i=1}^{N }V(\lambda_{i})}\prod_{i=1}^N \rd\lambda_{i}\,.$$
For classical values of  $\beta=1,2,4$ and $V(x)=\beta x^{2}/4$, $\bP_{N}^{\cont}$ corresponds to the joint law of the eigenvalues of the Gaussian Orthogonal (with real entries), Gaussian Unitary (with complex entries) or Gaussian Symplectic (with quaternion entries) Ensembles. 
When $V$ is quadratic, Dumitriu and Edelman \cite{DuEd} constructed  for any $\beta>0$  a real symmetric tri-diagonal matrix with independent chi-square and Gaussian entries, whose eigenvalues are distributed according to $\bP_{N}^{\cont}$.

The continuous $\beta$-ensembles could be studied in details. It is well known that the particles of $\beta$-ensembles are asymptotically distributed according to a so-called equilibrium measure. The fluctuations of the empirical measures could be studied in the case where the equilibrium measure has a connected support \cite{johanssonclt, KrSh, borot-guionnet} and in the multi-cut case \cite{Shc12,borot-guionnet2,BGK}.  Local fluctuations were first studied in the case $\beta=2$ where the determinantal structure of the law provides specific tools. They have been first analyzed in the pioneering works of Gaudin, Mehta and Dyson \cite{MR0112895,MR0278668,MR0277221} (see \cite{ME} for a review), where it was shown that the local correlation statistics are governed by the sine kernel for the GUE in the bulk. At the edge, the local fluctuations for the GUE were first identified by Tracy and Widom \cite{TW1} to be asymptotically  given by the Tracy-Widom distribution. Later, bulk universality was proven in \cite{MR1715324} for $V$ polynomial,  \cite{MR1702716} for $V$ analytic and \cite{PS,MR2375744} for $V$ locally $C^3$.
\cite{LL} provides the most general result for bulk universality in the case $\beta=2$. The edge universality for analytic $V$ was proven in \cite{MR1702716}.
In the case $\beta=1,4$, the distribution is Pfaffian and more difficult to analyze. For quadratic $V$, the local statistics were derived in \cite{MR0278668,MR0277221} for the bulk and \cite{TRW96} for the edge.  Universality was proven in  \cite{DeGibulk} in the bulk, and \cite{DeGiedge} at the edge, for monomial potentials $V$ (see \cite{DeGi2} for a review). For more general potentials, universality was shown in \cite{KrSh} ($\beta=1,4$, one-cut case) and \cite{Sh} ($\beta=1,4$, multi-cut case). The local fluctuations of more general $\beta$-ensembles were only derived recently \cite{VV,RRV} for quadratic $V$, based on Dumitriu-Edelman matrix representation. It was shown that correlation functions in the bulk are given by the Sine $\beta$ kernel, and local fluctuations at the edge converge to the Tracy-Widom $\beta$ law.
Universality in  $\beta$-ensembles was first addressed in the bulk \cite{MR3192527,MR2905803} ($\beta> 0$, $V\in C^4$), then at the edge  \cite{MR3253704} ($\beta\geq 1$, $V\in C^4$) and 
\cite{KRV} ($\beta>0$, $V$ convex polynomial). Later, the transportation of measure approach was developed to prove the universality of local fluctuations \cite{Shch} (bulk, $\beta>0$, $V$ analytic, both one-cut and multi-cut cases), \cite{BFG} (both bulk and edge, $\beta>0$, $V \in C^{{31}}$, one-cut case) and \cite{Bek} (edge, $\beta>0$, $V$ analytic, multi-cut case). 

A central point to study mesoscopic or microscopic fluctuations of $\beta$-ensembles is to derive rigidity estimates showing that particles are very close to their deterministic limit. They are usually formulated in terms of concentration of the Stieltjes transform of the empirical particle density at short scales. These results were first established for Wigner matrices in a series of papers \cite{MR2481753, MR2537522,MR2871147,MR2981427,MR3109424}, then extended to other matrix models, i.e. sparse random matrices \cite{MR3098073}, deformed Wigner ensembles \cite{Landon2016, MR3502606}. Beyond matrix models, rigidity estimates have been established for one-cut and multi-cut $\beta$-ensembles \cite{MR3192527, MR3253704,MR2905803, multicutRig}, and two-dimensional Coulomb gas \cite{Leble2015, Roland2015}. Another key tool to study fluctuations is  the loop (or Dyson-Schwinger) equations.
They were introduced to the mathematical community by Johansson \cite{johanssonclt} to derive macroscopic central limit theorems for general $\beta$-ensembles, see also \cite{borot-guionnet, borot-guionnet2, KrSh}.  They were later on used to study fluctuations on microscopic scales \cite{MR3192527, MR3253704} and mesoscopic scales \cite{BL16}.
The central idea is to analyze these equations by linearizing them around the limit, hence obtaining linear equations at the first order, and to show that higher order terms are negligible by concentration of measure arguments. To this end, one needs to solve the linear equations, which amounts to invert some linear operators. This can easily be done when the equilibrium measure has a connected support with density vanishing  like a square root near the edges. In this case, we say the potential and the equilibrium measure are off-critical.  It can be shown (see \cite{kuij}) that most potentials  are off-critical, in the sense that if we multiply the potential by a constant, the equilibrium measure will be off-critical for almost all such constants.  Critical potentials lead to different fluctuations at the edge \cite{Deic}.

It is natural to seek for similar results for the discrete $\beta$-ensembles. As for continuous models, the convergence of the empirical measure of the particles towards an equilibrium measure is easy to prove, and a large deviation principle can be derived \cite{feral}. However, proving global or local fluctuations is much more challenging in the discrete setting than in the continuous one.
When $\beta=2$, the ensembles are called the discrete orthogonal polynomial ensembles, which can be analyzed thanks to their determinantal structure. The central limit theorem for the empirical measure was proven  in \ \cite{BF, MR3298467, BD, BuGo, Borodin2016}. Local fluctuations were studied for a large family of integrable choices of weights in \cite{BOO,  JoArct, MR1969205,  MR1842783, Jo01, Jo02} and it was shown that local fluctuations in the bulk converge to the discrete sine kernel, whereas Tracy-Widom distribution describes the fluctuations at the edge. \cite{MR1952523, baik} provide the most general results for the local fluctuations in the case $\beta=2$. 
However, these results mentioned above depend on the integrability of the system, and do not include the multi-cut case.

In this paper, we study the  local fluctuations of the positions of extreme particles of discrete $\beta$-ensembles (also the positions of extreme holes in the case $\beta=2$ by duality) in off-critical situations (corresponding to the case where the density behaves like a square root near zero or $\theta^{-1}$).  We show that these local fluctuations at the edge  are the same as for continuous $\beta$-ensembles, i.e. are governed by the Tracy-Widom $\beta$ distributions. This reflects that the fluctuations of extreme particles are in a much larger space size than the mesh size of the lattice.  The central point is thus to obtain rigidity estimates of the particle locations on the optimal scale, or at least on a scale comparable to the scale of the fluctuations.  With the rigidity estimates as input, we then can prove universality by a comparison of discrete and continuous local measures, for which the edge universality was proven in \cite{MR3253704}. Typically, the continuous models can be studied as the degenerations of discrete models. However, our proof goes in the opposite direction, i.e. we prove the universality of discrete models by comparing with continuous models. We believe that this approach can be further developed to study other discrete models.

Our approach to rigidity estimates relies on a multiscale analysis combining ideas from \cite{MR3192527, MR3253704,MR2905803, Leble2015, Roland2015}, in particular the use of local measures, which are the measures on a small number of consecutive particles that are obtained by fixing all other particles which act as boundary conditions,  and the Nekrasov's equations from \cite{Borodin2016}. 
Nekrasov's equations  originated in the work of Nekrasov and his collaborators \cite{Nekrasov, Nek_PS,Nek_Pes} and require the specific form of the interaction in $\bP_N$ (a ratio of Gamma functions rather than the Coulomb interaction). 
Nekrasov's equations are a key tool to obtain fine estimates on the convergence to equilibrium in a spirit very close to the loop (or Dyson-Schwinger) equations. 
In  \cite{Borodin2016}, the authors analyzed Nekrasov's equations with similar arguments  as in the continuous case and  derived macroscopic central limit theorems for the discrete $\beta$-ensembles. It was shown in particular that the covariance structure of the central limit theorem is the same as that in the continuous case. Our multiscale iteration approach is as follows. The optimal rigidity estimates at macroscopic scale were proven in \cite{Borodin2016}. With the rigidity estimates on a larger scale, we derive large deviation estimates of the local measures. Those large deviation estimates all together lead to weak rigidity estimates of the original discrete $\beta$-ensemble, but in a finer scale. We then use the Nekrasov's equation to boostrap the rigidity estimates up to optimal errors. By iterating this procedure, after finite steps, we get the rigidity estimates on the optimal scale.

We generalize rigidity and the edge universality results in section \ref{s:multi-cut} to the case  of multi-cut models with fixed filling fractions. We then apply these results to the Krawtchouk ensembles, which correspond to binomial weights in section \ref{s:bw}. We also generalize to cases where particles are not constrained to stay within a bounded set  but are submitted to a uniformly convex potential in section \ref{s:cp}.  Later we discuss the case of uniformly random lozenge tilings of a hexagon with a hole in section \ref{s:lt}. It is similar to multi-cut models but where filling fractions are not fixed. Under the assumption that the filling fractions do not fluctuate too much (that is verified in a work in progress of G. Borot, V. Gorin and A. Guionnet \cite{GGG}), we show that there is universality of the fluctuations at the band edges. { Finally in section \ref{s:JM}, we use the universality result for the discrete $\beta$-ensembles to study the Jack deformation of the Plancherel measure. We prove that the distribution of the lengths of the first few rows in Young diagrams under the Jack defromation of the Plancherel measure, converges to the Tracy-Widom $\beta$ distribution for $\beta\geq 1$, which answers an open problem in \cite{MR3498866}.}

\emph{Conventions.} We use $C$ to represent large universal constants, and $c$ small universal constants, which may be different from line by line. Let $Y\geq 0$. We write that $X=\OO(Y)$ if there exists some universal constant such that $|X|\leq CY$. We write $X=\oo(Y)$, or $X\ll Y$ if the ratio $|X|/Y\rightarrow 0$ as $N$ goes to infinity. We write $X\asymp Y$ if there exist universal constants such that $cY\leq |X|\leq CY$. We write $X\lesssim Y$ if there exist universal constant such that $X\leq CY$. We write $X\gtrsim Y$ if there exist universal constants such that $X\geq cY$. We denote the set $\{1, 2,\cdots, N\}$ by $\qq{1,N}$. We denote the $\bZ_{\geq 0}$ the set of non-negative integers, and $\bZ_{>0}$ the set of positive integers. We say an event $\Omega$ holds with high probability, if there exist universal constant $c$, and $N\geq N_0(c)$ large enough, so that $\bP(\Omega)\geq 1-\exp(-c(\ln N)^2)$.

\noindent \textbf{Acknolwedgements:} { We thank A. Borodin and E. Dimitrov for useful comments on the draft of this paper.} J.H. thanks R. Bauerschmidt and M. Nikula for  explaining their insights of \cite{Roland2015}, and  A. Adhikari, A. Borodin, P. Sosoe and  H-T. Yau for helpful discussions. J.H. would also like to thank {\'E}cole Normale Sup{\'e}rieure de Lyon for hospitality during preparation of this manuscript.

\subsection{Background on Discrete $\beta$-Ensemble}\label{s:notation}
In this section we collect several assumptions of the discrete $\beta$-ensembles from \cite{Borodin2016}. 
For each $N>0$, we take an interval $(a(N),b(N))$ such that $b(N)-a(N)-N\theta\in \mathbb Z_{>0}$.
 \begin{definition} \label{Def_state_space_1cut}The state space $\bW_N^{\theta}$ consists of $N$--tuples $a(N)<\ell_1<\ell_2<\dots<\ell_N<b(N)$:
\begin{enumerate}
 \item $\ell_1-a(N)\in \mathbb Z_{> 0}$, and $b(N)-\ell_N\in\mathbb Z_{> 0}$.
 \item For $i\in\qq{1,N-1}$,  $\ell_{i+1}-\ell_{i}-\theta\in\bZ_{\geq 0}$.
\end{enumerate}
\end{definition}

We need to assume that the weights $w(\ell; N)$ as in \eqref{eqn:distr} depend on $N$ in a regular way.

\begin{assumption}\label{a:VN}
There exists $-\infty<\hat a<\hat b<+\infty$ such that as $N\rightarrow \infty$,
\begin{align}\label{e:locaN}
a(N)=N\hat a+\OO(\ln (N)),\quad b(N)=N\hat b+\OO(\ln(N)).
\end{align}
We require that $w(a(N); N)=w(b(N); N)=0$, and on $[a(N)+1,b(N)-1]$, it has the form
\begin{align*}
w(x; N)=\exp\left(-NV_N\left(\frac{x}{N}\right)\right),
\end{align*}
for a function $V_N$ that is continuous in the intervals $[(a(N)+1)/N,(b(N)-1)/N]$, and such that 
\begin{align*}
V_N(u)=V(u)+\OO\left(\frac{\ln(N)}{N}\right), 
\end{align*}
uniformly over $u\in [(a(N)+1)/N, (b(N)-1)/N]$. The function $V(u)$ is   twice continuously differentiable and the following bound holds for a constant $C>0$,
\begin{align*}
|V'(u)|\leq C(1+|\ln(u-\hat a)|+|\ln(u-\hat b)|).
\end{align*}
\end{assumption}

One important tool for the study  of measures $P_N$
is the Nekrasov's equation of Theorem \ref{t:aRN} for which we will need the following additional assumption.  

\begin{assumption}\label{a:wx}
There exists an open set $\cM$ of the complex plane, which contains the interval $[\hat a, \hat b]$, and  functions $\psi^{\pm}_N(x)$ such that 
\begin{align*}
\frac{w(x;N)}{w(x-1;N)}=\frac{\psi_N^+(x)}{\psi_N^-(x)},
\end{align*}
which satisfy
\begin{align*}
\psi_N^{\pm}(x)=\phi^{\pm}\left(\frac{x}{N}\right)+\OO\left(\frac{1}{N}\right),
\end{align*}
uniformly over $x/N$ in compact subsets of $\cM$. All the aforementioned functions are holomorphic in $\cM$. \end{assumption}

\begin{theorem}\label{t:aRN}
Let $\bP_N$ be a distribution on $N$-tuples $(\ell_1, \ell_2, \cdots, \ell_N)\in \bW_N^\theta$ as in \eqref{eqn:distr},  with
\begin{align*}
\frac{w(x,N)}{w(x-1), N}=\frac{\psi^+_N(x)}{\psi_N^-(x)}.
\end{align*}
We define
\begin{align*}
R_N(\xi)=\psi^-_N(\xi)\bE_{\bP_N}\left[\prod_{i=1}^{N}\left(1-\frac{\theta}{\xi-\ell_i}\right)\right]+\psi^+_N(\xi)\bE_{\bP_N}\left[\prod_{i=1}^{N}\left(1+\frac{\theta}{\xi-\ell_i-1}\right)\right].
\end{align*}
If $\psi_N^{\pm}(\xi)$ are holomorphic in a domain $N\cM=\{Nz: z\in \cM\}\in \bC$, then so is $R_N(\xi)$. 
\end{theorem}

We denote the empirical particle density as $\mu_N=N^{-1}\sum_{i=1}^N \delta_{\ell_i/N}$. Let $\mu(dx)=\rho_V(x)\rd x$ be the constrained equilibrium measure of the discrete $\beta$-ensemble $\bP_N$, which is given by the minimizer of the following functional 
\begin{align}\label{e:defIV}
-\theta \int_{x\neq y}\ln |x-y|\rho(x)\rho(y)\rd x\rd y+\int V(x)\rho(x)\rd x,
\end{align}
over all the densities $\rho(x)$, supported on $[\hat a, \hat b]$, with $0\leq \rho(x)\leq \theta^{-1}$. We define $\hat a_N\deq \min\{a(N)/N,\hat a\}$ and $\hat b_N\deq \max\{b(N)/N, \hat b\}$. Then both the empirical particle density $\mu_N$ and the constrained equilibrium measure $\mu$ are supported on $[\hat a_N, \hat b_N]$. We have the following characterization of the constrained equilibrium measure $\mu$.  
The restriction $0 \leq \rho_V(x)\leq \theta^{-1}$
leads to the subdivision of $[\hat a, \hat b]$ into three types of regions (we
borrow the terminology from \cite{MR1952523}):
\begin{itemize}
\item  Maximal (with respect to inclusion) closed connected intervals where $\rho_V(x) = 0$ are called
voids.
\item Maximal open connected intervals where $0 <\rho_V(x) < \theta^{-1}$ are called bands.
\item Maximal closed connected intervals where $\rho_V(x) = \theta^{-1}$ are called saturated regions.
\end{itemize}

In a related context of random tilings and periodically-weighted dimers the voids and saturated
regions are usually called frozen, while bands are liquid regions.
There exists a constant $f_V$ so that the  effective potential $F_V$ given by 
\begin{align}\label{e:defFV}
F_V (x) \deq -2\theta 
\int_{\hat a}^{\hat b}
\ln |x - y|\rho_V(y)dy + V (x)-f_V,
\end{align}
satisfies:
\begin{enumerate}
\item $F_V (x)  \geq 0$, for all $x$ in voids in $[\hat a, \hat b]$;
\item $F_V (x) \leq  0$, for all $x$ in saturated regions in $[\hat a, \hat b]$;
\item $F_V (x)  = 0$, for all $x$ in bands in $[\hat a, \hat b]$.
\end{enumerate}

We denote the classical particle locations $\gamma_1, \gamma_2, \cdots, \gamma_N$ corresponding to the constrained equilibrium measure $\mu$, as
\begin{align}\label{e:defloc}
\frac{i-1/2}{N}=\int_{\hat a}^{\gamma_i}\rho_V(x)\rd x, \quad i\in \qq{1, N}.
\end{align}
A convenient way to study the particle system $\bP_N$ is through the Stieltjes transform $G_N$ of the empirical particle density $\mu_N$, 
\begin{align}\label{e:defGN}
G_N(z)=\frac{1}{N}\int \frac{ \mu_N(x)}{z-x}=\frac{1}{N}\sum_{i=1}^N\frac{1}{z-\ell_i/N}.
\end{align}
We also define the Stieltjes transform $G_\mu$ of the equilibrium measure $\mu$,
\begin{align}\label{e:defGmu}
G_\mu(z)=\int \frac{\rho_V(x)\rd x}{z-x}.
\end{align}
Notice that $G_N$ and $G_\mu$ are well defined outside $[\hat a_N, \hat b_N]$, and are holomorphic there. 

We next define the following two functions $R_\mu$ and $Q_\mu$, which are important for our asymptotic study,
\begin{align}\begin{split}\label{e:defRQ}
R_\mu(z) \deq \phi^-(z)e^{-\theta G_\mu(z)}+\phi^+(z)e^{\theta G_\mu(z)},\\
Q_\mu(z) \deq \phi^-(z)e^{-\theta G_\mu(z)}-\phi^+(z)e^{\theta G_\mu(z)}.
\end{split}\end{align}
It is proven in \cite{Borodin2016} that under Assumptions \ref{a:VN}  and \ref{a:wx}, $R_\mu(z)$ is analytic on $\cM$. For the function $Q_\mu(z)$, we assume
\begin{assumption}\label{a:Hz}
We assume there exists a function $H(z)$ holomorphic in $\cM$ and numbers $A, B$ such that 
\begin{itemize}
\item { $\hat a < A<B< \hat b$};
\item $Q_\mu(z)=H(z)\sqrt{(z-A)(z-B)}$, where the branch of square root is such that $\sqrt{(z-A)(z-B)}\sim z$ as $z\rightarrow \infty$; 
\item $H(z)\neq 0$ for all $z$ in a neighborhood of  $[\hat a, \hat b]$.
\end{itemize}
\end{assumption}

{
\begin{remark}
Assumption \ref{a:Hz} is slightly different from\cite[Assumption 4]{Borodin2016}, i.e. we do not allow that $\hat a=A$ or $\hat b=B$. If this is the case, the equilibrium measure might not have square root behaviors at band edges. For example, for the Krawtchouk orthogonal polynomial ensemble in section \ref{s:bw}, if $\fm=2$, then $\hat a=A=0$, $\hat b=B=2$ and the equilibrium measure $\mu_\fm(x)={\bf 1}_{[0,2]}(x)/2$, which does not have square root behaviors at the edges.
\end{remark}
}

For any function $F(z)$ defined on the complex plane $\bC$, we denote,
\begin{align*}
F(E+0\ri)\deq\lim_{\eta\rightarrow 0+} F(E+\eta\ri),\quad F(E-0\ri)\deq\lim_{\eta\rightarrow 0+} F(E-\eta\ri).
\end{align*}
The measure $\mu$ can be recovered by its Stieltjes transform. And more precisely, we have
\begin{align}\label{e:Grho}
G_\mu(E+0\ri)=P.V.\int\frac{\rho_V(x)\rd x}{E-x}-\pi \rho_V(E)\ri.
\end{align}
We assume Assumption \ref{a:Hz}. It follows from \eqref{e:defRQ} and \eqref{e:Grho}, the density $\rho_V$ satisfies
\begin{align}\label{e:RQ1}
\frac{R_\mu(E)+Q_\mu(E+0\ri)}{R_{\mu}(E)+Q_\mu(E-0\ri)}=e^{2\pi \theta\ri  \rho_V(E)}.
\end{align}
For $E\in (A, B)$, we have $Q_\mu(E\pm 0\ri)=\pm H(E)\ri\sqrt{(B-E)(E-A)}\neq 0$. Therefore $0<\rho_V(E)<\theta^{-1}$. { If $R_\mu(A), R_\mu(B)\neq 0$, then either $\rho_V(E)$ or $\theta^{-1}-\rho_V(E)$ has square root behaviors at $A, B$. Otherwise, by symmetry, we assume $R_\mu(A)=0$. We also have $Q_\mu(A)=0$, then from \eqref{e:defRQ} $\phi^{\pm}(A)=0$, which is impossible since $\hat a<A<\hat b$.} For $E\not\in [A, B]$, we have $Q_\mu(E\pm 0\ri)=H(E)\sqrt{(E-B)(E-A)}\neq 0$. Therefore $e^{2\pi \theta \ri \rho_V(E)}=1$, so $\rho_V(E)\equiv 0$ or $\rho_V(E)\equiv \theta^{-1}$. Notice that under Assumptions \ref{a:VN} and \ref{a:wx}, we have $
e^{-V'(x)}=\phi^+(x)/\phi^-(x).
$
For $E\notin [A, B]$, we have $Q_\mu(E\pm 0\ri)=H(E)\sqrt{(E-A)(E-A)}\neq 0$. It follows from \eqref{e:defRQ} and \eqref{e:Grho}, the Stieltjes transform $G_\mu$ satisfies
\begin{align}\label{e:RQ2}
\frac{R_\mu(E)+Q_\mu(E+0\ri)}{R_{\mu}(E)-Q_\mu(E-0\ri)}=e^{V'(E)-2\theta \Re[G_\mu(E)]}=e^{F'_V(E)},
\end{align}
where $F'_V(E)$ is defined in \eqref{e:defFV}. As a result, $F'_V(E)$ has square root behaviors at $A, B$. We remark that, the righthand sides of \eqref{e:RQ1} and \eqref{e:RQ2} can be rewritten as the same expression,
\begin{align}\label{e:RQ3}
\frac{R_\mu(E)+Q_\mu(E+0\ri)}{R_{\mu}(E)-Q_\mu(E+0\ri)}=\left\{
\begin{array}{cc}
e^{2\pi \theta \ri \rho_V(E)}, & E\in [A, B],\\
e^{F'_V(E)}, & E\not\in [A,B].
\end{array}\right.
\end{align}
In summary, we have the following four cases. There exists $\varepsilon>0$ so that:
\begin{enumerate}
\item the constrained equilibrium measure is supported on $[A,B]$. There exist analytic functions $s_A$ and $s_B$ such that  $s_A(x)\sqrt{x-A}=2\pi \theta\rho_V(x)$ on $[A, A+\varepsilon]$, $s_A(x)\sqrt{A-x}=-F'_V(x)$ on $[A-\varepsilon, A)$; and  $s_B(x)\sqrt{B-x}=2\pi\theta \rho_V(x)$ on $[B-\varepsilon, B]$, $s_B(x)\sqrt{x-B}=F'_V(x)$ on $x\in (B, B+\varepsilon)$.
\item the constrained equilibrium measure is supported on $[A,\hat b]$, and $\rho(x)\equiv \theta^{-1}$ on $[B, \hat b]$. There exist analytic functions $s_A$ and $s_B$ such that  $s_A(x)\sqrt{x-A}=2\pi \theta\ri\rho_V(x)$ on $[A, A+\varepsilon]$, $s_A(x)\sqrt{A-x}=-F'_V(x)$ on $[A-\varepsilon, A)$; and  $s_B(x)\sqrt{B-x}=2\pi-2\pi\theta \rho_V(x)$ on $[B-\varepsilon, B]$, $s_B(x)\sqrt{x-B}=-F'_V(x)$ on $x\in (B, B+\varepsilon)$.

\item the constrained equilibrium measure is supported on $[\hat a,B]$, and $\rho(x)\equiv \theta^{-1}$ on $[\hat a, A]$. There exist analytic functions $s_A$ and $s_B$ such that  $s_A(x)\sqrt{x-A}=2\pi-2\pi \theta\rho_V(x)$ on $[A, A+\varepsilon]$, $s_A(x)\sqrt{A-x}=F'_V(x)$ on $[A-\varepsilon, A)$; and  $s_B(x)\sqrt{B-x}=2\pi\theta \rho_V(x)$ on $[B-\varepsilon, B]$, $s_B(x)\sqrt{x-B}=F'_V(x)$ on $x\in (B, B+\varepsilon)$.

\item the constrained equilibrium measure is supported on $[\hat a, \hat b]$, $\rho(x)\equiv \theta^{-1}$ on $[\hat a, A]$ and $\rho(x)\equiv \theta^{-1}$ on $[B, \hat b]$. There exist analytic functions $s_A$ and $s_B$ such that  $s_A(x)\sqrt{x-A}=2\pi-2\pi \theta\rho_V(x)$ on $[A, A+\varepsilon]$, $s_A(x)\sqrt{A-x}=F'_V(x)$ on $[A-\varepsilon, A)$; and  $s_B(x)\sqrt{B-x}=2\pi-2\pi\theta \rho_V(x)$ on $[B-\varepsilon, B]$, $s_B(x)\sqrt{x-B}=-F'_V(x)$ on $x\in (B, B+\varepsilon)$.
\end{enumerate}

For the proof of edge universality, we need stronger control on the potential $V_N$ (as in Assumption \ref{a:VN}), i.e. the difference between the derivative of $V_N$ and the derivative of its limit is negligible. 
\begin{assumption}\label{a:moreVN}
For any $u$ in a small neighborhood of $[A, B]$, { $V_N(u)$ is analytic} and the following holds
\begin{align*}
V_N'(u)=V'(u)+\OO\left(N^{ -1/3}\right).
\end{align*}
\end{assumption}

\subsection{Main Results}

The goal of this paper is to prove the edge universality of discrete $\beta$-ensemble. In the rest of the paper, we assume that $\hat a<A$ and $[\hat a, A]$ is a void region, i.e. $\rho_V(x)=0$ on $[\hat a, A]$.  As discussed in the previous section, the density vanishes like a square root at $A$,
\begin{align*} 
\rho_V(u)=(1+\OO(u-A))s_A\sqrt{u-A}/\pi,\quad u\rightarrow A^+.
\end{align*}

The proof of the edge universality consists of two steps. In the first step, we prove the rigidity of the particle configuration, i.e. with high probability, the particle locations are very close  to their classical locations up to optimal scale, which occupies the main part of the paper. With the rigidity of the particle configuration as input, the edge universality of discrete $\beta$-ensembles follows from a direct comparison with continuous $\beta$-ensemble, for which the edge universality was proven in \cite{MR3253704}.

To state the rigidity theorem, we need some more definitions. We fix a small parameter $\fa>0$, and define the spectral domain $\cD_r$  given for $r\geq 0$ by
\begin{align}\begin{split}\label{def:domainD}
\cD_r&\deq \cD_r^{\rm int}\cup \cD^{\ext},\\
\cD_r^{\rm int}&\deq \{E+\ri \eta \in \cM\cap \bC_+: E\in[A, B-r], \eta\sqrt{\kappa_E+\eta}\geq N^{\fa}/N\},\\
\cD^{\ext}&\deq \{E+\ri \eta \in \cM\cap \bC_+: E\leq A, \eta\geq (N^{\fa}/N)^{2/3}\},
\end{split}
\end{align}
where $\kappa_E=\dist(E,\{A,B\})$.

\begin{theorem}\label{t:rigidity}
We assume Assumptions \eqref{a:VN}, \eqref{a:wx}, \eqref{a:Hz} and that $[\hat a, A]$ is a void region, i.e. $\rho_V(x)=0$ on $[\hat a, A]$, then the following holds:
\begin{enumerate}
\item
Fix $s\geq (\ln N)^2$ and small $r>0$. With probability at least $1-e^{-cs}$, we have  uniformly for any $z=E+\ri \eta\in \cD_r$, 
\begin{align}\label{e:bulkrig}
\left|G_N(z)-G_\mu(z)\right|\leq \frac{s}{N\eta}.
\end{align}

\item 
Fix a small $\fc$ such that $0<\fc<\fa/4$. With probability at least $1-\exp(-c(\ln N)^{2})$, we have uniformly for $z=E+\ri (N^{\fa}/N)^{2/3}\in  \cM$ with $E\leq A-N^{2\fc}(N^{\fa}/N)^{2/3}$, 
\begin{align}\label{e:edgerig}
\left|G_N(z)-G_\mu(z)\right|\ll \frac{1}{N\eta}.
\end{align}
\end{enumerate}
A similar statement holds if $[B, \hat b]$ is a void region in the vicinity of $\hat b$.
\end{theorem}

\begin{remark}
We believe that the rigidity estimate \eqref{e:bulkrig} also holds for saturated regions. However, for  technical reasons (see e.g. the proof of Proposition \ref{prop:ldb}), the current method does not work in this case.
\end{remark}

Following a standard application of the Helffer-Sj{\" o}strand
functional calculus along the lines of \cite[Lemma B.1]{MR2639734}, the following result may be deduced from
Theorem \ref{t:rigidity}.
\begin{corollary}
We assume Assumptions \eqref{a:VN}, \eqref{a:wx}, \eqref{a:Hz} and that $[\hat a, A]$ is a void region, i.e. $\rho_V(x)=0$ on $[\hat a, A]$, then the following holds:
For any $\fa>0$ and $r>0$,  with probability at least $1-\exp(-c(\ln N)^{2})$, uniformly for $i<(\mu([A,B])-r)N$, we have
\begin{align*}
\left|\frac{\ell_i}{N}-\gamma_i\right|\leq \frac{N^{\fa}}{N^{2/3}\min\{i, N-i\}^{1/3}},
\end{align*}
where $\gamma_i$ are the classical particle locations of the equilibrium measure $\mu$ as defined in \eqref{e:defloc}. A similar statement holds if $[B, \hat b]$ is a void region.
\end{corollary}

With the rigidity of the particle configuration as input, the edge universality of discrete $\beta$-ensemble follows from a direct comparison with continuous $\beta$-ensembles.

\begin{theorem}\label{t:edgeuniversality}
Fix $\beta\geq 1$, $\theta=\beta/2$,  $0<\fb<1/13$ and $L=N^{\fb}$. We assume Assumptions \ref{a:VN}, \ref{a:wx}, \ref{a:Hz}, \ref{a:moreVN} and that $[\hat a, A]$ is a void region, $\rho_V(u)=(1+\OO(u-A))s_A\sqrt{u-A}/\pi$. Take any fixed $m\geq 1$ and a continuously differentiable compactly supported function $O:\bR^m\rightarrow \bR$. For any index set $\Lambda\in \qq{1, L}$ with $|\Lambda|=m$, we have
\begin{align}\label{e:edgeuniv}
\left|\bE_{\bP_N}\left[O\left(\left(N^{2/3}k^{1/3}s_A^{2/3}(\ell_k/N-\gamma_k)\right)_{k\in \Lambda}\right)\right]-\bE_{G\beta E}\left[O\left(\left(N^{2/3}k^{1/3}(x_k/N-\tilde \gamma_k)\right)_{k\in \Lambda}\right)\right]\right|=\OO(N^{-\chi}),
\end{align}
where $\gamma_k$ are classical eigenvalue locations of $\rho_V$, and $\tilde \gamma_k$ are classical eigenvalue locations of semi-circle distribution. A similar statement holds if $[B, \hat b]$ is a void region.
\end{theorem}

\section{Rigidity of discrete $\beta$-Ensembles}

Our argument for the rigidity of discrete $\beta$-ensemble is based on a multi-scale analysis. In this section we introduce some basic notations, and give an outline of the proof.

We define the spectral domains on the scale $\tilde \eta=M/N$, for any $1\ll M\leq N$:
\begin{align}\begin{split}\label{def:D}
\cD_{M,r}&\deq \cD_{M,r}^{\rm int}\cup \cD_{M}^{\ext},\\
\cD_{M,r}^{\rm int}&\deq \{E+\ri \eta \in \cM\cap \bC_+: E\in[A, B-r], \eta\sqrt{\kappa_E+\eta}\geq M/N\},\\
\cD_M^{\ext}&\deq \{E+\ri \eta \in \cM\cap \bC_+: E\leq A, \eta\geq (M/N)^{2/3}\}.
\end{split}
\end{align}
where $\kappa_E=\dist(E, \{A,B\})$.
The spectral domain $\cD_{M,r}$ contains information of the particle configuration away from the right endpoint $B$. For the multi-scale analysis, we also need certain weak control on the particle configuration around the right endpoint $B$. To do this, we fix $\fa>0$ and  define the following spectral domain on the scale $\tilde \eta=N^{-(1-\fa)/2}$,
\begin{align}\begin{split}\label{def:D_0}
\cD_*&\deq \cD_*^{\rm int}\cup \cD_*^{\ext},\\
\cD_*^{\rm int}&\deq \{E+\ri \eta \in \cM\cap \bC_+: E\in[A, B], \eta\sqrt{\kappa_E+\eta}\geq N^{-(1-\fa)/2}\},\\
\cD_*^{\ext}&\deq \{E+\ri \eta \in \cM\cap \bC_+: E\notin [A, B], \eta\geq N^{-(1-\fa)/3}\}.
\end{split}
\end{align}

The proof of Theorem \ref{t:rigidity} uses the iteration of the following two-step scheme. If uniformly for any $z=E+\ri\eta\in \cD_{M_i,r_0}\cup \cD_*$, we have
\begin{align}\label{e:midstep1}
\bP_N\left(|G_N(z)-G_\mu(z)|\leq \frac{M_i}{N\eta}\right)\geq 1-e^{-cM_i},
\end{align}
then Theorem \ref{t:lowerscale} implies that for $M_{i+1}=M_i^{\fb/2}$, where $1<{\fb}<2$, with probability $1-e^{-cM_i}$, 
\begin{align*}
|G_N(z)-G_\mu(z)|\ll\sqrt{\kappa_E+\eta} 
\end{align*}
uniformly for $z=E+\ri \eta \in \cD_{M_{i+1},r_0+r}\cup \cD_*$, where $\kappa_E=\dist(E,\{A,B\})$.
The hypotheses of Theorem 
\ref{t:bootstrap} are therefore fulfilled and we deduce that
uniformly for any $z=E+\ri\eta\in \cD_{M_{i+1},r_0+r}\cup \cD_*$, and $s\geq (\ln N)^2$, we have
\begin{align}\label{e:midstep2}
\bP_N\left(|G_N(z)-G_\mu(z)|\lesssim \frac{s}{N\eta}\right)\geq 1-e^{-cs}\,.
\end{align}
This gives 
 \eqref{e:midstep1} at a smaller scale. Then we can use \eqref{e:midstep2} as the input, and iterate the above scheme again, the optimal rigidity follows after finite steps.


\subsection{Rigidity of discrete $\beta$-Ensembles}
\label{s:rig1}

For the equilibrium measure of discrete $\beta$-ensemble, there are void regions and saturated regions. For the energy level close to void regions, the empirical particle density is a natural object to study. However, for the energy level close to saturated regions, the \emph{dual} of the empirical particle density is a more natural object to estimate. In the following we define the \emph{dual} equilibrium measure and the \emph{dual} empirical particle density, and some control parameters. 

We recall the constrained equilibrium measure $\mu$, the empirical particle density $\mu_N$ and their Stieltjes transforms $G_\mu$ and $G_N$ as defined in Section \ref{s:notation}. We define the dual equilibrium measure $\mu^{\dual}$, 
\begin{align}\label{e:dualmu}
\mu^{\dual}=\rho_V^{\dual}(x)\rd x, \quad \rho_V^{\dual}(x)={\bf 1}_{[\hat a, \hat b]}\theta^{-1}-\rho_V(x),
\end{align}
and the dual empirical particle density $\mu_N^{\dual}$,
\begin{align}\label{e:dualmuN}
\mu_N^{\dual}=\frac{1}{\theta N}\sum_{i=0}^{N}\sum_{x\in \ell_i+\bZ_{> 0},\atop x<\ell_{i+1}}\delta_{x/N},
\end{align}
where $\ell_0=a(N)$, $\ell_{N+1}=b(N)$, and $\bZ_{> 0}$ is the set of positive integers. When $\theta=1$, this is essentially the empirical particle density of the dual ensemble studied in \cite[Section 3.2]{MR1952523}. We denote the Stieltjes transforms of \eqref{e:dualmu} and \eqref{e:dualmuN} by $G_\mu^{\dual}(z)$ and $G_N^{\dual}(z)$ respectively. We remark that $\mu_N+\mu_N^{\dual}$ is an \emph{almost} $1/N$-discretization of the measure ${\bf 1}_{[a(N)/N, b(N)/N]}\theta^{-1}\rd x$. And it implies the following estimates, which are useful later 
\begin{align}\begin{split}\label{e:approxmeasure}
&\left|\int_{a(N)/N}^{b(N)/N} \frac{1}{\theta(z-x)} \rd x - \int \frac{1}{z-x} \rd(\mu_N +\mu_N^{\dual})\right|=\OO\left(\frac{1}{N\eta}\right),\\
&\left|\int_{a(N)/N}^{b(N)/N} \frac{1}{\theta(z-x)^2} \rd x - \int \frac{1}{(z-x)^2} \rd(\mu_N +\mu_N^{\dual})\right|=\OO\left(\frac{1}{N\eta^2}\right),
\end{split}
\end{align}
where $\eta=\Im[z]$.

For any $z=E+\ri \eta\in \bC_+$, let $\kappa_E=\dist(E, \{A,B\})$, $\hat \kappa_E=\dist(E, \{\hat a, \hat b\})$, and we define the control parameter:
\begin{align}\begin{split}\label{e:defcontrolmu}
\Theta_\mu(z)=\min\left\{\frac{|\Im[G_\mu(z)]|}{N\eta},\frac{ |\Im[G_\mu^{\dual}(z)]|}{N\eta}+\frac{1}{N(\eta+\hat \kappa_E)}\right\}.
\end{split}\end{align}
Thanks to the square root behaviors of the equilibrium density $\rho_V(x)$ at $A, B$, the asymptotic behavior of $\Theta_\mu$ depends on $\eta$, $\kappa_E$ and $\hat \kappa_E$,
\begin{align}\label{e:Thetaest}
\Theta_\mu(z)\asymp\left\{\begin{array}{ll}
\sqrt{\kappa_E+\eta}/N\eta, &\quad E\in [A, B],\\
1/N\sqrt{\kappa_E+\eta}, & \quad E\not\in [A, B] \text{ on voids}, \\
1/N\sqrt{\kappa_E+\eta}+1/N(\hat \kappa_E+\eta), & \quad E\not\in [A, B] \text{ on saturated regions}.
\end{array}\right.
\end{align}
We also define,
\begin{align*}
\Theta_N(z)=\left\{\begin{array}{ll}
|\Im[G_N(z)]|/N\eta, & \text{ if } \Theta_\mu(z)=|\Im[G_\mu(z)]|/N\eta,\\
|\Im[G_N^{\dual}(z)]|/N\eta+1/N(\eta+\hat \kappa_E), & \text{ if }  \Theta_\mu(z)=|\Im[G_\mu^{\dual}(z)]|/N\eta+1/N(\eta+\hat \kappa_E).
\end{array}\right.
\end{align*}

In the following we collect some basic estimates on the Green's function $G_N(z)$. It turns out that in the discrete setting, the Green's function $G_N(z)$ behaves much more regularly, compared with the continuous setting.
\begin{lemma}\label{l:GNprop}
For any particle configuration $(\ell_1, \ell_2, \cdots, \ell_N)\in \bW_N^\theta$, and $z=E+\ri \eta$ with $\eta\geq \ln N/N$, we have
\begin{align}\label{e:GNprop}
|G_N(z)|=\OO(\ln N), \quad |\Im[G_N(z)]|=\OO(1), \quad |\del_z G_N(z)|=\OO(1/\eta),
\end{align}
where the implicit constants are universal. Moreover, we have the following relation between densities and their duals
\begin{align}\label{e:Gdual}
\left|G_N(z)+G_N^{\dual}(z)-G_\mu(z)-G_\mu^{\dual}(z)\right|=\OO\left(\frac{\ln N}{N\eta}\right).
\end{align}
\end{lemma}
\begin{proof}
By the definition of $\bW_N^\theta$, the  particles $\ell_i$ are well separated by distance at least $\theta$. We have the estimate,
\begin{align*}
|G_N(E+\ri \eta)|\leq \frac{1}{N}\sum_{i=1}^N \frac{1}{|E-\ell_i/N+\ri \eta|}
\leq \frac{2\sqrt{2}}{N}\sum_{i=0}^N \frac{1}{i\theta/N+\eta}=\OO(\ln N).
\end{align*}
For the imaginary part of $G_N(z)$, we have
\begin{align*}
|\Im[G_N(E+\ri \eta)]|= \frac{1}{N}\sum_{i=1}^N \frac{\eta}{(E-\ell_i/N)^2+\eta^2 }
\leq \frac{2}{N}\sum_{i=0}^N \frac{\eta}{(i\theta /N)^2+\eta^2}=\OO(1).
\end{align*}
For the derivative of $G_N(z)$, we have
\begin{align*}
|\del_zG_N(E+\ri \eta)|\leq \frac{1}{N}\sum_{i=1}^N \frac{1}{(E-\ell_i/N)^2+\eta^2 }
\leq \frac{|\Im[G_N(E+\ri\eta)]|}{\eta}=\OO(1/\eta).
\end{align*}
Finally for \eqref{e:Gdual}, we notice that $G_{\mu}(z)+G_{\mu}^{\dual}(z)$ is the Stieltjes transform of the measure $\theta^{-1}{\bf 1}_{[\hat a, \hat b]}$, and from \eqref{e:approxmeasure} $G_N(z)+G_{N}^{\dual}(z)$ is approximated by the Stieltjes transform of the measure $\theta^{-1}{\bf 1}_{[a(N)/N, b(N)/N]}$. Therefore, we have
\begin{align*}
\left|G_N(z)+G_N^{\dual}(z)-G_\mu(z)-G_\mu^{\dual}(z)\right|=\OO\left(\frac{1}{N\eta}\right)
+\left|\left(\int_{\hat a}^{\hat b}-\int_{a(N)/N}^{b(N)/N}\right) \frac{\rd x}{\theta (z-x)}\right|=\OO\left(\frac{\ln N}{N\eta}\right),
\end{align*}
where we used Assumption \ref{e:locaN}, $\hat a=a(N)/N+\OO(\ln N/N)$ and $\hat b=b(N)/N+\OO(\ln N/N)$.
\end{proof}

Thanks to \eqref{e:Gdual}, the following estimate holds,
\begin{align}\label{e:Thetadif}
\Theta_N(z)-\Theta_\mu(z)=\pm \frac{\Im[G_N(z)]-\Im[G_\mu(z)]}{N\eta}+\OO\left(\frac{\ln N}{(N\eta)^2}\right).
\end{align}

The following lemma implies that, the rigidity estimates of the Green function $G_N(z)$ are equivalent to 
the rigidity estimates of the quantity $\sum_{i=1}^{N} \ln (1+1/N(z-\ell_i/N))$, which naturally appears in the Nekrasov's equation.
\begin{lemma}\label{l:GNapprox}
For any particle configuration $(\ell_1, \ell_2, \cdots, \ell_N)\in \bW_N^\theta$, $z=E+\ri \eta$ with $\eta\geq \ln N/N$, $\alpha=\pm 1, \pm \ri$ and $0\leq t\leq 1$, we have
\begin{align}\label{e:GNapprox}
\sum_{i=1}^N \ln \left(1+\frac{\alpha t}{N(z-\ell_i/N)} \right)=\alpha t G_N(z)+\OO\left(\Theta_N(z)+\frac{\ln N}{(N\eta)^2}\right).
\end{align}
where the implicit constant is universal.
\end{lemma}
\begin{proof}
By Taylor expansion, it follows
\begin{align}\begin{split}\label{e:taylor}
\sum_{i=1}^N \ln \left(1+\frac{\alpha t}{N(z-\ell_i/N)}\right)
=&\alpha t G_N(z)-\frac{(\alpha t)^2}{2N^2}\sum_{i=1}^N\frac{1}{(z-\ell_i/N)^2}+\OO\left(\frac{1}{N^3}\sum_{i=1}^N \frac{1}{|z-\ell_i/N|^3}\right).
\end{split}\end{align}
The last term on the righthand side of \eqref{e:taylor} is  $\OO((N\eta)^{-2})$. For the second term on the righthand side of \eqref{e:taylor}, we can directly estimate it,
\begin{align}\label{e:voiderror}
\left|\frac{1}{N^2}\sum_{i=1}^N\frac{1}{(z-\ell_i/N)^2}\right|
\lesssim \frac{1}{N^2}\sum_{i=1}^N\frac{1}{|z-\ell_i/N|^2}=\frac{|\Im[G_N(z)]|}{N\eta}.
\end{align}
There is another way to estimate it. Since $\mu_N+\mu^{\dual}_N$ is an \emph{almost} $1/N$-discretization of ${\bf 1}_{[a(N)/N, b(N)/N]}\theta^{-1}\rd x$, we deduce from \eqref{e:approxmeasure} that
\begin{align*}
\int \frac{1}{(z-x)^2}\left(\rd \mu_N(x)+\rd\mu_N^{\dual}(x)-{\bf 1}_{[a(N)/N, b(N)/N]}\theta^{-1}\rd x\right)=\OO\left(\frac{1}{N\eta^2}\right).
\end{align*}
Therefore, we have
\begin{align}\begin{split}\label{e:saturatederror}
\frac{1}{N^2}\sum_{i=1}^N\frac{1}{(z-\ell_i/N)^2}
&=-\frac{1}{N}\int \frac{1}{(z-x)^2}\rd\mu_N^{\dual}(x)+\frac{1}{N}\left(\frac{1}{z-b(N)/N}-\frac{1}{z-a(N)/N}\right)+\OO\left(\frac{1}{(N\eta)^2}\right)\\
&=\OO\left(\frac{|\Im[G_N^{\dual}(z)]|}{N\eta}+\frac{1}{N(\eta+\hat \kappa_E)}+\frac{1}{(N\eta)^2}\right).
\end{split}\end{align}
The statement follows from combining \eqref{e:voiderror} and \eqref{e:saturatederror}.
\end{proof}

We recall that $\mu$ is the limit of the  empirical distribution of  the particles locations under $\bP_N$ and $G_\mu(z)$ its Stieltjes transform. For any $v\in \bC_+$, $0\leq t\leq 1$, $\al=\pm1,\pm i$, and large number $K=K(N,\Im[v])$, we introduce the deformed probability measure
\begin{align}\begin{split}\label{e:defdeformm}
\bP_N^{K,t,v,\al}
=&\frac{Z_N}{Z_N^{K,t,v,\al}} e^{ K\left(\sum_{i=1}^N \ln \left(1+\frac{\al t}{N(v-\ell_i/N)}\right)+\ln \left(1+\frac{\bar \al t}{N(\bar v-\ell_i/N)}\right)\right)}\bP_N\\
=&\frac{Z_N}{Z_N^{K,t,v,\al}} e^{ 2K\Re\left[\sum_{i=1}^N \ln \left(1+\frac{\al t}{N(v-\ell_i/N)}\right)\right]}\bP_N.
\end{split}\end{align}

The strategy to derive the optimal rigidity consists of two steps. The first step is the large deviation estimate at scale $\tilde\eta$, (we start with scale $O(1)$), which gives us certain control on $G_N(z)-G_\mu(z)$ with respect to the measure $\bP_N^{K,t,v,\al}$. However we don't expect it to be optimal. Then we use the loop equation to get improved estimates. By iteration, the optimal rigidity at scale $\tilde \eta$ follows. Then using the optimal rigidity at larger scale, we can repeat the process at a smaller scale. We can reach the optimal scale $\OO(1/N)$, in finite steps.

\def\cD{{\mathcal D}}

In the macroscopic scale, for any deformed measure $\bP_N^{K, t,v,\al}$, with $K\leq N$, notice that \eqref{e:GNprop} implies
\begin{align*}
2K\Re\left[\sum_{i=1}^N \ln \left(1+\frac{\al t}{N(v-\ell_i/N)}\right)\right]=\OO(2K \Re[\alpha t G_N(v)])=\OO(N\ln N).
\end{align*}
Therefore, by a large deviation argument, similar to \cite[Proposition 2.15] {Borodin2016}, see also Proposition \ref{prop:ldb}, we have

\begin{proposition}\label{p:LDP0}
For two probability measures $\nu,\rho$ set
\begin{equation}\label{eq_quadratic_main}
 \cD(\nu,\rho)=\left(-\int_{\mathbb R}\int_{\mathbb R}\ln|x-y| \rd(\nu(x)-\rho(x))\rd(\nu(y)-\rho(y))\right)^{1/2}.
\end{equation}
Fix a parameter $p>2$ and let $\tilde \mu_N$ denote the convolution of the empirical measure
$\mu_N$  with the  uniform measure on the interval $[-N^{-p}/2,N^{-p}/2]$. Let $K\leq N$.
Then for any $\gamma>0$
\begin{equation}\label{b1} \bP_N^{K, t,v,\al} \Bigl( \cD(\tilde\mu_N,\mu)\geq \gamma\Bigr)\leq\exp\bigl( CN\ln(N)
  -\gamma^2 N^2\bigr).
 \end{equation}
 
As a consequence, there exists a constant $C>0$, such that if  $\hat a_N =\min\{\hat a, a(N)/N\}$ and $\hat b_N=\min\{\hat b, b(N)/N\}$, for any $z$ such that $\dist(z, [\hat a_N, \hat b_N])\geq N^{-p+1}$, 
\begin{align}\label{e:initerror}
 \bP_N^{K, t,v,\al} \left( \left|G_N(z)-G_\mu(z)\right|\geq  \frac{C(\ln N \ln \dist(z, [\hat a_N, \hat b_N]))^{1/2}  }{N^{1/2}\dist(z, [\hat a_N, \hat b_N])}\right) \leq e^{- N\ln N}\,.
\end{align}
\end{proposition}
\begin{proof} The first point is  a direct consequence of \cite[Proposition 2.15] {Borodin2016} as the density of $\bP_N^{K, t,v,\al}$ with respect to $\bP_N$ is uniformly bounded by $e^{c N\ln N}$.
For the second point, notice that there is an alternative formula for $\cD(\nu,\rho)$, if $\nu-\rho$ has zero mass, cf.\cite{BeGu},
\begin{equation}
\label{eq_quadratic_alternative}
 \cD(\nu,\rho)=\sqrt{\int_0^\infty \frac{1}{t} \left|\int_{\mathbb R} e^{\i tx}
 (\nu(x)-\rho(x))\rd x\right|^2 dt}.
\end{equation}
Let $d=\dist(z,[\hat a_N, \hat b_N])$. Note that as $\tilde\mu_N$ and $\mu$ are supported in the finite interval  $[\hat a_N, \hat b_N]$,
$$G_N(z)-G_\mu(z)=\int f_z(x) \rd (\mu_N-\mu)(x)$$
with $f_z(x)=(z-x)^{-1}\chi(x)$ where $\chi$ is nonegative, equals to one on $[\hat a_N, \hat b_N]$ and vanishes outside a slightly larger interval $[\hat a_N-d/2,\hat b_N+d/2]$. We further assume that $\|\chi'\|_{\infty}=\OO(1/d)$.
Therefore, by taking the Fourier transform of $f_z$,
\begin{align}\begin{split}\label{bent}
|G_N(z)-G_\mu(z)|&\leq \left|\int f_z(x) \rd (\tilde\mu_N-\mu)(x)\right|+\left|\int f_z(x) \rd (\mu_N-\tilde\mu_N)(x)\right|\\
&\leq \left|\int  \hat f_z(s) \widehat{ (\tilde\mu_N-\mu)}(s) \rd s\right|+N^{-p} \sup_{x\in [\hat a_N, \hat b_N]} |f'_z(x)|\\
&\leq2 \left(\int s|\hat f_z(s)|^2 \rd s\right)^{1/2} \cD(\tilde\mu_N,\mu) + \frac{1}{N^{p}d^2}.
\end{split}
\end{align}
Finally, since $\int f'_z(x) \rd x=0$,
\begin{align}\begin{split}\label{e:logxy}
\int s|\hat f_z(s)|^2 \rd s&= \int s^{-1}|\hat f_z'(s)|^2 \rd s\\
&= -\int \int \log|x-y|  f'_z(x) \rd x  \bar f'_z(y) \rd y\\
&\leq \int\int \left|\log|x-y|\right|  \left(\frac{\chi(x)}{|z-x|^2}+\frac{|\chi'(x)|}{|z-x|}\right) \rd x    \left(\frac{\chi(y)}{|z-y|^2}+\frac{|\chi'(x)|}{|z-x|}\right)\rd y.
\end{split}\end{align}
We can divide the integral domain of \eqref{e:logxy} into two pieces, $|x-y|\leq d$ and $|x-y|\geq d$. It is easy to check that both integrals are bounded by $\OO(\ln d/d^2)$. We can  estimate the right hand side of \eqref{e:logxy} by $\OO(\ln d/d^2)$. Hence, \eqref{bent} and \eqref{b1} give \eqref{e:initerror}. \end{proof}

Once we have the estimate \eqref{e:initerror}, we can use the loop equation to improve the estimate. More generally, we have
the following theorem. We postpone its proof to the end of of this section.
\begin{theorem}\label{t:bootstrap}
Let $v=E+\ri \eta\in \bC_+$, $\kappa_E=\dist(E,\{A,B\})$, and $\hat\kappa_E=\dist(E, \{\hat a, \hat b\})$. We assume Assumptions \ref{a:VN}, \ref{a:wx} and \ref{a:Hz}, and that $N\eta\sqrt{\kappa_E+\eta}\gg 1$. For any $0\leq t\leq 1$, $\alpha=\pm 1, \pm\ri$, and $K\ll (N\eta)^2\sqrt{\kappa_E+\eta}$, $K\leq N$,  we assume that with high probability w.r.t. $\bP_N^{K,t,v,\al}$ 
\begin{align}\label{e:oldbound}
G_N(v)=G_\mu(v)+\OO(\varepsilon),
\end{align}
then we have
\begin{align}\label{e:bootstrap}
\bE_{\bP_N^{K,t,v,\al}}\left[G_N(v)-G_\mu(v)\right]=\OO\left(\frac{\varepsilon^2}{\sqrt{\kappa_E+\eta}}+\varepsilon_0\right),
\end{align}
where,
\begin{align*}
\varepsilon_0=\frac{1}{\sqrt{\kappa_E+\eta}}\left(\Theta_\mu(v)+\frac{K\min\{|\Im[G_\mu(v)]|,|\Im[G_\mu^{\dual}(v)]|\}}{(N\eta)^2}+\frac{\ln N}{(N\eta)^2}+\frac{\ln N}{N}+\frac{\ln NK}{(N\eta)^3}+\frac{K^2}{(N\eta)^4}\right).
\end{align*}
 Let $\tilde \varepsilon_0 =\max \{\veps_0, (\ln N)^{2}K^{-1}\}$. If  there exists some $\fc>0$, such that $\varepsilon \leq N^{-\fc}\sqrt{\kappa_E+\eta}$, then we have
\begin{align}\label{e:LDP1}
\bP_N\left(|G_N(v)-G_\mu(v)|\lesssim s\tilde\varepsilon_0\right)\geq 1-e^{-csK\tilde \veps_0},
\end{align}
for any $s\geq 1$.
\end{theorem}

As an easy application of Theorem \ref{t:bootstrap}, with Proposition \ref{p:LDP0} as input, we have the following optimal rigidity estimate on the scale $\tilde \eta\gg N^{-1/2}$.
\begin{corollary}\label{c:initialstep} We assume Assumptions \ref{a:VN}, \ref{a:wx} and \ref{a:Hz}.
Fix $s\geq (\ln N)^2$. With probability at least $1-e^{-cs}$, it holds uniformly for any $z=E+\ri \eta\in \cD_{*}$ (as defined in \eqref{def:D_0} for a given $\fa>0$),
\begin{align*}
\left|G_N(z)-G_\mu(z)\right|\lesssim \frac{s}{N\eta}.
\end{align*}

\end{corollary}
\begin{proof}
For any $ z=E+\ri \eta\in \cD_{*}$ , with $\kappa_E=\dist(E, \{A,B\})$ and $\hat \kappa_E=\dist(E,\{\hat a, \hat b\})$, we have $\eta\sqrt{\kappa_E+\eta}\geq N^{-(1-\fa)/2}$. By the large deviation estimates  of Proposition \ref{p:LDP0}, the following holds with high probability with respect to the deformed measure $\bP_N^{N\eta, t,v,\alpha}$ as defined in \eqref{e:defdeformm}
\begin{align*}
\left|G_N(z)-G_\mu(z)\right|\lesssim \frac{N^{\fa/4}}{N^{1/2}\eta}\lesssim N^{-\fa/4}\sqrt{\kappa_E+\eta}.
\end{align*}
 By taking $K=N\eta$ and $\varepsilon=N^{-\fa/4}\sqrt{\kappa_E+\eta}$ in Theorem \ref{t:bootstrap}, we have 
 \begin{align*}
  \varepsilon_0\lesssim
 \frac{1}{\sqrt{\kappa_E+\eta}}\left(\Theta_\mu(z)+\frac{\ln N}{(N\eta)^2}+\frac{\ln N}{N}\right).
 \end{align*}
 From \eqref{e:Thetaest} and $\hat \kappa_E+\kappa_E\gtrsim 1$, we see that $\Theta_\mu(z)/\sqrt{\kappa_E+\eta}\lesssim 1/N\eta$. Moreover, from the definition of the domain $\cD_{*}$, $N\eta\sqrt{\kappa_E+\eta}\gg 1$.  We have
 \begin{align*}
 \varepsilon_0\lesssim\frac{1}{N\eta}+\frac{\ln N}{N\sqrt{\kappa_E+\eta}}+\frac{\ln N}{(N\eta)}\lesssim \frac{\ln N}{N\eta},\quad \td \veps_0=\frac{(\ln N)^2}{N\eta}.
 \end{align*}
Therefore, by \eqref{e:LDP1}, we deduce that
\begin{align}\label{e:LDPM3}
\bP_N\left(|G_N(z)-G_\mu(z)|\lesssim \frac{s}{N\eta}\right)\geq 1-e^{-cs},
\end{align}
for any $s\geq (\ln N)^2$. By an union bound and the Lipschitz property of $G_N(z)-G_\mu(z)$ from Lemma \ref{l:GNprop}, we find a constant $c>0$ such that  for any $s\geq (\ln N)^2$,
\begin{align}\label{e:uniformLDPM3}
\bP_N\left(\sup_{z\in \cD_*}\Im [z] |G_N(z)-G_\mu(z)|\lesssim \frac{s}{N}\right)\geq 1-e^{-cs}\,.
\end{align}
\end{proof}

\begin{theorem}\label{t:lowerscale}
Fix small constants $r_0,r>0$, and a parameter $M$ so that $N^{\fa}\leq M\ll N$. We assume the following holds
\begin{align}\label{e:LDPM2}
\bP_N\left(\left|G_N(z)-G_\mu(z)\right|\geq \frac{M}{N\Im[v]}\right)\leq e^{-cM},
\end{align}
uniformly for any $z\in \cD_{M,r_0}\cup \cD_{*}$. Fix $L=M^\fb$, where $1<\fb<2$. For any $z=E+\ri \eta\in \bC_+$, with $\kappa_E=\dist(E, \{A,B\})$, the following holds with probability $1-e^{-cM}$,
\begin{enumerate}
\item if $E-A\gtrsim (L/N)^{2/3}$ and $E\leq B-r_0-r$, then 
\begin{align}\label{e:Gconcentratebulk}
|G_N(z)-G_\mu(z)|\leq C(\ln N)^2\left(\frac{\sqrt{L}}{N\eta}+\frac{M}{L}\sqrt{\kappa_E}\right).
\end{align}
\item If $E\leq A$ or $E-A\lesssim (L/N)^{2/3}$, then 
\begin{align}\label{e:Gconcentrateedge}
|G_N(z)-G_\mu(z)|\leq C(\ln N)^2\left(\frac{\sqrt{L}}{N\eta}+\left(\frac{M}{L}\right)^{1/2}\left(\frac{L}{N}\right)^{1/3}\wedge\frac{M}{L}\left(\frac{L}{N}\right)^{2/3}\frac{1}{\sqrt{\kappa_E+\eta}}+\frac{M}{L}\left(\frac{L}{N}\right)^{1/3}\right).
\end{align}
\end{enumerate}
\end{theorem}
The proof of Theorem \ref{t:lowerscale} is given in Section \ref{s:lm}.

\begin{proof}[Proof of Theorem \ref{t:rigidity}]
Let $M_1\deq N^{1/2+\fa}$, and an arbitrarily small constant $r>0$. By Corollary \ref{c:initialstep}, there exists $c>0$ so that
for any $s\geq (\ln N)^2$,
\begin{align}\label{e:uniformLDPM3}
\bP_N\left(|G_N(z)-G_\mu(z)|\lesssim \frac{s}{N\eta}\right)\geq 1-e^{-cs},
\end{align}
uniformly for any $z\in \cD_{*}$. Since $\cD_{M_1,r}\cup \cD_{*}=\cD_{*}$, \eqref{e:uniformLDPM3} holds on domain $\cD_{M_1,r}\cup \cD_{*}$. In the following we construct a decreasing sequence: $M_1\gg M_2\gg\cdots M_n=N^{\fa}$. Given that \eqref{e:uniformLDPM3} holds on domain $\cD_{M_k, kr}\cup \cD_{*}$, we prove that \eqref{e:uniformLDPM3} holds on the domain $\cD_{M_{k+1}, (k+1)r}\cup \cD_*$. Thus,  $\eqref{e:uniformLDPM3}$ holds on the  domain $\cD_{M_n,nr}=\cD_{nr}$, and \eqref{e:bulkrig} follows, since we can take $r$ arbitrarily small.

We assume that \eqref{e:uniformLDPM3} holds for any $z\in \cD_{M_k,kr}\cup \cD_*$. Especially \eqref{e:LDPM2} holds for any $M$ such that $M_k\leq M\leq N$ in the spectral domain $\cD_{M,kr}\cup \cD_*$. We assume that $M_k> N^{\fa}$, otherwise we have finished the proof of \eqref{e:bulkrig}. We take $L_k=M_k^\fb$, ($1<\fb<2$ will be chosen later) and $z=E+\ri \eta\in \cD_{M_k^{\fb/2},(k+1)r}\setminus\cD_{M_k,kr}$ with $\kappa_E=\dist(E, \{A,B\})$. More precisely, $\cD_{M_k^{\fb/2},(k+1)r}\setminus\cD_{M_k,kr}$ is given by $E\in [A,B-(k+1)r]$, $M_k^{\fb/2}\leq N\eta\sqrt{\kappa_E+\eta}\leq M_k$; or $E\leq A$, and $(M_k^{\fb/2}/N)^{2/3}\leq\eta\leq (M_k/N)^{2/3}$.

Thanks to Theorem \ref{t:lowerscale} with $r_0=kr$, for any $M$ such that $M_k\leq M\leq N$, we have
\begin{align}\label{e:LDP4.5}
\bP_N\left(|G_N(z)-G_{\mu}(z)|\leq \omega(M,z)\right)\geq 1-e^{-cM},
\end{align}
where $\omega(M,z)$ is defined as the right hand side of \eqref{e:Gconcentratebulk} or \eqref{e:Gconcentrateedge}: Setting  $L=M^\fb$, if $E-A\gtrsim (L/N)^{2/3}$ and $E\leq B-(k+1)r$, then
\begin{align*}
\omega(M,z)\deq
C(\ln N)^2\left(\frac{\sqrt{L}}{N\eta}+\frac{M}{L}\sqrt{\kappa_E}\right),
\end{align*}
whereas if $E\leq A$ or $E-A\lesssim (L/N)^{2/3}$, then
\begin{align*}
\omega(M,z)\deq
C(\ln N)^2\left(\frac{\sqrt{L}}{N\eta}+\left(\frac{M}{L}\right)^{1/2}\left(\frac{L}{N}\right)^{1/3}\wedge\frac{M}{L}\left(\frac{L}{N}\right)^{2/3}\frac{1}{\sqrt{\kappa_E+\eta}}+\frac{M}{L}\left(\frac{L}{N}\right)^{1/3}\right).
\end{align*}
In the following we check that for any $z\in \cD_{M_k^{\fb/2},(k+1)r}\setminus\cD_{M_k,kr}$ and $M\geq N^{\fa}$, $N\eta\omega (M,z)$ is much smaller than $M$. For the bulk case, if $E-A\gtrsim (L/N)^{2/3}$ and $E\leq B-(k+1)r$, we have
\begin{align*}
N\eta \omega(M,z)=C(\ln N)^2\left(\sqrt{L}+\frac{M}{L}N\eta\sqrt{\kappa_E}\right)
\leq C(\ln N)^2\left(M^{\fb/2}+\frac{M}{L}M_k\right)\ll M.
\end{align*}
where we used that on $\cD_{M_k^{\fb/2},(k+1)r}\setminus\cD_{M_k,kr}$, $N\eta\sqrt{\kappa_E+\eta}\leq M_k$.
For the edge case, if $E\leq A$ or $E-A\lesssim (L/N)^{2/3}$, we have $\eta\leq (M_k/N)^{2/3}$, and
\begin{align*}\begin{split}
N\eta \omega(M,z)&\leq C(\ln N)^2\left(\sqrt{L}+N\eta\left(\frac{M}{L}\right)^{1/2}\left(\frac{L}{N}\right)^{1/3}\wedge N\sqrt{\eta}\frac{M}{L}\left(\frac{L}{N}\right)^{2/3}+N\eta \frac{M}{L}\left(\frac{L}{N}\right)^{1/3}\right)\\
&\leq C(\ln N)^2\left(M^{\fb/2}+\left(\frac{M_k}{L}\right)^{1/3}M+\left(\frac{M_k}{L}\right)^{2/3}M\right)
\ll M. 
\end{split}
\end{align*}

Fix any $z=E+\ri \eta\in \cD_{M_k^{\fb/2},(k+1)r}\setminus \cD_{M_k,kr}$, we consider the deformed measure 
\begin{align*}
\bP_N^{N\eta,t,z,\al}
=&\frac{Z_N}{Z_N^{N\eta,t,z,\al}} e^{ \tilde H}\bP_N,\quad \tilde H(z)=2N\eta\Re\left[\sum_{i=1}^N \ln \left(1+\frac{\al t}{N(z-\ell_i/N)}\right)\right],
\end{align*}
for some $0\leq t\leq 1$ and $\alpha=\pm1,\pm\ri$. 
By Lemma \ref{l:GNapprox}, we have  
\begin{align}\label{e:defX}
X(z)\deq \left|\tilde H(z) - 2t N\eta \Re\left[\alpha G_\mu(z)\right]\right|
= 2N\eta\left|G_N(z)-G_\mu(z)\right|+\OO\left(|\Im[G_N(z)]|\right).
\end{align}
Thanks to \eqref{e:LDP4.5}, for any $M_k\leq M\leq N$, we have $|G_N(z)-G_\mu(z)|\leq\omega(M,z)$, with probability at least $1-e^{-cM}$ with respect to $\bP_N$.
Therefore, we have for any $z\in\mathbb C^+$
\begin{align}
\begin{split}
&\phantom{{}={}}\bP_N^{N\eta,t,z,\al}\left( |G_N(z)-G_{\mu}(z)|\geq \omega(M_k,z) \right)
=\frac{\int {\bf{1}}(|G_N(z)-G_{\mu}(z)|\geq \omega(M_k,z)) e^{\tilde H(z)}\rd \bP_N}{\int e^{\tilde H(z)}\rd \bP_N}\\
&
\leq\frac{\int {\bf{1}}(|G_N(z)-G_{\mu}(z)|\geq \omega(M_k,z)) e^{\tilde H(z)}\rd \bP_N}{\int {\bf 1}(X(z)\leq N\eta\psi(M_k)) e^{\tilde H(z)}\rd \bP_N}\\
&\leq\frac{\int {\bf{1}}(|G_N(z)-G_{\mu}(z)|\geq \omega(M_k,z)) e^{\tilde H(z)-2t N \eta \Re[\alpha G_\mu(z)]}\rd \bP_N}{\int {\bf 1} (X(z)\leq N\eta\omega(M_k,z)) e^{\tilde H(z)-2t N \eta \Re[\alpha G_\mu(z)]}\rd \bP_N}\\
&\leq e^{N\eta\omega(M_k,z)}
\int {\bf{1}}(|G_N(z)-G_{\mu}(z)|\geq \omega(M_k,z)) e^{X(z)}\rd \bP_N.
\label{e:tg}
\end{split}
\end{align}
We take $m=\lfloor\log_2(N/M_k)\rfloor$ and decompose the integral region as 
$\{|G_N(z)-G_\mu(z)|\geq N\eta\omega(M_k,z)\}\subset \cup_{i=0}^{m} A_i$, where
\begin{align*}
A_i&=\{|G_N(z)-G_\mu(z)|\in [N\eta\omega( 2^iM_k,z), N\eta\omega(2^{i+1}M_k,z)]\},\quad 0\leq i\leq m-1,\\
A_m&=  \{|G_N(z)-G_\mu(z)|\geq N\eta\omega( 2^mM_k,z)\},
\end{align*}
and the set above is empty if $N\eta\omega( 2^iM_k,z)>N\eta\omega(2^{i+1}M_k,z)$. 
By \eqref{e:GNprop}, $|G_N(z)-G_\mu(z)|\lesssim \ln N$, and combining with \eqref{e:defX} we have that $X(z)\lesssim N\eta\ln N$ on $A_m$. Similarly for $0\leq i\leq m-1$, on $A_i$ we have $X(z)\lesssim N\eta\omega(2^{i+1}M_k,z)$. On the other hand, by \eqref{e:LDP4.5}, we have $\bP_N(A_m)\leq e^{-cN}$ and for $0\leq i\leq m-1$, $\bP_N(A_i)\leq e^{-c2^iM_k}$. Hence, 
We then have 
\begin{align*}
\int {\bf 1} (|G_N(z)-G_\mu(z)|\geq N\eta\omega(M_k,z)) e^{X(z)}\rd \bP_N
&\leq \sum_{i=0}^m \int_{A_i} e^{X(z)}\rd \bP_N\\
&\leq \sum_{i=0}^{m-1} e^{  CN\eta\omega(2^{i+1}M_k,z)} \bP_N(A_i)+e^{CN\eta\ln N}\bP_N(A_m)\\
&\leq\sum_{i=0}^m e^{  CN\eta\omega(2^{i+1}M_k,z)} e^{-c 2^i M_k}+e^{CN\eta\ln N}e^{-cN}\leq e^{-cM_k/2}
\end{align*}
where we used the fact that $\omega(2 M,z)\lesssim \omega(M,z)$ as well as $N\eta \omega(M,z)\ll M$. 
Hence, \eqref{e:tg} yields
$$\bP_N^{N\eta,t,z,\al}\left( |G_N(z)-G_{\mu}(z)|\geq \omega(M_k,z) \right)\leq e^{-cM_k/2+ N\eta\omega(M_k,z)}\leq e^{-cM_k/3}\,.$$
It follows that, with respect to the new measure $\bP_N^{N\eta, t,z,\alpha}$, the following holds 
\begin{align}\label{e:LDPM4}
|G_N(z)-G_{\mu}(z)|\leq\omega(M_k,z)
\end{align}
with probability $1-e^{-cM_k/3}$, provided $z\in \cD_{M_k^{\fb/2},(k+1)r}\setminus\cD_{M_k,kr}$.

We take $\fb=3/2$, and $M_{k+1}\deq \max\{N^{\fa}, M_k^{\fb/2}N^{\fa/5}\}$. Notice that $\cD_{M_{k+1},(k+1)r}\subset \cD_{M_k^{\fb/2}, (k+1)r}$. Then for any $z=E+\ri \eta\in \cD_{M_{k+1},(k+1)r}\setminus\cD_{M_{k},kr}$ with $\kappa_E=\dist(E, \{A,B\})$, \eqref{e:LDPM4} holds. We have for  the bulk case,
\begin{align}\label{e:LDPM4bulk}
\omega(M_k,z)=C(\ln N)^2\left(\frac{M_k^{\fb/2}}{N\eta}+M_{k}^{-(\fb-1)}\sqrt{\kappa_E}\right)\leq N^{-\fa/6}\sqrt{\kappa_E+\eta},
\end{align}
where we used $M_k\geq N^{\fa}$ and  $N\eta\sqrt{\kappa_E+\eta}\geq M_{k+1}$. For the edge case, since $N\eta\sqrt{\kappa_E+\eta}\geq M_{k+1}$, $\kappa_E+\eta\gtrsim (M_{k+1}/N)^{2/3}$ and $\fb=3/2$, we get
\begin{align}
\begin{split}
\omega(M_k,z)&= C(\ln N)^2 \left(\frac{M_k^{\fb/2}}{N\eta}+\frac{M_{k}^{1/2-\fb/6}}{N^{1/3}}\wedge\frac{M_{k}^{1-\fb/3}}{N^{2/3}\sqrt{\kappa_E+\eta}}+\frac{M_k^{1-2\fb/3}}{N^{1/3}}\right)\\
&\leq C(\ln N)^2 \left(\frac{M_k^{\fb/2}}{N\eta}+\frac{M_{k}^{1/2-\fb/6}}{N^{1/3}}\right)\leq 
N^{-\fa/16} \sqrt{\kappa_E+\eta}.
\label{e:LDPM4edge}
\end{split}
\end{align} 
It follows from \eqref{e:LDPM4}, \eqref{e:LDPM4bulk} and \eqref{e:LDPM4edge}, that for any $z\in \cD_{M_{k+1},(k+1)r}\setminus \cD_{M_k,kr}$, the assumptions in Theorem \ref{t:bootstrap} are satisfied. By the same argument as in Corollary \ref{c:initialstep}, we deduce that for any $s\geq (\ln N)^2$,
\begin{align*}
\bP_N\left(|G_N(z)-G_\mu(z)|\lesssim \frac{s}{N\eta}\right)\geq 1-e^{-cs},
\end{align*}
uniformly for any $z\in \cD_{M_{k+1},(k+1)r}$. From our choices of $M_j$, we have
\begin{align*}
M_j=\max\{N^{4\fa/5+(3/4)^{j-1}(1/2+\fa/5)} , N^{\fa}\}, \quad j=1,2,3,\cdots.
\end{align*}
And $M_n=N^{\fa}$ where $n=1+\lceil \log_{3/4}(\fa/5)\rceil$.
It follows taht,  by repeating the above process $n$ times, we have for any $s\geq (\ln N)^2$,
\begin{align*}
\bP_N\left(|G_N(z)-G_\mu(z)|\lesssim \frac{s}{N\eta}\right)\geq 1-e^{-cs},
\end{align*}
uniformly for any $z\in \cD_{M_n,nr}=\cD_{nr}$. Since we can take $r>0$ arbitrarily small, this finishes the proof of \eqref{e:bulkrig}.

In the following we prove \eqref{e:edgerig}. Fix a small constant $\fc>0$ (which will be chosen later), $\eta=(N^{\fa}/N)^{2/3}$, $E\leq A-N^{2\fc}(N^{\fa}/N)^{2/3}$ and $z=E+\ri \eta$. We set $\fb=3/2$. For any $M$ such that $N^{\fa}\leq M\leq N$, let $L=M^{\fb}$. Since $z\in \cD_r$, \eqref{e:bulkrig} holds. Thanks to Theorem \ref{t:lowerscale},  we have
$$
\bP_N\left(|G_N(z)-G_{\mu}(z)|\leq \omega(M,z)\right)\geq 1-e^{-cM}.
$$
Moreover, notice that for $N^\fa\leq M\leq N$
$$
N^{1+\fc}\eta \omega(M,z)
=N^{\fc}M^{3/4}+N^{\fc+2\fa/3}M^{1/4}\wedge N^{\fa/3}M^{1/2}+N^{\fc+2\fa/3}
\ll M. 
$$
provided $\fc< \fa/4$. By the same argument as for \eqref{e:LDPM4}, we have
\begin{align*}
|G_N(z)-G_{\mu}(z)|\lesssim\omega(N^{\fa},z)\leq N^{-\fc}\sqrt{\kappa_E+\eta},
\end{align*}
holds with high probability, with respect to the deformed measure, $\bP_N^{N^{1+\fc}\eta, v, t,\alpha}$.

We take $K=N^{1+\fc}\eta$ in Theorem \ref{t:bootstrap}. Since $\kappa_E\geq N^{2\fc} \eta$, we have 
$$
\veps_0\lesssim \frac{N^{\fc}}{N\kappa_E}+\frac{N^{2\fc}}{(N\eta)^2\sqrt{\kappa_E+\eta}}\lesssim \frac{1}{N^{1+\fc}\eta},\quad \td \veps_0=\frac{(\ln N)^2}{N^{1+\fc}\eta}
$$
and by \eqref{e:LDP1},
\begin{align}\label{e:LDPM3}
\bP_N\left(|G_N(z)-G_\mu(z)|\lesssim \frac{s}{N^{1+\fc}\eta}\right)\geq 1-e^{-cs},
\end{align}
for any $s\geq (\ln N)^2$. By an union bound and the Lipschitz property of $G_N(z)-G_\mu(z)$ from Lemma \ref{l:GNprop}, we conclude that
\begin{align}\label{e:uniformLDPM4}
\bP_N\left(\sup_{z=E+\ri (N^{\fa}/N)^{2/3}\atop 
E\leq A-N^{2\fc}(N^{\fa}/N)^{2/3}}
|G_N(z)-G_\mu(z)|\leq \frac{1}{N^{1+\fc} (N^{\fa}/N)^{2/3} }\right)\geq 1-\exp(-c(\ln N)^2),
\end{align}
provided  $0<\fc<\fa/4$.

\end{proof}

\subsection{Bootstrap and Proof of Theorem \ref{t:bootstrap}}
In this section, we give the proof of Theorem \ref{t:bootstrap}. The main ingredient is  Nekrasov's equation, which plays the role of Dyson-Schwinger (or loop) equation in continuous $\beta$-ensembles. Heuristically, if the weak estimate $G_N(v)=G_\mu(v)+\OO(\varepsilon)$ holds with high probability w.r.t. the deformed measure $\bP_N^{K,t,v,\alpha}$, then  a careful analysis of Nekrasov's equation implies that
\begin{align*}
\bE_{\bP_N^{K,t,v,\alpha}}[G_N(v)-G_\mu(v)]\lesssim \OO(|G_N(v)-G_\mu(v)|^2)=\OO(\varepsilon^2).
\end{align*} 
The error is improved to $\OO(\varepsilon^2)$. By integrating over $t$, we deduce an upper bound on the Laplace transform of $G_N(v)-G_\mu(v)$. As a result, it follows that $G_N(v)=G_\mu(v)+\OO(\varepsilon^2)$ holds with high probability w.r.t. the deformed measure. \eqref{e:LDP1} is proven by iterating this procedure.  

\begin{proof}[Proof of Theorem \ref{t:bootstrap}]
The perturbed measure (defined in \eqref{e:defdeformm}) is given by
\begin{align*}
\bP_N^{K,t,v,\al}
=&\frac{Z_N}{Z_N^{K,t,v,\al}} \prod_{i=1}^N  \left(1+\frac{\al t}{N(v-\ell_i/N)}\right)^K\left(1+\frac{\bar \al t}{N(\bar v-\ell_i/N)}\right)^K \bP_N.
\end{align*}
For simplicity of notations, in the following proof, we write $\bE_{\bP_N^{K,t,v,\alpha}}$ as $\bE$ and $\alpha t$ as $t$ (which is therefore complex).
We construct two analytic functions $\phi_N^+(x)$ and $\phi_N^-(x)$ such that 
\begin{align*}
\frac{\phi_N^+(x)}{\phi_N^-(x)}=
\frac{w(x;N)}{w(x-1;N)}\frac{\left(1+\frac{ t}{N(v-x/N)}\right)^K\left(1+\frac{\bar  t}{N(\bar v-x/N)}\right)^K}{\left(1+\frac{ t}{N(v-(x-1)/N)}\right)^K\left(1+\frac{\bar  t}{N(\bar v-(x-1)/N)}\right)^K}.
\end{align*}
In fact, by our Assumption \ref{a:wx}, we can take
\begin{align}\begin{split}\label{e:defPhi+-}
\phi_N^+(x)=&\psi_N^+(x)
\left(v-\frac{x}{N}+\frac{t}{N}\right)^K\left(v-\frac{x-1}{N}\right)^K\left(\bar v-\frac{x}{N}+\frac{\bar t}{N}\right)^K\left(\bar v-\frac{x-1}{N}\right)^K,\\
\phi_N^-(x)=&\psi_N^-(x)
\left(v+\frac{t}{N}-\frac{x-1}{N}\right)^K\left(v-\frac{x}{N}\right)^K\left(\bar v+\frac{\bar t}{N}-\frac{x-1}{N}\right)^K\left(\bar v-\frac{x}{N}\right)^K.
\end{split}\end{align}
By Theorem \ref{t:aRN}, $R_N(zN)$ defined by
\begin{align}\begin{split}\label{e:defRNPhi}
&R_N(zN)
\deq
\phi_N^-(zN)\bE\left[\prod_{i=1}^N \left(1-\frac{\theta}{N(z-\ell_i/N)}\right)\right]
+\phi_N^+(zN)\bE\left[\prod_{i=1}^N \left(1+\frac{\theta}{N(z-(\ell_i+1)/N)}\right)\right].
\end{split}\end{align}
is analytic in $\mathcal M$.
We rearrange the above expression of $R_N(zN)$, 
\begin{align}\begin{split}\label{e:perturbedRN}
&\quad\frac{R_N(zN)}{\left(v-z+t/N\right)^K\left(v+1/N-z\right)^K\left(\bar v-z+\bar t/N\right)^K\left(\bar v+1/N-z\right)^K}\\
&=\psi_N^-(zN)
\bE\left[\prod_{i=1}^N \left(1-\frac{\theta}{N(z-\ell_i/N)}\right)-e^{-\theta G_\mu(z)}\right]
+\psi_N^+(zN)
\bE\left[\prod_{i=1}^N \left(1+\frac{\theta}{N(z-(\ell_i+1)/N)}\right)-e^{\theta G_\mu(z)}\right]
\\&+\left(\psi_N^-(zN)e^{-\theta G_\mu(z)}+\psi_N^+(zN)e^{\theta G_{\mu}(z)}\right)+\cE(z)\psi_N^-(zN)\bE\left[\prod_{i=1}^N \left(1-\frac{\theta}{N(z-\ell_i/N)}\right)\right],
\end{split}\end{align}
where
\begin{align}\begin{split}\label{e:defE}
\cE(z)\deq&\frac{\left(v+\frac{t}{N}+\frac{1}{N}-z\right)^K\left(v-z\right)^K\left(\bar v+\frac{\bar t}{N}+\frac{1}{N}-z\right)^K\left(\bar v-z\right)^K}{\left(v-z+\frac{t}{N}\right)^K\left(v+\frac{1}{N}-z\right)^K\left(\bar v-z+\frac{\bar t}{N}\right)^K\left(\bar v+\frac{1}{N}-z\right)^K}-1\\
=&\left(1-\frac{t}{N^2}\frac{1}{(v-z+t/N)(v-z+1/N)}\right)^K\left(1-\frac{\bar t}{N^2}\frac{1}{(\bar v-z+\bar t/N)(\bar v-z+1/N)}\right)^K-1\\
=&-\frac{tK}{N^2}\frac{1}{(v-z)^2}-\frac{\bar t K}{N^2}\frac{1}{(\bar v -z)^2}+\OO\left(\frac{K}{(N\eta)^3}+\frac{K^2}{(N\eta)^4}\right),
\end{split}\end{align}
provided $\eta\gg 1/N$, $|z-v|\gtrsim \eta$ and $K\ll (N\eta)^2$.

We take two set of oriented contours: $\cC_1$ consists of two clockwise oriented circles, one is centered at $v$ with radius $\eta/2$, and the other is centered at $\bar v$ with radius $\eta/2$.  $\cC_2\subset \cM$ is a contour which encloses a small neighborhood of $[\hat a, \hat b]$ (especially, for any point on $z\in \cC_2$,  $\dist(z, [\hat a_N, \hat b_N])\gtrsim 1$ and $H(z)\neq 0$) and has counterclockwise orientation. Notice that $\eta\gg 1/N$, so  $v, v+t/N, v+1/N$ are all inside the contour $\cC_1$. By dividing \eqref{e:perturbedRN} by $2\pi \ri (z-v) H(z)$ (as defined in Assumption \ref{a:Hz}), and integrating along the union of contours $\cC_1\cup \cC_2$, we get
\begin{align*}
0=I_1+I_2+I_3+I_4,
\end{align*}
where we used Theorem \ref{t:aRN} that the lefthand side of \eqref{e:perturbedRN} is analytic outside the contour $\cC_1$, and 
\begin{align}\begin{split}\label{e:defIs}
I_1&=\frac{1}{2\pi i}\int_{\cC_1}\Bigl\{\psi_N^-(zN)
\bE\left[\prod_{i=1}^N \left(1-\frac{\theta}{N(z-\ell_i/N)}\right)-e^{-\theta G_\mu(z)}\right]\\
&+\psi_N^+(zN)
\bE\left[\prod_{i=1}^N \left(1+\frac{\theta}{N(z-(\ell_i+1)/N)}\right)-e^{\theta G_\mu(z)}\right]\Bigr\}\frac{\rd z}{(z-v)H(z)},\\
I_2&=\frac{1}{2\pi i}\int_{\cC_2}\Bigl\{\psi_N^-(zN)
\bE\left[\prod_{i=1}^N \left(1-\frac{\theta}{N(z-\ell_i/N)}\right)-e^{-\theta G_\mu(z)}\right]\\
&+\psi_N^+(zN)
\bE\left[\prod_{i=1}^N \left(1+\frac{\theta}{N(z-(\ell_i+1)/N)}\right)-e^{\theta G_\mu(z)}\right]\Bigr\}\frac{\rd z}{(z-v)H(z)},\\
I_3&=\frac{1}{2\pi i}\int_{\cC_1\cup \cC_2}\left(\psi_N^-(zN)e^{-\theta G_\mu(z)}+\psi_N^{+}(zN)e^{\theta G_\mu(z)}\right)\frac{\rd z}{(z-v)H(z)}\\
I_4&=\frac{1}{2\pi i}\int_{\cC_1\cup \cC_2}\cE(z)\psi_N^-(zN)\bE\left[\prod_{i=1}^N \left(1-\frac{\theta}{N(z-\ell_i/N)}\right)\right]\frac{\rd z}{(z-v)H(z)}.
\end{split}\end{align}

From \eqref{e:GNprop}, we have $|\del_z G_N(z)|\leq 1/\eta$. As a consequence, if $|z-v|\leq \eta/2$, $|G_N(z)-G_N(v)|\leq 1/2$ and if $|z-\bar v|\leq \eta/2$, $|G_N(z)-G_N(\bar v)|\leq 1/2$. By our assumption, with high probability, $G_N(v)=G_\mu(v)+\OO(\epsilon)=\OO(1)$, and $G_N(\bar v)=\overline{G_N(v)}=\OO(1)$. 
It follows that, uniformly for $z$ on or inside $\cC_1$, with high probability, $G_N(z)=\OO(1)$. Moreover, by \eqref{e:GNapprox} and \eqref{e:GNprop}, $\prod_{i=1}^N \left(1-\frac{t}{N(z-\ell_i/N)}\right)$ is uniformly bounded by $C\ln N$ for all particle configuration and $\eta N\gg 1$. Hence, we can neglect in the expectation the set where $G_N(z)$ is eventually big (but smaller than $\ln N$). 
As a consequence,  we find that
\begin{align}\begin{split}{\label{e:Ebound}}
&\bE\left[\prod_{i=1}^N \left(1-\frac{\theta}{N(z-\ell_i/N)}\right)\right]
=\bE\left[e^{-\theta G_N(z)+\OO(1)}\right] 
=\OO(1),\\
&\bE\left[\prod_{i=1}^N \left(1+\frac{\theta}{N(z-(\ell_i+1)/N)}\right)\right]
=\bE\left[e^{\theta G_N(z-1/N)+\OO(1)}\right]
=\OO(1).
\end{split}\end{align}

In the following, we analyze the contour integrals for $I_1, I_2, I_3, I_4$. We show that  the leading term of $I_1$ gives the expectation of $G_N(v)-G_\mu(v)$ and prove that $I_2, I_3, I_4$ are all very small, which implies the upper bound \eqref{e:bootstrap}.

For $I_1$, the integrand has a single pole at $z=v$ inside the contour $\cC_1$. Therefore, $I_1$ equals
\begin{align}\label{e:I1simplifys}
\frac{\psi_N^-(vN)}{H(v)}
\bE\left[\prod_{i=1}^N \left(1-\frac{\theta}{N(v-\ell_i/N)}\right)-e^{-\theta G_\mu(v)}\right]
+\frac{\psi_N^+(vN)}{H(v)}
\bE\left[\prod_{i=1}^N \left(1-\frac{\theta}{N(v-(\ell_i+1)/N)}\right)-e^{\theta G_\mu(v)}\right]\,.
\end{align}
We notice that from \eqref{e:Ebound} and Assumption \ref{a:Hz}, the expectation terms and  $1/H(v)$ are all bounded uniformly (independent of $N$). Thus, by Assumption \ref{a:wx}, we can replace $\psi_N^{\pm}(vN)$ by $\phi^{\pm}(v)$, which gives an error $\OO(1/N)$. The remaining part of the first term in \eqref{e:I1simplifys} simplifies as follows:
\begin{align}\begin{split}\label{e:I1first}
&\frac{\phi^-(v)}{H(v)}
\bE\left[\prod_{i=1}^N \left(1-\frac{\theta}{N(v-\ell_i/N)}\right)-e^{-\theta G_\mu(v)}\right]\\
=&\frac{\phi^-(v)e^{-\theta G_\mu(v)}}{H(v)}
\bE\left[e^{-\theta(G_N(v)-G_\mu(v))+\OO\left(\Theta_N(v)+\frac{\ln N}{(N\eta)^2}\right)}-1\right]\\
=&-\frac{\phi^-(v)e^{-\theta G_\mu(v)}}{H(v)}
\bE\left[\theta(G_N(v)-G_\mu(v))\right]+\OO\left(\bE\left[|G_N(v)-G_\mu(v)|^2+\Theta_N(v)\right]+\frac{\ln N}{(N\eta)^2}\right)\\
=&-\frac{\phi^-(v)e^{-\theta G_\mu(v)}}{H(v)}
\bE\left[\theta(G_N(v)-G_\mu(v))\right]+\OO\left(\varepsilon^2+\frac{\left|\bE\left[\Im[G_N(v)-G_\mu(v)]\right]\right|}{N\eta}+\Theta_\mu(v)+\frac{\ln N}{(N\eta)^2}\right),
\end{split}\end{align}
where we used Lemma \ref{l:GNapprox} in the second line, and \eqref{e:Thetadif} in the last line.
We have a similar estimate for the second term in \eqref{e:I1simplifys}:
\begin{align}\begin{split}\label{e:l1second}
&\frac{\phi^+(v)}{H(v)}
\bE\left[\prod_{i=1}^N \left(1+\frac{\theta}{N(v-(\ell_i+1)/N)}\right)-e^{\theta G_\mu(v)}\right]\\
=&\frac{\phi^+(v)e^{\theta G_\mu(v)}}{H(v)}
\bE\left[\theta(G_N(v)-G_\mu(v))\right]+\OO\left(\varepsilon^2+\frac{\left|\bE\left[\Im[G_N(v)-G_\mu(v)]\right]\right|}{N\eta}+\Theta_\mu(v)+\frac{\ln N}{(N\eta)^2}\right).
\end{split}\end{align}
Putting \eqref{e:I1first} and \eqref{e:l1second} together,and recalling Assumption \ref{a:Hz},  we get the following estimate of $I_1$:
\begin{align}\begin{split}{\label{e:I1}}
I_1
&=\theta \sqrt{(v-A)(v-B)}
\bE\left[G_N(v)-G_\mu(v)\right]+\OO\left(\varepsilon^2+\frac{\left|\bE\left[\Im[G_N(v)-G_\mu(v)]\right]\right|}{N\eta}+\Theta_\mu(v)+\frac{1}{N}+\frac{\ln N}{(N\eta)^2}\right).
\end{split}\end{align}

Next we estimate $I_2$. For $z\in \cC_2$, we have $\dist(z, [\hat a_N, \hat b_N])\gtrsim 1$. So from Theorem \ref{p:LDP0},  $|G_N(z)-G_\mu(z)|\lesssim (\ln N/N)^{1/2}$ holds with high probability. Taylor expanding the exponential as $e^{x}=1+x+\OO(x^2)$, we deduce
\begin{align*}
\psi_N^-(zN)
\bE\left[\prod_{i=1}^N \left(1-\frac{\theta}{N(z-\ell_i/N)}\right)-e^{-\theta G_\mu(z)}\right]
&=\phi^-(z)
\bE\left[e^{-\theta G_N(z)+\OO(1/N)}-e^{-\theta G_\mu(z)}\right]+\OO\left(\frac{1}{N}\right)\\
&=-\theta \phi^-(z)e^{-\theta G_\mu(z)}\left(G_N(z)-G_\mu(z)\right)+\OO\left(\frac{\ln N}{N}\right).
\end{align*}
We have a similar estimate for the second term in $I_2$. Combining them we obtain
\begin{align}\begin{split}\label{e:I2}
I_2&=\frac{1}{z\pi \ri}\int_{\cC_2}\frac{\theta(G_N(z)-G_\mu(z))\left(\phi^+(z)e^{\theta G_\mu(z)}-\phi^{-}(z)e^{-\theta G_\mu(z)}\right)}{(z-v)H(z)}\rd z+\OO\left(\frac{\ln N}{N}\right)\\
&=\frac{\theta}{z\pi \ri}\int_{\cC_2}\frac{(G_N(z)-G_\mu(z))\sqrt{(z-A)(z-B)}}{z-v}\rd z+\OO\left(\frac{\ln N}{N}\right)=\OO\left(\frac{\ln N}{N}\right).
\end{split}\end{align}
where the integral vanishes, since the integrand is analytic outside $\cC_2$, and behaves like $1/z^2$ as $z\rightarrow \infty$. 

For $I_3$, since $e^{\pm \theta G_\mu(z)}=\OO(1)$, we can replace $\psi_N^{\pm}(zN)$ by $\phi^{\pm}(z)$, which gives an error $\OO(1/N)$. We can rewrite the integrand as
$$
\psi_N^-(zN)e^{-\theta G_\mu(z)}+\psi_N^{+}(zN)e^{\theta G_\mu(z)}
=\phi^-(z)e^{-\theta G_\mu(z)}+\phi^{+}(z)e^{\theta G_\mu(z)}+\OO\left(\frac{1}{N}\right)
=R_\mu(z)+\OO\left(\frac{1}{N}\right).
$$
Since $R_\mu(z)/H(z)$ is bounded and analytic inside $\cC_2$, we have 
\begin{align}\label{e:I3}
I_3=\frac{1}{2\pi \ri}\int_{\cC_1\cup \cC_2} \frac{R_\mu(z)+\OO(1/N)}{(z-v)H(z)}\rd z=\OO\left(\frac{\ln N}{N}\right).
\end{align}

For $I_4$, if $z\in \cC_2$, we have $|z-v|, |z-\bar v|\gtrsim 1$, and
$$
|\cE(z)|\lesssim\frac{K}{N^2}.
$$
Therefore, it follows,
\begin{align}\label{e:I4_1}
\left|\frac{1}{2\pi i}\int_{ \cC_2}\cE(z)\psi_N^-(zN)\bE\left[\prod_{i=1}^N \left(1-\frac{\theta}{N(z-\ell_i/N)}\right)\right]\frac{\rd z}{(z-v)H(z)}\right|\lesssim \frac{K}{N^2}.
\end{align}
The estimates are more involved for $z\in \cC_1$. We recall that $\cC_1$ consists of two circles, one is centered at $v$ with radius $\eta/2$, and the other is centered at $\bar v$ with radius $\eta/2$. Similar to Lemma \ref{l:GNapprox}, for $z\in \cC_1$, we have with high probability, since $G_N(z)=O(1)$
$$
\prod_{i=1}^N \left(1-\frac{\theta}{N(z-\ell_i/N)}\right)=e^{-\theta G_N(z)+\OO(1/N\eta)}=e^{-\theta G_N(z)}+\OO\left(\frac{1}{N\eta}\right),
$$
and the same estimate holds for the expectation.
Thanks to Assumption \ref{a:wx}, \eqref{e:defE} and \eqref{e:Ebound}, we have
\begin{align}\begin{split}\label{e:I4}
&\frac{1}{2\pi i}\int_{ \cC_1}\cE(z)\psi_N^-(zN)\bE\left[\prod_{i=1}^N \left(1-\frac{\theta}{N(z-\ell_i/N)}\right)\right]\frac{\rd z}{z-v}=\OO\left(\frac{1}{N\eta}+\frac{K}{(N\eta)^3}+\frac{K^2}{(N\eta)^4}\right)\\
&\qquad-\frac{1}{2\pi i}\int_{ \cC_1}\left(\frac{tK}{N^2}\frac{1}{(v-z)^2}+\frac{\bar t K}{N^2}\frac{1}{(\bar v -z)^2}\right)\phi^-(z)\bE\left[e^{-\theta G_N(z)}\right]\frac{\rd z}{(z-v)H(z)}.
\end{split}\end{align}
Let
$$
f(z)\deq \frac{\phi^-(z)\bE\left[e^{-\theta G_N(z)}\right]}{H(z)}\,.
$$
$f$ is analytic inside $\cC_1$ so that the last integral in \eqref{e:I4} can be estimated by
\begin{align}\begin{split}\label{e:I4_2}
I_4'=&-\frac{1}{2\pi i}\frac{tK}{N^2}\int_{ \cC_1}\frac{f(z)}{(z-v)^3}\rd z-\frac{1}{2\pi i}\frac{\bar tK}{N^2}\int_{ \cC_1}\frac{f(z)}{(z-\bar v )^2(z-v)} \rd z\\
=&\frac{tK}{2N^2}f''(v)-\frac{1}{2\pi i}\frac{\bar tK}{N^2}\int_{ \cC_1}f(z)\left(\frac{1}{(v-\bar v)^2}\left(\frac{1}{z-v}-\frac{1}{z-\bar v}\right)-\frac{1}{(v-\bar v)(z-\bar v)^2}\right)\rd z\\
=&\frac{tK}{2N^2}f''(v)+\frac{\bar tK}{N^2}\left(\frac{1}{(v-\bar v)^2}\left(f(v)-f(\bar v)\right)-\frac{1}{v-\bar v}f'(\bar v)\right),
\end{split}\end{align}

Since $\phi^-(z)$ and $H(z)$ are both analytic, similar to the proof of Lemma \ref{l:GNprop}, we have 
\begin{align}\label{e:estf}
|f(v)-f(\bar v)|\lesssim (v-\bar v)+\bE[|\Im[G_N(v)]|],\quad |f'(\bar v)|\lesssim \frac{\bE[|\Im[G_N(v)]|]}{\eta}, \quad |f''(v)|\lesssim \frac{\bE[|\Im[G_N(v)]|]}{\eta^2},
\end{align}
which gives an upper bound for \eqref{e:I4_2}, 
\begin{equation}\label{e:I4_3}
I_4'\lesssim \frac{K}{N^2\eta}+\frac{K\bE[|\Im[G_N(v)]|]}{(N\eta)^2}.
\end{equation}
Thanks to \eqref{e:Gdual}, we can rewrite $f(z)$ in the following way:
\begin{eqnarray*}
f(z)=\frac{\phi^-(z)\bE\left[e^{-\theta G_N(z)}\right]}{H(z)}
&=&\frac{\phi^-(z)\bE\left[e^{-\int_{\hat a}^{\hat b}\frac{\rd x}{z-x}+\theta G^{\dual}_N(z)+\OO(\ln N/N\eta)}\right]}{H(z)}\\
&=&\frac{(z-\hat b)\phi^-(z)\bE[e^{\theta G_N^{\dual}(z)}]}{(z-\hat a)H(z)}+\OO\left(\frac{\ln N}{N\eta}\right).
\end{eqnarray*}
Since $\phi^-(z)$ vanishes at $\hat a$, $\phi^-(z)/(z-\hat a)$ is an analytic function on $\cM$. We have similar estimates as in \eqref{e:estf}, but replacing $G_N(v)$ by $G_N^{\dual}(v)$, which gives another upper bound for  $I_4'$ in \eqref{e:I4_2},
\begin{equation}\label{e:I4_4}
I_4'\lesssim \frac{K}{N^2\eta}+\frac{K\bE[|\Im[G_N^{\dual}(v)]|]}{(N\eta)^2}+\frac{\ln N K}{(N\eta)^3}.
\end{equation}
Therefore, combining \eqref{e:I4_1}, \eqref{e:I4_3} and \eqref{e:I4_4} together, we have the following estimate of $I_4$:
\begin{align}\label{e:I4}
I_4\lesssim \frac{K\min\{\bE[|\Im[G_N(v)]|],\bE[|\Im[G_N^{\dual}(v)]|]\}}{(N\eta)^2}+\OO\left(\frac{1}{N}+\frac{K}{N^2\eta}+\frac{\ln N K}{(N\eta)^3}+\frac{K^2}{(N\eta)^4}\right).
\end{align}

It follows by combining the estimates \eqref{e:I1}, \eqref{e:I2}, \eqref{e:I3} and \eqref{e:I4} all together, we have
\begin{align}\begin{split}\label{e:bootstrap0}
&\theta \sqrt{(v-A)(v-B)}
\bE\left[G_N(v)-G_\mu(v)\right]
=\OO\left(\frac{1}{N\eta}+\frac{K}{(N\eta)^2}\right)\left|\bE[\Im[G_N(v)-G_\mu(v)]]\right|+\\
&+\OO\left(\varepsilon^2+\Theta_\mu(v)+\frac{K\min\{|\Im[G_\mu(v)]|,|\Im[G_\mu^{\dual}(v)]|\}}{(N\eta)^2}+\frac{\ln N}{(N\eta)^2}+\frac{\ln N}{N}+\frac{\ln N K}{(N\eta)^3}+\frac{K^2}{(N\eta)^4}\right).
\end{split}\end{align}
Notice that $\sqrt{(v-A)(v-B)}\asymp \sqrt{\kappa_E+\eta}$, and by our assumption $N\eta\sqrt{\kappa_E+\eta}\gg 1$, and $(N\eta)^2\sqrt{\kappa_E+\eta}\gg K$, it follows by rearranging \eqref{e:bootstrap0}, that $\bE\left[G_N(v)-G_\mu(v)\right]$ is bounded by
\begin{align*}\begin{split}
\frac{1}{\sqrt{\kappa_E+\eta}}\OO\left(\varepsilon^2+\Theta_\mu(v)+\frac{K\min\{|\Im[G_\mu(v)]|,|\Im[G_\mu^{\dual}(v)]|\}}{(N\eta)^2}+\frac{\ln N}{(N\eta)^2}+\frac{\ln N}{N}+\frac{\ln NK}{(N\eta)^3}+\frac{K^2}{(N\eta)^4}\right).
\end{split}\end{align*}

This finishes the proof of \eqref{e:bootstrap}. In the following we prove \eqref{e:LDP1}. Set

$$
\varepsilon_2:=\frac{\varepsilon^2}{\sqrt{\kappa_E+\eta}}+\tilde\varepsilon_0.$$
If $\varepsilon \leq N^{-\fc} \sqrt{\kappa_E+\eta}$, and $\varepsilon\geq \tilde\varepsilon_0$, we have
$$
\varepsilon_2\leq \varepsilon N^{-\fc}+\tilde\varepsilon_0,\quad \varepsilon_2\leq 2\varepsilon.$$
Notice that 
\begin{align}\begin{split}\label{e:derdensity}
&\del_t \ln \bE_{\bP_N}\left[e^{ 2K\Re\left[\sum_{i=1}^N \ln \left(1+\frac{\al t}{N(v-\ell_i/N)}\right)\right]}\right]
=2K \bE_{\bP_N^{K, t,v,\al}}\left[\Re\left[ \al G_N\left(v+\frac{\al t}{N}\right)\right]
\right].
\end{split}\end{align}
For the righthand side of \eqref{e:derdensity}, we use \eqref{e:voiderror} and \eqref{e:saturatederror} to find 
\begin{align}\begin{split}\label{e:diffat}
G_N\left(v+\frac{\al t}{N}\right)-G_N(v)
=-\frac{\alpha t}{N^2}\frac{1}{(v-\ell_i/N)^2}+\OO\left(\frac{1}{(N\eta)^2}\right)
=\OO\left(\Theta_N(z)+\frac{1}{(N\eta)^2}\right).
\end{split}\end{align}
 \eqref{e:bootstrap} and \eqref{e:diffat} give the following  bound of the righthand side of \eqref{e:derdensity},
\begin{align*}
&\phantom{{}={}}
2K \bE_{\bP_N^{K, t,v,\al}} \left[ \Re\left[ \al G_N\left(v+\frac{\al t}{N}\right)\right]\right]\\
&=2K\bE_{\bP_N^{K, t,v,\al}}\left[\Re[\alpha(G_N(v)-G_\mu(v))]\right]+2K\Re[\alpha G_\mu(v)]+K\OO\left(\bE_{\bP_N^{K,t,v,\al}}\left[\Theta_{N}(v)\right]+\frac{1}{(N\eta)^2}\right)\\
&=2K\Re[\alpha G_\mu(v)]+K\OO\left(\varepsilon_2+\frac{|\bE_{\bP_N^{K,t,v,\alpha}}\left[\Im [G_N(v)-G_\mu(v)]\right]|}{N\eta}+\Theta_{\mu}(v)+\frac{1}{(N\eta)^2}\right)\\
&=2K\Re[\alpha G_\mu(v)]+\OO(K\varepsilon_2),
\end{align*}
where in the last two lines we used \eqref{e:Thetadif} and \eqref{e:bootstrap}.
It follows by integrating over $t$  that for any  $\al=\pm 1, \pm i$,
\begin{align*}
\bE_{\bP_N}\left[e^{ 2K\Re\left[\sum_{i=1}^N \ln \left(1+\frac{\al  }{N(v-\ell_i/N)}\right)\right]}\right]= e^{2K \Re[\al  G_\mu(v)]+\OO\left(K\varepsilon_2\right)}.
\end{align*}
Let $K_2=K/2$, for any $\al_2=\pm1, \pm \ri$, by Cauchy-Schwarz inequality, we have
$$
\bE_{\bP_N^{K_2, t,v,\al}}\left[e^{ 2K_2\Re\left[\sum_{i=1}^N \ln \left(1+\frac{\al_2  }{N(v-\ell_i/N)}\right)\right]}\right]= e^{2K_2 \Re[\al_2  G_\mu(v)]+\OO\left(K\varepsilon_2\right)}.
$$
By the Markov inequality, the above implies that, with probability $1-e^{-csK\varepsilon_2}$ with respect to the measure $\bP_N^{K_2, t,v,\al}$, the following holds
\begin{align}\label{e:realpartbound}
\Re\left[\sum_{i=1}^N \ln \left(1+\frac{\al_2  }{N(v-\ell_i/N)}\right)\right]-\Re[\alpha_2G_\mu(v)]\lesssim s\varepsilon_2,
\end{align}
where $s\geq 1$. Thanks to \eqref{e:Thetadif} and Lemma \ref{l:GNapprox}, \eqref{e:realpartbound} leads to
\begin{align}\label{e:realpartbound2}
\Re[\alpha_2(G_N(v)-G_\mu(v))]+\OO\left(\frac{|\Im[G_N(v)-G_\mu(v)]|}{N\eta}+\Theta_\mu(v)+\frac{\ln N}{(N\eta)^2}\right)\lesssim s\varepsilon_2.
\end{align}
Since $N\eta\gg 1$, and $\varepsilon_2\geq \Theta_\mu(v) +\ln N/(N\eta)^2$, we deduce by taking $\alpha_2=\pm \ri$ in \eqref{e:realpartbound2} that
\begin{align}\label{e:impartbound}
|\Im[G_N(v)-G_\mu(v)]|\lesssim s\varepsilon_2.
\end{align}
By taking $\alpha_2=\pm 1$ in \eqref{e:realpartbound2} and using \eqref{e:impartbound} for the error, we get
\begin{align}\label{e:realpartbound3}
\left|\Re[G_N(v)-G_\mu(v)]\right|\lesssim s\varepsilon_2.
\end{align}
Finally \eqref{e:impartbound} and \eqref{e:realpartbound3} together imply that 
$$
\left|G_N(v)-G_\mu(v)\right|\lesssim s\varepsilon_2.
$$
w.r.t. $\bP_N^{K_2,t,v,\alpha}$ for any $0\leq t\leq 1$ and $\alpha=\pm1, \pm \ri$.
Especially, by taking $s=1$, with high probability, w.r.t. $\bP_N^{K_2,t,v,\alpha}$, 
\begin{align}\label{e:newbound}
G_N(v)=G_\mu(v)+\OO(\varepsilon_2).
\end{align}
\eqref{e:newbound} is in the same form as \eqref{e:oldbound}. However, we pay a price here, the moment $K$ drops by half. If we  repeat the above procedure  for $\lceil 1/\fc\rceil$ times, We will reach the optimal rigidity, with probability $1-e^{-csK\tilde\varepsilon_0}$ 
$$
\left|G_N(v)-G_\mu(v)\right|\lesssim s\tilde\varepsilon_0,
$$
with respect to the measure $\bP_N^{K',t,v,\alpha}$ where $K'=2^{-\lceil 1/\fc\rceil}K\asymp K$. 
The statement \eqref{e:LDP1} follows by taking $t=0$.

\end{proof}

\section{Local Measure}
\label{s:lm}

In this Section we prove Theorem \ref{t:lowerscale}. The main idea is to localize our measure by conditionning, and therefore get estimates on finer scales. We fix small constants $r_0,r>0$, the parameters $N^{\fa}\leq M\ll N$ and $L=M^{\fb}$, as in the statement of Theorem \ref{t:lowerscale}. We recall that $\rd\mu(x)=\rho_V(x)\rd x$ denotes the constrained equilibrium measure and  $\gamma_1,\gamma_2,\cdots, \gamma_N$ the corresponding classical particle locations defined in \eqref{e:defloc}. 

\begin{definition}\label{def:rigidity}
Fix a scale $\tilde\eta\geq \ln N/N$. We say that $(\ell_1,\ell_2,\cdots, \ell_N)\in\bW_N^{\theta}$ satisfies the rigidity estimates  on the scale $\tilde \eta$ in the spectral domain $\tilde \cD$, if 
\begin{align}\label{e:rigidityMN}
\left|G_N(z)-G_\mu(z)\right|\leq \frac{\tilde \eta}{\Im[z]},\quad z\in\tilde \cD.
\end{align}
The above estimates \eqref{e:rigidityMN} combining with \eqref{e:Gdual} lead to
$$
\left|G^{\dual}_N(z)-G^{\dual}_\mu(z)\right|\leq \frac{2\tilde \eta}{\Im[z]},\quad z\in\tilde \cD.
$$
\end{definition}

In the following propositions, we collect some consequences of rigidity estimates in the spectral domains $\cD_{M,r_0}$ and $\cD_*$ (as defined in \eqref{def:D} and \eqref{def:D_0})
The proofs follow an application of the Helffer-Sj{\" o}strand functional calculus along the lines of \cite[Lemma B.1]{MR2639734}. We postpone them to the appendix \ref{ap:proofrig}. 

\begin{proposition}\label{prop:rig}
We assume $(\ell_1,\ell_2,\cdots, \ell_N)\in \bW_N^{\theta}$ satisfies the  rigidity estimates  on the scale $\tilde \eta=M/N$ in the spectral domains $\cD_{M,r_0}$ (as in Definition \ref{def:rigidity}). For any  interval $I=[a,b]\subset \cM$ with length $|I|=\eta$, we denote $\kappa_I=\dist\{I, \{A,B\}\}$. There are two cases:
\begin{enumerate}
\item the bulk case, where $A\leq a$ and $b\leq B-r_0$. We assume that $\kappa_I\geq \tilde \eta^{2/3}$, $\eta\geq \tilde \eta/\sqrt{\kappa_I}$, and
$$
\kappa_I\asymp \sup\{\dist(x, \{A,B\}): x\in I\};
$$
\item the edge case, where $a\leq A$ and $b\leq B-r_0$. We assume $\eta\geq \tilde\eta^{2/3}$.
\end{enumerate}
Then for any $C^2$ function $f$ supported on $I$, with $f(x)=\OO(1)$, $f'(x)=\OO(1/\eta)$ and $f''(x)=\OO(1/\eta^2)$, we have
\begin{align}\label{e:rig}
\left|\frac{1}{N}\sum_{i=1}^Nf(\ell_i/N)-\int f(x)\rd \mu(x)\right|\lesssim
\td \eta \ln N.
\end{align}
Moreover, there exists a constant $C$, such that uniformly for any index $i\in \qq{1, N}$ such that $\gamma_i\leq B-r_0-r/2$,
\begin{align}\label{e:location}
\gamma_{i-CN\tilde \eta\ln N }\leq \ell_i/N\leq \gamma_{i+CN\tilde \eta\ln N},
\end{align}
where we make the convention that $\gamma_i=a(N)/N$ if $i\leq 0$.
\end{proposition}

\begin{proposition}\label{prop:rig2}
We assume that $(\ell_1,\ell_2,\cdots, \ell_N)\in\bW_N^{\theta}$ satisfies the rigidity estimates  on the scale $\tilde \eta=M/N$ in the spectral domains $\cD_*$ (as in Definition \ref{def:rigidity}). Then for any $C^2$ function $f$ supported in $\cM$, with $f(x)=\OO(1)$, $f'(x)=\OO(1)$ and $f''(x)=\OO(1)$, we have
\begin{align}\label{e:rigO1}
\left|\frac{1}{N}\sum_{i=1}^Nf(\ell_i/N)-\int f(x)\rd \mu(x)\right|\lesssim
\td \eta.
\end{align}
\end{proposition}

\begin{remark}\label{r:decompint}

Given any $b\in [A, B]$, with $\kappa_b=\dist(b, \{A,B\})\gtrsim (M/N)^{2/3}$. Fix a scale $\tilde\eta\gtrsim M/N$. Let $b_0\deq b$. Given $b_k<B$, we define $b_{k+1}$ in the following way: if $B-b_{k}\geq 2^{k+1} \tilde\eta/\sqrt{\kappa_b}$, we define $b_{k+1}\deq b_k+2^k\tilde\eta/\sqrt{\kappa_b}$; otherwise, we take $b_{k+1}\deq b_k+2^{k+2}\tilde\eta/\sqrt{\kappa_b}$. Let $\hat b_N=\max\{\hat b, b(N)/N\}$. Given $B<b_k<\hat b$, we define $b_{k+1}\deq b_k+2^{k+2}\tilde\eta/\sqrt{\kappa_b}$. In this way, we get a sequence of points $b_0<b_1<b_2 <\cdots<b_{n}<B< b_{n+1}<b_{n+2}<\cdots <b_{n'}<\hat b_N<b_{n'+1}$. It is easy to check that
\begin{itemize}
\item For $k\in\qq{0,n}$, we have $b_k-A\asymp (b_0-A)+2^{k}\tilde\eta/\sqrt{\kappa_b}$ and $B-b_k\asymp 2^{n}\tilde\eta/\sqrt{\kappa_b}$.
\item For $k\in \qq{n,n'}$, we have $b_{k+1}-b_{k}=2^{k+2}\tilde\eta/\sqrt{\kappa_b}\gtrsim (M/N)^{2/3}$. 
\end{itemize}

We denote $I_k=[b_{k-1},b_{k+1}]$ for $k\in\qq{1,n'}$. With the properties of $b_k$ listed above, it is easy to check $I_k$, for $k\in \qq{1,n-1}$, are bulk intervals, and $I_k$ for $k\geq n$, are edge intervals, and they satisfy the conditions in Proposition \ref{prop:rig}. We can take bump functions $\chi_k$ supported on $I_k$, with  values in $[0,1]$, $\chi_k'(x)\asymp 1/|I_k|$ and $\chi_k'(x)\asymp 1/|I_k|^2$, such that $\chi_1(x)+\chi_2(x)+\cdots+\chi_{n'}(x)=1$ on $[b_1,\hat b_N]$. We call $\chi_1, \chi_2,\cdots, \chi_{n'}$ a \emph{dyadic partition of unity on the scale $\tilde\eta$}.

For any $C^2$ function $f$, by combining Proposition \ref{prop:rig} with a dyadic decomposition given above, we can estimate the  integral
\begin{align}\label{e:int}
\int_b^{\hat b_N} f(x)\left(\rd\mu_N(x)-\rd \mu(x)\right),
\end{align}
for some $b\in [A,B]$ with $\kappa_b=\dist(b, \{A,B\})\gtrsim (M/N)^{2/3}$ by  decomposing   it as  the sum
\begin{align}\label{e:decompint}
\int_{b_0}^{b_1} (1-\chi_1(x))f(x)\left(\rd\mu_N(x)-\rd\mu(x)\right)+\sum_{k=1}^{n'}\int \chi_k(x) f(x)\left(\rd\mu_N(x)-\rd\mu(x)\right).
\end{align}
Then we can estimate \eqref{e:decompint} term by term using Proposition \ref{prop:rig}. 

By symmetry, for the integral 
\begin{align*}
\int_{\hat a_N}^a f(x)\left(\rd\mu_N(x)-\rd\mu(x)\right),
\end{align*}
where $\hat a_N=\min\{\hat a, a(N)/N\}$, we can construct a sequence of points $a=a_0>a_1>a_2>\cdots>a_n>A>a_{n+1}>a_{n+2}>\cdots>a_{n'}>\hat a_N>a_{n'+1}$ and bump functions $\chi_1, \chi_2,\cdots, \chi_{n'}$.

\end{remark}

We assume that with high probability, the rigidity estimates hold on the scale $ M/N$ in the spectral domain $\cD_{M,r_0}\cup \cD_*$, as in Definition \ref{def:rigidity}. We fix $M\ll L \ll N$ and consider the conditioned discrete $\beta$ ensembles. More precisely, there are two cases:
\begin{enumerate}
\item the bulk case: We fix a bulk index $K\in \qq{\lfloor L/2\rfloor+1, N_0}$, where $N_0\deq \max\{i: \gamma_i\leq B-r_0-r\}$.  We condition on the particles $\ell_1, \ell_2,\cdots, \ell_K$, $\ell_{K+L+1}, \ell_{K+L+2},\cdots,\ell_N$ and resample the particles $\ell_{K+1}, \ell_{K+2},\cdots, \ell_{K+L}$. Then $\ell_{K+1}/N, \ell_{K+2}/N,\cdots, \ell_{K+L}/N\in I$, where 
\begin{align}\label{e:bulkI}
I=[a,b],\quad a=(\ell_{K}+\theta)/N,\quad b=(\ell_{K+L+1}-\theta)/N.
\end{align}
We denote $\kappa_I=\dist(I, \{A,B\})$.
\item the left edge case: We condition on the particles $\ell_{L+1}, \ell_{L+2},\cdots,\ell_N$ and resample the particles $\ell_{1}, \ell_{2},\cdots, \ell_{L}$.  Then $\ell_{1}/N, \ell_{2}/N,\cdots, \ell_{L}/N\in I$, where 
\begin{align}\label{e:edgeI}
I=[a,b], \quad a=a(N)/N,\quad b=(\ell_{L+1}-\theta)/N.
\end{align}
We denote $\kappa_I=\dist(b, \{A,B\})$.
\end{enumerate}

The conditioned discrete $\beta$-ensemble is given by 
\begin{align}\label{e:condensemble}
\bP_{L}(\ell_{K+1},\ell_{K+2},\cdots, \ell_{K+L} )=\frac{1}{Z_{L}}\prod_{K+1\leq i<j\leq K+L}\frac{\Gamma(\ell_j-\ell_i+1)\Gamma(\ell_j-\ell_i+\theta)}{\Gamma(\ell_j-\ell_i)\Gamma(\ell_j-\ell_i+1-\theta)}
\prod_{i=K+1}^{K+L} \hat w(\ell_i; L),
\end{align}
where
\begin{align*}
\hat w(Nu; L)=e^{-NV_N(u)}\prod_{j\leq K, \text{ or }\atop j\geq K+L+1}\frac{\Gamma(\ell_j-Nu+1)\Gamma(\ell_j-Nu+\theta)}{\Gamma(\ell_j-Nu)\Gamma(\ell_j-Nu+1-\theta)},
\end{align*}
if $u\in I$, otherwise $\hat w(Nu; L)=0$. 
Notice that $K\in \qq{\lfloor L/2\rfloor+1, N_0}$ (where $N_0=\max\{i:\gamma_i\leq B-r_0-r\}$) corresponds to the bulk case, $K=0$ corresponds to the left edge case. In the rest of this section, We assume that $K\in \qq{\lfloor L/2\rfloor+1, N_0}$ or $K=0$.

\begin{definition}
We define the set $\cG_{M,L,K}$ of \emph{``good" boundary conditions} as the particle configurations
\begin{align*}
(\ell_1, \ell_2,\cdots, 
\ell_K,\ell_{K+L+1}, \ell_{K+L+2},\cdots \ell_{N})
\end{align*}
 such that with high probability with respect to the conditioned discrete $\beta$-ensemble $\bP_L$, the particle configuration $(\ell_1, \ell_2,\cdots, 
\ell_K,\ell_{K+1},\ldots,\ell_{K+L},\ell_{K+L+1}, \ell_{K+L+2},\cdots \ell_{N})$ 
satisfies the rigidity estimates on the scale $M/N$ in the spectral domain $\cD_{M,r_0}\cup \cD_*$, as in Definition \ref{def:rigidity}.
\end{definition}

The following proposition is an easy consequence of Proposition \ref{prop:rig}.
\begin{proposition}\label{prop:Iproperty}
We assume that $(\ell_1, \ell_2,\cdots, 
\ell_K,\ell_{K+L+1}, \ell_{K+L+2},\cdots \ell_{N})\in \cG_{M,L,K}$, i.e. it is in the set of ``good" boundary conditions, and $I$ is as defined in \eqref{e:bulkI} or \eqref{e:edgeI}.
In the bulk case, we have $\kappa_I\gtrsim (L/N)^{2/3}$,  $|I|\asymp L/N\sqrt{\kappa_I}\lesssim \kappa_I$, and $\rho_V(x)\asymp\sqrt{\kappa_I}$. And for the left edge case $\kappa_I=b-A\asymp (L/N)^{2/3}$, and $\rho_V(x)\asymp\sqrt{[x-A]_+}$. In all cases, we have $I\subset [a(N)/N, B-r_0-r/2]$, 
\begin{align*}
\left|N\int_I \rho_V(x)\rd x-L\right|=\OO\left(M\ln N\right) .
\end{align*}
\end{proposition}

We can sample the particle configurations $(\ell_1, \ell_2, \cdots, \ell_N)$ using the following two-step sampling process: we first sample the boundary $(\ell_1, \ell_2,\cdots, 
\ell_K,\ell_{K+L+1}, \ell_{K+L+2},\cdots \ell_{N})$ according to the marginal distribution; then we sample $(\ell_{K+1}, \ell_{K+2},\cdots, \ell_{K+L})$ according to $\bP_L$. As a consequence of this point of view, the following proposition holds.
\begin{proposition}\label{prop:rigiditylocalmeasure}
We assume the rigidity estimates (as in Definition \ref{def:rigidity}) hold on the scale $M/N$ with probability $1-p_1$ for the discrete $\beta$-ensemble. Then with probability at least $1-p_1e^{c(\log N)^2}$, we have $(\ell_1, \ell_2, \cdots, \ell_K, \ell_{K+L+1}, \ell_{K+L+2},\cdots, \ell_N)\in \cG_{M,L,K}$, i.e. it is in the set of ``good" boundary conditions. 
\end{proposition}

\begin{remark}
If the rigidity estimates (as in Definition \ref{def:rigidity}) hold on the scale $M/N$ with high probability for the discrete $\beta$-ensemble, then with high probability, $(\ell_1, \ell_2, \cdots, \ell_K, \ell_{K+L+1}, \ell_{K+L+2},\cdots, \ell_N)$ belongs to $ \cG_{M,L,K}$.
\end{remark}

It follows from  Stirling's formula that
\begin{align*}
\frac{\Gamma(h+\theta)}{\Gamma(h)}=h^\theta(1+\OO(h^{-1}))=h^\theta e^{\OO(h^{-1})}.
\end{align*}
We can simplify the expression of the conditioned discrete $\beta$-ensemble \eqref{e:condensemble}
\begin{align}\begin{split}\label{e:newproduct}
\prod_{K+1\leq i<j\leq K+L}\frac{\Gamma(\ell_j-\ell_i+1)\Gamma(\ell_j-\ell_i+\theta)}{\Gamma(\ell_j-\ell_i)\Gamma(\ell_j-\ell_i+1-\theta)}
&=\prod_{K+1\leq i<j\leq K+L}|\ell_j-\ell_i|^{2\theta}e^{\OO(|\ell_j-\ell_i|^{-1})}\\
&=e^{\OO(L\ln N)}\prod_{K+1\leq i<j\leq K+L}|\ell_j-\ell_i|^{2\theta},
\end{split}\end{align}
and
\begin{align}\begin{split}\label{e:newweight}
\hat w(Nu; L)&=e^{-NV(u)+\OO(\ln N)}\prod_{j\leq K, \text{ or }\atop j\geq K+L+1}|\ell_j-Nu|^{2\theta}e^{\OO(|\ell_j-Nu|^{-1})}\\
&=e^{-NV(u)+\OO(\ln N)}\prod_{j\leq K, \text{ or }\atop j\geq K+L+1}|\ell_j-Nu|^{2\theta},
\end{split}\end{align}
where we used $|\ell_j-\ell_i|\geq |j-i|\theta$ and Assumption \ref{a:VN}. We can rewrite the conditioned discrete $\beta$-ensemble \eqref{e:condensemble} by plugging in \eqref{e:newproduct} and \eqref{e:newweight}
\begin{align}\label{e:conddist}
\bP_{L}(\ell_{K+1},\ell_{K+2},\cdots, \ell_{K+L} )=\frac{1}{Z_{L}'} e^{-H_L(\ell_{K+1},\ell_{K+2},\cdots, \ell_{K+L})+\OO(L\ln N)}
\end{align}
where the Hamiltonian is given by 
\begin{align}\label{e:defHL}
H_L=-\theta\sum_{K+1\leq i\neq j\leq K+L} \ln \left|\frac{\ell_i}{N}-\frac{\ell_j}{N}\right|+L \sum_{i=K+1}^{K+L}W\left(\frac{\ell_i}{N}\right),
\end{align}
where 
\begin{align}\label{e:Wdefinition2}
W(x)=\left\{
\begin{array}{cc}
\frac{N}{L}\left(V(x)-\frac{2\theta}{N}\sum_{j\leq K, \text{ or }\atop j\geq K+L+1}\ln \left|x-\frac{\ell_j}{N}\right|-f_V\right),  & x\in I;\\
-\infty, & x\notin I ,
\end{array}
\right.
\end{align}
where $f_V$ is as defined in \eqref{e:defFV}.
For the new potential $W(x)$, we can rewrite it as 
\begin{align}\label{e:expW}
W(x)=\frac{N}{L}\left(F_V(x)+2\theta\int_I\ln |x-y|\rho_V(y)\rd y+2\theta\int_{I^c}\ln|x-y|(\rho_V(y)\rd y-\rd \mu_N(y))\right),
\end{align}
where $F_V(x)\geq 0$ vanishes inside $[A,B]$, and $F_V'(x)$ has square root behavior outside $[A,B]$.

\begin{proposition}\label{prop:Westimate}
We assume $(\ell_1, \ell_2, \cdots, \ell_K, \ell_{K+L+1}, \ell_{K+L+2}, \cdots, \ell_N)\in \cG_{M,L,K}$, i.e. it is in the set of ``good" boundary conditions, and $I$ is as defined in \eqref{e:bulkI} or \eqref{e:edgeI}. Then for $x\in I$, we have $W(x)=NF_V(x)/L+\OO(\ln N)$ and $W'(x)=\OO(\ln N N/L)$.
\end{proposition}
\begin{proof}
For the second term in \eqref{e:expW}, if $I$ is a bulk interval, let $\kappa_I=\dist(I, \{A,B\})$. Take $x\in I$. Then from Proposition \ref{prop:Iproperty}, $|I|\asymp L/N\sqrt{\kappa_I}$ and $\rho_V(x)\asymp \sqrt{\kappa_I}$, we deduce that
\begin{align*}
\frac{2\theta N}{L}\int_I \ln |x-y|\rho_V(y)\rd y\asymp\frac{N\sqrt{\kappa_I}}{L} \int_I\ln|x-y|\rd y=\OO\left(\ln N\right).
\end{align*}
If $I$ is the left edge interval, then $b-A\asymp (L/N)^{2/3}$, where $b$ is the right endpoint of $I$, $\rho_V(x)\asymp \sqrt{[x-A]_+}$ and
\begin{align*}
\frac{2\theta N}{L}\int_I \ln |x-y|\rho_V(y)\rd y\asymp\frac{N}{L} \int_A^b\ln|x-y|\sqrt{[y-A]_+}\rd y=\OO\left(\ln N\right).
\end{align*}
For the last term in \eqref{e:expW},
\begin{align}\label{e:sumtwo}
\frac{2\theta N}{L}\int_{y\geq b}\ln|x-y|(\rho_V(y)\rd y-\rd \mu_N(y))+\frac{2\theta N}{L}\int_{y\leq a}\ln|x-y|(\rho_V(y)\rd y-\rd \mu_N(y)).
\end{align}
We first show the bound for the first term. We recall the construction from Remark \ref{r:decompint}. Let $\kappa_b=\dist(b, \{A,B\})$, $\kappa_b\gtrsim (L/N)^{2/3}$, and  construct the bump functions $\chi_1,\chi_2,\cdots, \chi_{n'}$ with scale $\tilde \eta=M/N$.
\begin{align}\begin{split}\label{e:errorsum}
&\int_{y\geq b}\ln|x-y|(\rho_V(y)\rd y-\rd \mu_N(y))\\
=&\int_{b_0}^{b_1}(1-\chi_1(y))\ln|x-y|(\rho_V(y)\rd y-\rd \mu_N(y))+\sum_{k=1}^{n'} \int_{y\geq b}\chi_k(y)\ln|x-y|(\rho_V(y)\rd y-\rd \mu_N(y)).
\end{split}
\end{align}
The first term in \eqref{e:errorsum} can be bounded by $\ln N M/N$. For the sum, from our construction of $\chi_1,\chi_2,\cdots, \chi_{n'}$, for each $k\in \qq{1,n'}$, we have
\begin{align*}
\chi_k(y)\ln |x-y|=\OO(\ln N),\quad
\del_y\left(\chi_k(y)\ln |x-y|\right)=\OO(\ln N/|I_k|),\quad
\del_{y}^2\left(\chi_k(y)\ln |x-y|\right)=\OO(\ln N/|I_k|^2),
\end{align*}
where $\chi_k(y)$ is supported on $I_k$. Moreover, either $I_k\subset [b, B-r_0]$, or $I_k$ is not completely contained in $[b,B-r_0]$, in which case $|I_k|\gtrsim 1$. Therefore, it follows from Proposition \ref{prop:rig} and \ref{prop:rig2}, we have
\begin{align*}
\begin{split}
\sum_{k=1}^{n'} \int_{y\geq b}\chi_k(y)\ln|x-y|(\rho_V(y)\rd y-\rd \mu_N(y))\lesssim\sum_{k=1}^{n'}\frac{(\ln N)^2 M}{N}\lesssim \frac{(\ln N)^3M}{N}.
\end{split}\end{align*}
 The second term in \eqref{e:sumtwo} can be estimated in the same way, and the first claim $W(x)=NF_V(x)/L+\OO(\ln N)$ follows.

For $W'(x)$, using \eqref{e:Wdefinition2}, we have  for $x\in I$
\begin{align*}\begin{split}
W'(x)
&=\frac{N}{L}\left(V'(x)-\frac{2\theta}{N}\sum_{j\leq K}\frac{1}{x-\ell_i/N}-\frac{2\theta}{N}\sum_{j\geq K+L+1}\frac{1}{x-\ell_i/N}\right)\\
&=\frac{N}{L}\left(\OO(\ln N)+\frac{4\theta}{N}\sum_{j\geq 1}\frac{1}{j\theta/N }\right)=\OO\left(\frac{N\ln N }{L}\right).
\end{split}\end{align*}
\end{proof}

We denote the constrained equilibrium measure of the conditioned discrete $\beta$-ensemble \eqref{e:conddist} as the unique probability measure on $I$  characterized by
\begin{align}\label{e:expW2}
W(x)-2\theta \int_{a}^b \ln|x-y|\rho_W(y)\rd y-f_W=F_W(x)=\left\{
\begin{array}{lll}
\leq 0,  &\text{on }\rho_W(x)=N/L\theta,\\
=0,  &\text{on } 0< \rho_W(x)< N/L\theta,\\
\geq 0, &\text{on }  \rho_W(x)=0.
\end{array}
\right.
\end{align}

We will need some  properties of the constrained equilibrium measure. Given an interval $I\subset \bR$,  a positive measure $\sigma$ supported on $I$ with total mass $\|\sigma\|>1$, and a potential function $Q$.
We denote $\rho_{Q}$ the unconstrained equilibrium measure, i.e the unique probability measure such that there exists a constant $f_{Q}$ so that
\begin{align}\label{defrhoW}
Q(x)-2\theta \int_I \ln|x-y|\rho_Q(y)\rd y-f_Q=\left\{
\begin{array}{lll}
=0,  &\text{on } \rho_Q(x)>0,\\
\geq 0, &\text{on }  \rho_Q(x)=0.
\end{array}
\right.
\end{align}
The equilibrium measure constrained by being bounded above by $\sigma$ is denoted $ \rho^\sigma_Q$.

\begin{theorem}\label{t:bound}[Theorem 2.6 \cite{MR1482996}]  $\supp(\rho_Q)\subset \supp(\rho_Q^\sigma)$ and 
\begin{align}\nonumber
{\bf1}_{\{\rho_Q^\sigma<\sigma\}}\rho_Q^\sigma\geq {\bf1}_{\{\rho_Q^\sigma<\sigma\}}\rho_Q.
\end{align}
\end{theorem}

\begin{theorem}\label{t:supp}[Theorem 2.17 \cite{MR1482996}] If $Q$ is convex on $I$, then the support of $\rho_Q^\sigma$ is a single interval.
\end{theorem}

\begin{theorem} \label{t:dual}[Corollary 2.11 \cite{MR1482996}] Let 
\begin{align}\nonumber
s=\|\sigma\|-1,\quad
U=\frac{1}{s}\left(2\theta\int_I \ln |x-y|\sigma(y)\rd y-Q(x)\right).
\end{align}
Then we have
\begin{align}\nonumber
\sigma=\rho^\sigma_W+s\rho^{\sigma/s}_U.
\end{align}
\end{theorem}

\begin{proposition}\label{prop:equilibriumstructure}
We assume $(\ell_1, \ell_2, \cdots, \ell_{K}, \ell_{K+L+1}, \ell_{K+L+2}, \cdots, \ell_N)\in \cG_{M,L,K}$, i.e. it is in the set of ``good" boundary conditions, and $I$ is as defined in \eqref{e:bulkI} or \eqref{e:edgeI}. Then both $\rho_W$ and $N/L\theta-\rho_W$ are supported on single intervals. Especially, there exists a subinterval $[\alpha, \beta]\subset I=[a,b]$, such that 
\begin{enumerate}
\item For $x\in (\alpha, \beta)$, $0<\rho_W(x)<N/L\theta$;
\item In the bulk case,  if $a<\alpha$, then $\rho_W(x)\equiv 0$ or $\rho_W(x)\equiv N/L\theta$ on $x\in [a,\alpha]$; if $\beta<b$, then $\rho_W(x)\equiv 0$ or $\rho_W(x)\equiv N/L\theta$ on $x\in [\beta,b]$. In the left edge case: $a<\alpha$ and $\rho_W(x)\equiv 0$ on $x\in [a,\alpha]$; if $\beta<b$, then $\rho_W(x)\equiv 0$ or $\rho_W(x)\equiv N/L\theta$ for $x\in [\beta,b]$. 
\end{enumerate}
\end{proposition}
\begin{remark} The previous proposition uses crucially that $1/\theta-\rho_{V}$ is uniformly bounded below by a positive constant  on $I$.
\end{remark}
\begin{proof}
Thanks to Theorems \ref{t:supp} and \ref{t:dual}, we only need to show that $W$ and $2(N/L)\int_I\ln |x-y|\rd y-W$ are convex on $I$. We first check that $W$ is convex,
\begin{align}\label{e:w''}
W''(x)=\frac{N}{L}\left(V''(x)+\frac{2\theta}{N}\sum_{j\geq K+L+1}\frac{1}{(x-\ell_i/N)^2}+\frac{2\theta}{N}\sum_{j\leq K}\frac{1}{(x-\ell_i/N)^2}\right).
\end{align}
We  study the first sum, the second sum is analogous. Let $\kappa_b=\dist(b, \{A,B\})$, notice that if $I$ is a bulk interval $\kappa_b\gtrsim (L/N)^{2/3}$, and if $I$ is the left edge interval $\kappa_b\asymp (L/N)^{2/3}$. We take an interval $J=[b, b']$, where $b'\deq b+c(L/N\sqrt{\kappa_b})$, for some small $c>0$, such that $J$ satisfies the conditions in Proposition \ref{prop:rig}. Then we have
\begin{align}\nonumber
\int_{J}\rd\mu_N(x)=\int_J \rho_V(x)\rd x +\OO(\ln N M/N)\asymp L/N.
\end{align}
Therefore, we can lower bound the first sum in \eqref{e:w''} for all $x\in I=[a,b]$:
\begin{align}\nonumber
\frac{1}{N}\sum_{j=K+L+1}^{N}\frac{1}{(x-\ell_i/N)^2}
\gtrsim \frac{1}{N}\sum_{j\geq K+L+1: \atop\ell_j/N\leq b'}\frac{1}{(x-b')^2}
\gtrsim \frac{L}{N}\frac{1}{(x-b')^2}
\gtrsim\frac{ L}{N}\frac{1}{(L/N\sqrt{\kappa_b})^2}
\gtrsim \frac{N}{L\kappa_b}\gtrsim \left(\frac{N}{L}\right)^{1/3}.
\end{align} We deduce  that $W''(x)\gtrsim (N/L)^{4/3}$, and $W$ is convex on $I$.

The second derivative of $2(N/L)\int_I\ln |x-y|\rd y-W$ is given by
\begin{align}\nonumber
\frac{2N}{L}\left(\frac{1}{x-a}+\frac{1}{b-x}-\frac{V''(x)}{2}-\frac{\theta}{N}\sum_{j\geq K+L+1}\frac{1}{(x-\ell_i/N)^2}-\frac{\theta}{N}\sum_{j\leq K}\frac{1}{(x-\ell_i/N)^2}\right).
\end{align}
Notice that for any $\ell> 0$
\begin{align}\nonumber
 \frac{\theta}{N}\sum_{j=1}^{\infty}\frac{1}{(\ell+j\theta/N)^{2}} \le \theta\int_0^\infty\frac{dx}{(\ell+\theta x)^{2}}=\frac{1}{\ell}.
\end{align}
Let $J$ be the interval as defined above. The same argument gives
\begin{align*}
\frac{1}{b-x}-\frac{\theta}{N}\sum_{j=K+L+1}^{N}\frac{1}{(x-\ell_j/N)^2}
&\geq \frac{\theta}{N}\sum_{j=1}^{\infty}\frac{1}{(b-x+j\theta/N)^{2}} -\frac{\theta}{N}\sum_{j=K+L+1}^{N}\frac{1}{(x-\ell_j/N)^2}\\
&\geq \frac{\theta}{(b'-x)^2}\left(|J|/\theta-\int_J \rd\mu_N(x)\right)\\
&\geq  \frac{\theta}{(b'-x)^2}\left(|J|/\theta-\int_J \rho_V(x)\rd x -\OO(M/N)\right)\\
&\gtrsim \frac{|J|}{(b'-x)^2}\gtrsim \frac{N\sqrt{\kappa_b}}{L}\gtrsim \left(\frac{N}{L}\right)^{2/3},
\end{align*}
where we used that $\ell_{K+L+1}/N-x\geq \theta/N$, $|\ell_{j+1}-\ell_i|\geq \theta$ and $\rho_V(x)<\theta^{-1}$. 
Similarly if $I$ is a bulk interval (otherwise this sum is empty), 
\begin{align}\nonumber
\frac{1}{x-a}-\frac{\theta}{N}\sum_{j=1}^{K}\frac{1}{(x-\ell_i/N)^2}
\gtrsim \left(\frac{N}{L}\right)^{2/3}.
\end{align}
It follows that $\left((N/L)\int_I\ln |x-y|\rd y-W\right)''\gtrsim (N/L)^{5/3}$, and  $(N/L)\int_I\ln |x-y|\rd y-W$ is convex.

\end{proof}

\begin{proposition}\label{prop:xiWestimate}
If $(\ell_1, \ell_2,\cdots, \ell_{K},\ell_{K+L+1}, \ell_{K+L+2}, \cdots, \ell_{N})\in \cG_{M,L,K}$, i.e. it is in the set of ``good" boundary conditions, and $I$ is as defined in \eqref{e:bulkI} or \eqref{e:edgeI}, then for $x\in I$, we have $F_W(x)=NF_V(x)/L+\OO(\ln N)$. Moreover for any $p>0$, on $[a+N^{-p},\alpha-N^{-p}]\cup [\beta+N^{-p}, b-N^{-p}]$ ($\alpha$ and $\beta $ are as in Proposition \ref{prop:equilibriumstructure}), we have $F_W'(x)=\OO(\ln N N/L)$. 
\end{proposition}
\begin{proof}
From the defining relation \eqref{e:expW2} of $F_W$,  and $W(x)=NF_V(x)/L+\OO(\ln N)$ from Proposition \ref{prop:Westimate}, we have
\begin{align}\nonumber
F_W(x)=\frac{NF_V(x)}{L}+\OO(\ln N) +\int_I \ln |x-y|\rho_W(y)\rd y +f_W
=\frac{NF_V(x)}{L}+f_W+\OO(\ln N),
\end{align}
where we used that $\rho_W$ has total mass $1$ and $\rho_W\leq N/L\theta$. If $I$ is a bulk interval, then for any $x\in \supp\rho_W$, we have $x\in [A,B]$, and $F_W(x)=F_V(x)=0$. Therefore $f_W=\OO(\ln N)$. If $I$ is the left edge interval, there are two cases. If $\rho_W\equiv 1$ on $x\in [\beta, b]$ (where $\beta$ is as in Propostion \ref{prop:equilibriumstructure}), then $b-\beta\lesssim L/N\ll b-A\asymp (L/N)^{2/3}$. As a consequence $\beta\in [A,B]$, $F_W(\beta)=0$ and $F_{V}(\beta)=0$. We have $f_W=\OO(\ln N)$. Otherwise, we have $F_W(a)\geq 0$ and $F_V(a)= 0$. So $f_W\geq -C\ln N$. On the other hand, for any $x\in \supp \rho_W$, we have $F_W(x)=0$ and $F_V(x)\geq 0$. So $f_W\leq C\ln N$. In any case, we have $f_W=\OO(\ln N)$, and $F_W(x)=NF_V(x)/L+\OO(\ln N)$.

From Proposition \eqref{prop:Westimate}, $W'(x)=\OO(\ln N N/L)$. Thus,
\begin{align}\nonumber
F_W'(x)=\OO(\ln N N/L) +P.V.\int_I \frac{\rho_W(y)}{x-y}\rd y.
\end{align}
By symmetry, we assume $x\in [\beta+N^{-p}, b-N^{-p}]$ (the case when $x\in [a+N^{-p}, \alpha-N^{-p}]$ is analogous). Since $\rho_W\leq N/L\theta$ and on $[\beta,b]$ $\rho_W\equiv 0$ or $\rho_W \equiv N/L\theta$,
\begin{align}\nonumber
P.V.\int_I \frac{\rho_W(y)}{x-y}\rd y=
\int_a^\beta \frac{\rho_W(y)}{x-y}\rd y+P.V. \int_{\beta}^b  \frac{\rho_W(y)}{x-y}\rd y=\OO(\ln N N/L). 
\end{align}

\end{proof}

For any particle configuration $\ell_{K+1}< \ell_{K+2}<\cdots< \ell_{K+L}\in \bW_N^\theta$, we denote  $ \mu_L=1/L \sum_{i=K+1}^{K+L}\delta_{\ell_i/N}$, and $\rho_L(x)\rd x$ the convolution of the empirical measure $\mu_L$ with uniform measure on the interval $[-N^{-p}/2, N^{-p}/2]$. For any measure $\rho(x)\rd x$, we denote 
\begin{align}\label{def:energyfunctional}
\cD^{2}[\rho]=-\int \ln|x-y| \rho(x)\rho(y)\rd x\rd y.
\end{align}
We have the following large deviation estimate:
\begin{proposition}\label{prop:ldb}
Given $(\ell_1, \ell_2,\cdots, \ell_{K},\ell_{K+L+1}, \ell_{K+L+2}, \cdots, \ell_{N})\in \cG_{M,L,K}$, i.e. it is in the set of ``good" boundary conditions, and $I$ is as defined in \eqref{e:bulkI} or \eqref{e:edgeI}. Let $\rho_W$ be the density defined by \eqref{defrhoW}. Then there exists $C\in \bR$ such that for all $\gamma>0$ we have
\begin{align}\label{e:ldb}
\bP_L\left(\cD [\rho_W-\rho_L]\geq \gamma\right)\leq \exp\left(C\ln N L-\gamma^2 L^2\right),
\end{align}
and for any Lipschitz function $g(x)$ on $I$,  if $\|g\|_{{1/2}}=\left(\int_0^{\infty} s|\hat g|^{2} ds \right)^{1/2}$, we have
\begin{align}\label{e:integ}
\bP_L\left(\left|\int g(x)(\rd \mu_L(x)-\rho_W(x)\rd x)\right|\geq \gamma \|g\|_{1/2}+N^{-p}\|g\|_{\text{Lip}}\right)\leq \exp\left(C\ln N L-\gamma^2 L^2\right).
\end{align}
\end{proposition}
\begin{proof}
The Hamiltonian $H_L$ from \eqref{e:defHL} can be written as
\begin{align}\begin{split}\label{e:ldb1}
&L^{-2}H_L(\ell_{K+1}, \ell_{K+2},\cdots, \ell_{K+L})
=-\theta \int_{x\neq y}\ln|x-y|\rd\mu_L(x)\rd\mu_L(y)+\int W(x)\rd \mu_L(x)\\
=&-\theta \int \ln|x-y|\rho_L(x)\rho_L(y)\rd x\rd y+\int W(x)\rho_L(x)\rd x+\OO\left(\frac{\ln N}{L}\right)\\
=&\theta \cD^2[\rho_W-\rho_L]+\theta \int  \ln|x-y|\rho_W(x)\rho_W(y)\rd x\rd y+\int F_W(x)\rho_L(x)\rd x+f_W+\OO\left(\frac{\ln N}{L}\right),
\end{split}\end{align}
where we used 
\begin{align}\nonumber
W(x)=2\theta\int \ln |x-y|\rho_W(y)\rd y +F_W(x)+f_W.
\end{align}
From the relation \eqref{e:expW2}, $F_W(x)<0$ happens only if $x\in [a,\alpha]$ and $\rho_W\equiv N/L\theta$, or $x\in [\beta, b]$ and $\rho_W\equiv N/L\theta$, where $\alpha,\beta$ are as in Proposition \eqref{prop:equilibriumstructure}.
\begin{align}\label{e:intFW}
\int F_W(x)\rho_L(x)\rd x\geq \int_{F_W<0} F_W(x)\rho_L(x)\rd x.
\end{align}
In the following, we estimate the right hand side of \eqref{e:intFW}. Say $F_W<0$ on $[a,\alpha]$ as in Proposition \eqref{prop:equilibriumstructure}. Since from Proposition \ref{prop:xiWestimate}, we know $F_W(x)=F_V(x)N/L+\OO(\ln N)$ and $F_W'(x)=\OO(\ln N N/L)$. Especially,  we have       $1_{F_W<0}F_W(x)=\OO(\ln N)$. It follows
\begin{align}\nonumber
\int_{a}^{\alpha} F_W(x)\rho_L(x)\rd x=\frac{N}{\theta L}\sum_{\ell_i\in [a,\alpha]}\int_{\ell_i-\theta/2N}^{\ell_i+\theta/2N} F_W(x)\rd x +\OO\left(\frac{\ln N}{L}\right)\geq
\frac{N}{\theta L}\int_{a}^{\alpha} F_W(x)\rd x +\OO\left(\frac{\ln N}{L}\right).
\end{align}
If $F_W<0$ on $[\beta, b]$, the same estimate holds. And it follows,
\begin{align}\label{e:xiWbound}
\int F_W(x)\rho_L(x)\rd x\geq \frac{N}{\theta L}\int_{F_W< 0} F_W(x)\rd x +\OO\left(\frac{\ln N}{L}\right).
\end{align}
Thus we can rewrite the distribution $\bP_L$ as
\begin{align}\begin{split}\label{e:upperbound0}
&\bP_L(\ell_{K+1}, \ell_{K+2}\cdots, \ell_{K+L})
=\frac{1}{Z_L'}e^{-H_L(\ell_{K+1}, \cdots, \ell_{K+L})+\OO(L\ln N)}\\
\leq &\frac{1}{Z_L'}\exp\left(L^2\left(\theta \cD^2[\rho_W]-\theta \cD^2[\rho_W-\rho_L]-\frac{N}{\theta L}\int_{F_W<0} F_W(x)\rd x-f_W+\OO\left(\frac{\ln N}{L}\right)\right)\right).
\end{split}\end{align}

In the following we obtain a lower bound for the partition function $Z_L'$. Let $x_i$, $i=1,2,\cdots, L$ be quantiles of $\rho_W$ defined through
\begin{align}\nonumber
\int_a^{x_i}\rho_W(x)\rd x=\frac{i-1}{L}, \quad 1\leq i \leq L.
\end{align}
Since $\rho_W \leq N/L\theta$, $\theta(x_{i+1}-x_i)\geq 1/N$, and there exists an universal constant $C$ and a particle configuration $(\tilde \ell_{K+1}, \tilde \ell_{K+2},\cdots,\tilde  \ell_{K+L})\in \bW_N^\theta$ such that  
\begin{align}\label{e:l-x}
|\tilde \ell_{K+i}/N-x_i|\leq C/N.
\end{align}
We denote the empirical particle density $ \rd \tilde\mu_L=1/L \sum_{i=1}^{L}\delta_{x_i}$. We can approximate the Hamiltonian $H_L(\tilde \ell_{K+1}, \tilde \ell_{K+2}, \cdots, \tilde \ell_{K+L})
$ by $H_L(Nx_1, Nx_2,\cdots, Nx_L)$, and their difference is negligible,
\begin{align}\begin{split}\nonumber
&H_L(\tilde \ell_{K+1}, \tilde \ell_{K+2}, \cdots, \tilde \ell_{K+L})
=-\theta \sum_{i\neq j} \ln \left|\tilde \ell_{K+i}/N-\tilde \ell_{K+j}/N\right|+L\sum_i W(\tilde \ell_{K+i}/N)\\
=&H_L(Nx_1, Nx_2,\cdots, Nx_L)-\theta \sum_{i\neq j} \ln \left|\frac{\tilde \ell_{K+i}/N-\tilde \ell_{K+j}/N}{x_i-x_j}\right|+L\sum_i \left(W(\tilde \ell_{K+i}/N)-W(x_i)\right)\\
=&H_L(Nx_1, Nx_2,\cdots, Nx_L)+\OO\left(\sum_{i\neq j}  \left|\frac{(\tilde \ell_{K+i}/N-x_i)-(\tilde \ell_{K+j}/N-x_j)}{x_i-x_j}\right|+L\sum_i \|W'\|_{\infty}| \tilde\ell_{K+i}/N-x_i|\right)\\
=&H_L(Nx_1, Nx_2,\cdots, Nx_L)+\OO\left(L\ln N\right),
\end{split}\end{align}
where we used $|\tilde \ell_i/N-x_i|\leq C/N$ from \eqref{e:l-x}, $|x_i-x_j|\geq \theta|i-j|/N$ and $\|W'\|_{\infty}=\OO(N\ln N/L)$ from Proposition \ref{prop:Westimate}. Similar to the derivation in \eqref{e:ldb1}, we have
\begin{align}\begin{split}\label{e:HLx}
&L^{-2}H_L(Nx_1,Nx_2,\cdots, Nx_L)
=-\theta \int_{x\neq y} \ln \left|x-y\right|\rd\tilde \mu_L \rd \tilde \mu_L+\int  W\left(x\right)\rd \tilde\mu_L\\
=&-\theta \cD^2[\rho_W]-\theta \int_{x\neq y}\ln |x-y|(\rho_W(x)\rd x-\rd\tilde\mu_L(x))(\rho_W(y)\rd y -\rd\tilde \mu_L(x))+\int F_W(x)\rd\tilde\mu_L(x) +f_W.
\end{split}\end{align}
For the second term on the righthand side of \eqref{e:HLx}, we can divide the integral region into small rectangles, $[x_i, x_{i+1})\times [x_j, x_{j+1})$, for any $1\leq i,j\leq L$. We have the trivial upper bound for $i<j$
\begin{align}\label{e:trivialbound}
\int_{x\in[x_i,x_{i+1}),\atop y\in[x_j, x_{j+1})}1_{x\neq y}\ln |x-y|(\rho_W(x)\rd x-\rd\tilde\mu_L(x))(\rho_W(y)\rd y -\rd\tilde \mu_L(x))
=\OO\left(\frac{\ln N}{L^2}\right).
\end{align}
We can directly bound the integrals over rectangles close to diagonal by \eqref{e:trivialbound}, 
\begin{align}\begin{split}\nonumber
 &\int_{x\neq y}\ln |x-y|(\rho_W(x)\rd x-\rd\tilde\mu_L(x))(\rho_W(y)\rd y -\rd\tilde \mu_L(x))\\
=&2\sum_{i+2\leq j} \int_{x\in[x_i,x_{i+1}),\atop y\in[x_j, x_{j+1})}\ln |x-y|(\rho_W(x)\rd x-\rd\tilde\mu_L(x))(\rho_W(y)\rd y -\rd\tilde \mu_L(x))+\OO\left(\frac{\ln N}{L}\right)\\
=&2\sum_{i+2\leq j}  \int_{x\in[x_i,x_{i+1}),\atop y\in[x_j, x_{j+1})}\ln\frac{ (y-x)(x_j-x_i)}{(y-x_i)(x_j-x)}\rho_W(x)\rho_W(y)\rd x\rd y +\OO\left(\frac{\ln N}{L}\right).
\end{split}\end{align}
For $x\in[x_i,x_{i+1})$, $y\in [x_j,x_{j+1})$ and $j\geq i+2$, we have
\begin{align}\nonumber
0\leq \ln\frac{ (y-x)(x_j-x_i)}{(y-x_i)(x_j-x)}
\leq \ln \frac{(x_{j+1}-x_{i+1})(x_{j}-x_i)}{(x_{j+1}-x_i)(x_{j}-x_{i+1})}.
\end{align}
Therefore, it follows
\begin{align}\begin{split}\nonumber
0&\leq 2\sum_{i+2\leq j}  \int_{x\in[x_i,x_{i+1}),\atop y\in[x_j, x_{j+1})}\ln\frac{ (y-x)(x_j-x_i)}{(y-x_i)(x_j-x)}\rho_W(x)\rho_W(y)\rd x\rd y\\
&\leq\frac{2}{L^2}\ln \prod_{i+2\leq j} \frac{(x_{j+1}-x_{i+1})(x_{j}-x_i)}{(x_{j+1}-x_i)(x_{j}-x_{i+1})}\\
&\leq\frac{2}{L^2}\ln \prod_{j=3}^L \frac{(x_{j+1}-x_{j-1})(x_{j}-x_{j-1})}{(x_{3}-x_1)(x_{L+1}-x_{1})}=\OO\left(\frac{\ln N}{L}\right),
\end{split}\end{align}
where we used that $\theta/N\leq x_j-x_i\leq 1$ for $i<j$. The second term on the righthand side of \eqref{e:HLx} can be bounded,
\begin{align}\label{e:HLx2}
\theta \int_{x\neq y}\ln |x-y|(\rho_W(x)\rd x-\rd\tilde\mu_L(x))(\rho_W(y)\rd y -\rd\tilde \mu_L(x))=\OO\left(\frac{\ln N}{L}\right).
\end{align}
In the following we estimate the third term on the righthand side of \eqref{e:HLx}. We recall the defining relations \eqref{e:expW2} of $\rho_W$. $\tilde \mu_L(x)$ is supported on the support of $\rho_W$, so on the support of $\tilde \mu_L(x)$, we have $F_W(x)\leq 0$, and 
$$
\int F_W(x)\rd\tilde\mu_L(x)=\int_{F_W<0} F_W(x)\rd\tilde\mu_L(x)
$$
Moreover, $F_W(x)<0$ only if $x\in [a,\alpha]$ and $\rho_W\equiv N/L\theta$ on $[a,\alpha]$, or $x\in [\beta, b]$ and $\rho_W\equiv N/L\theta$ on $[\beta,b]$. Say $F_W(x)<0$ on $[a,\alpha]$. From our choices of $x_i$, if $x_i,x_{i+1}\in [a,\alpha]$, then $x_{i+1}-x_i=\theta/N$. The same argument as for \eqref{e:xiWbound} implies 
\begin{align}\nonumber
\int_{a}^{\alpha} F_W(x)\rd \tilde  \mu_L(x)\rd x=\frac{N}{\theta L}\int_{a}^{\alpha} F_W(x) \rd x+\left(\frac{\ln N}{L}\right),
\end{align}
and the same estimate holds for the integral over $[\beta, b]$ if $\rho_W\equiv N/L\theta$ on $[\beta, b]$. It follows,
\begin{align}\label{e:HLx3}
\int F_W(x)\rd\tilde\mu_L(x)=\frac{N}{\theta L}\int_{F_W<0} F_W(x)\rd x+\left(\frac{\ln N}{L}\right).
\end{align}

Combining \eqref{e:HLx2} and \eqref{e:HLx3} together, we conclude that
\begin{align}\begin{split}\label{e:ldb2}
Z_L'&\geq \bP_L(\tilde \ell_1, \tilde \ell_2,\cdots,\tilde  \ell_L)
=\exp\left(L^2\left(\theta \cD^2[\rho_W]-\frac{N}{\theta L}\int_{F_W<0 }F_W(x)\rd x-f_W+\OO\left(\frac{\ln N}{L}\right)\right)\right).\\
\end{split}\end{align}
\eqref{e:ldb} follows from \eqref{e:upperbound0} and \eqref{e:ldb2}. And \eqref{e:integ} follows from \eqref{e:ldb} by the same argument as in \cite[Corollary 2.16]{Borodin2016} (see also the proof of Proposition \ref{p:LDP0}).
\end{proof}

We collect below a few estimates on the local equilibrium measure that we prove in the appendix.

\begin{proposition}\label{prop:abestimate}[Estimates on $\alpha,\beta$]
We assume that $(\ell_1, \ell_2,\cdots, \ell_K, \ell_{K+L+1}, \ell_{K+L+2}, \cdots, \ell_N)\in \cG_{M,L,K}$, i.e. it is in the set of ``good" boundary conditions, $I$ is as defined in \eqref{e:bulkI} or \eqref{e:edgeI}, and $[\alpha,\beta]\subset I$ is from Proposition \ref{prop:equilibriumstructure}. Then we have the following estimates for the local equilibrium measure:
\begin{enumerate}
\item The bulk case: Let $\kappa_I=\dist(I, \{A,B\})$. If $\rho_W(x)\equiv 0$ on $[a,\alpha]$, then $\alpha-a=\OO(\ln NM/N\sqrt{\kappa_I})$. If $\rho_W(x)\equiv N/L\theta$ on $[a,\alpha]$, then $\alpha-a=\OO(\ln NM/N)$. If $\rho_W(x)\equiv 0$ on $[\beta,b]$, then $b-\beta=\OO(\ln NM/N\sqrt{\kappa_I})$. If $\rho_W(x)\equiv N/L\theta$ on $[\beta,b]$, then $b-\beta\equiv \OO(\ln NM/N)$.  
\item The left edge case: Let $\kappa_I=\dist(b, A)$. If $\rho_W(x)\equiv 0$ on $[\beta,b]$, then $b-\beta=\OO(\ln N M/N\sqrt{\kappa_I})$. If $\rho_W(x)\equiv N/L\theta$ on $[\beta,b]$, then $b-\beta=\OO(\ln N M/N)$.  
\end{enumerate}
\end{proposition}

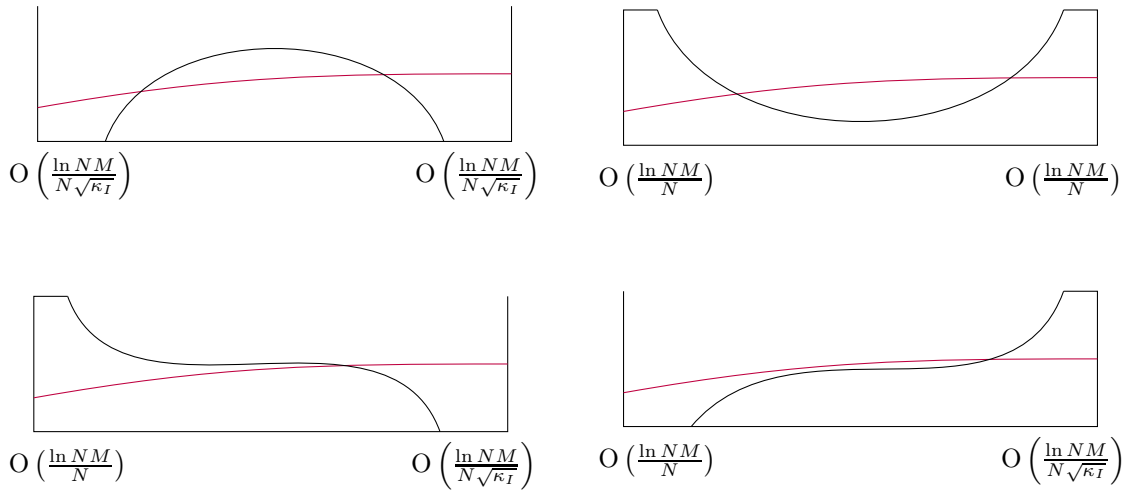
\begin{figure}[h]
\caption{equilibrium measure of localized discrete $\beta$-ensemble(bulk case).}

\vspace{1cm}
\begin{minipage}{0.4\textwidth}
\begin{tikzpicture}[scale=0.9]
\centering
\draw (0,2)--(0,0)--(7,0)--(7,2);
\draw[purple] (0,0.5) to [out=10, in=180] (7,1);
\draw (1,0) to [out=70,in=110] (6,0);
\node at (0.5,-0.5) {$\OO\left(\frac{\ln NM}{N\sqrt{\kappa_I}}\right)$};
\node at (6.5,-0.5) {$\OO\left(\frac{\ln NM}{N\sqrt{\kappa_I}}\right)$};
\end{tikzpicture}
\end{minipage}
\hspace{1cm}
\begin{minipage}{0.4\textwidth}
\begin{tikzpicture}[scale=0.9]
\centering
\draw (0.5,2)--(0,2)--(0,0)--(7,0)--(7,2)--(6.5,2);
\draw[purple] (0,0.5) to [out=10, in=180] (7,1);
\draw (0.5,2) to [out=-70,in=-110] (6.5,2);
\node at (0.5,-0.5) {$\OO\left(\frac{\ln N M}{N}\right)$};
\node at (6.5,-0.5) {$\OO\left(\frac{\ln N M}{N}\right)$};
\end{tikzpicture}
\end{minipage}

\vspace{1cm}
\begin{minipage}{0.4\textwidth}
\begin{tikzpicture}[scale=0.9]
\centering
\draw (0.5,2)--(0,2)--(0,0)--(7,0)--(7,2);
\draw[purple] (0,0.5) to [out=10, in=180] (7,1);
\draw (0.5,2) to [out=-70,in=110] (6,0);
\node at (0.5,-0.5) {$\OO\left(\frac{\ln N M}{N}\right)$};
\node at (6.5,-0.5) {$\OO\left(\frac{\ln N M}{N\sqrt{\kappa_I}}\right)$};
\end{tikzpicture}
\end{minipage}
\hspace{1cm}
\begin{minipage}{0.4\textwidth}
\begin{tikzpicture}[scale=0.9]
\centering
\draw (0,2)--(0,0)--(7,0)--(7,2)--(6.5,2);
\draw[purple] (0,0.5) to [out=10, in=180] (7,1);
\draw (1,0) to [out=50,in=-110] (6.5,2);
\node at (0.5,-0.5) {$\OO\left(\frac{\ln N M}{N}\right)$};
\node at (6.5,-0.5) {$\OO\left(\frac{\ln N M}{N\sqrt{\kappa_I}}\right)$};
\end{tikzpicture}
\end{minipage}
\end{figure}

\begin{proposition}\label{prop:bulkeq} [Estimates on the local density in the bulk]
We assume that the particle configuration $(\ell_1, \ell_2,\cdots, \ell_K, \ell_{K+L+1}, \ell_{K+L+2}, \cdots, \ell_N)\in \cG_{M,L,K}$, i.e. it is in the set of  ``good" boundary conditions, and $I=[a,b]$ is a bulk interval, i.e. $K\in\qq{\lfloor L/2\rfloor+1, N_0}$ (where $N_0\deq \max\{i: \gamma_i\leq B-r_0-r\}$), $a=(\ell_K+\theta)/N$ and $b=(\ell_{K+L+1}-\theta)/N$. Then for any $x\in I$ such that $x-a\asymp |I|$ and $b-x\asymp |I|$, the following holds
$$
\left|\rho_W(x)-\frac{N}{L}\rho_V(x)\right|\lesssim \frac{\ln N M}{L}\frac{N\sqrt{\kappa_I}}{L},
$$
where $\kappa_I=\dist(I, \{A,B\})$.
\end{proposition}

\begin{proposition}\label{prop:alphabetaloc}[Estimates on $\alpha$ when $K=0$]
We assume that $(\ell_1, \ell_2,\cdots, \ell_K, \ell_{K+L+1}, \ell_{K+L+2}, \cdots, \ell_N)\in \cG_{M,L,K}$, i.e. it is in the set of ``good" boundary conditions, and $I$ is as defined in \eqref{e:bulkI} or \eqref{e:edgeI}. Then 
for the left edge case, i.e. $K=0$, the local equilibrium measure satisfies
$$
 |A-\alpha|=\OO\left(\ln N\left(\frac{M}{L}\right)^{1/3}\left(\frac{M}{N}\right)^{2/3}\right),
$$
where $\alpha$ is as in Proposition \ref{prop:equilibriumstructure}.
\end{proposition}

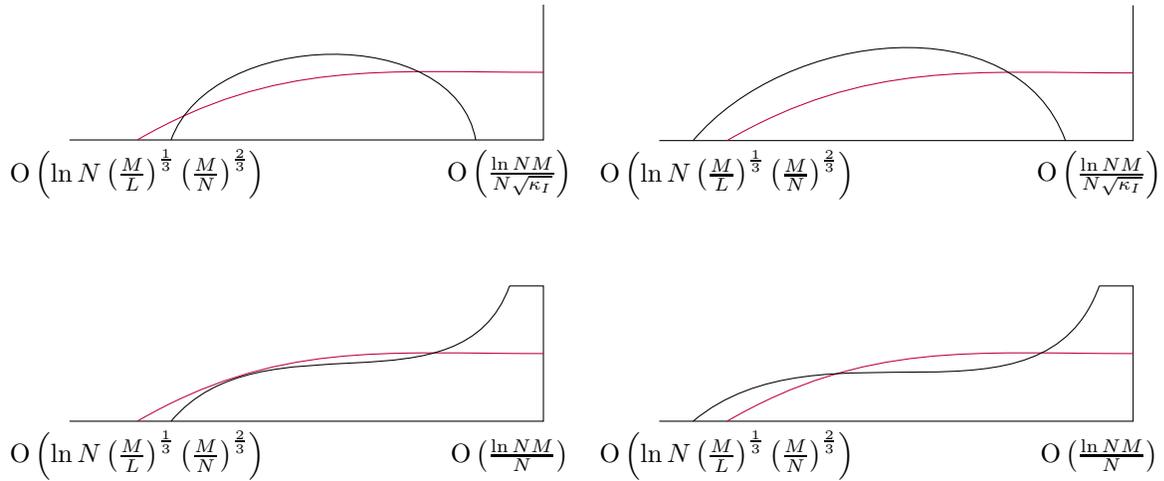
\begin{figure}[h]
\caption{equilibrium measure of localized discrete $\beta$-ensemble (left edge case, $\kappa_I\asymp(L/N)^{2/3}$).}
\vspace{1cm}
\begin{minipage}{0.4\textwidth}
\begin{tikzpicture}[scale=0.9]
\centering
\draw (0,0)--(7,0)--(7,2);
\draw[purple] (1,0) to [out=30, in=180] (7,1);
\draw (1.5,0) to [out=70,in=100] (6,0);
\node at (1,-0.5) {$\OO\left(\ln N\left(\frac{M}{L}\right)^{\frac{1}{3}}\left(\frac{M}{N}\right)^{\frac{2}{3}}\right)$};
\node at (6.5,-0.5) {$\OO\left(\frac{\ln NM}{N\sqrt{\kappa_I}}\right)$};
\end{tikzpicture}
\end{minipage}
\hspace{1cm}
\begin{minipage}{0.4\textwidth}
\begin{tikzpicture}[scale=0.9]
\centering
\draw (0,0)--(7,0)--(7,2);
\draw[purple] (1,0) to [out=30, in=180] (7,1);
\draw (0.5,0) to [out=50,in=110] (6,0);
\node at (1,-0.5) {$\OO\left(\ln N\left(\frac{M}{L}\right)^{\frac{1}{3}}\left(\frac{M}{N}\right)^{\frac{2}{3}}\right)$};
\node at (6.5,-0.5) {$\OO\left(\frac{\ln NM}{N\sqrt{\kappa_I}}\right)$};
\end{tikzpicture}
\end{minipage}

\vspace{1cm}
\begin{minipage}{0.4\textwidth}
\begin{tikzpicture}[scale=0.9]
\centering
\draw (0,0)--(7,0)--(7,2)--(6.5,2);
\draw[purple] (1,0) to [out=30, in=180] (7,1);
\draw (1.5,0) to [out=50,in=-110] (6.5,2);
\node at (1,-0.5) {$\OO\left(\ln N\left(\frac{M}{L}\right)^{\frac{1}{3}}\left(\frac{M}{N}\right)^{\frac{2}{3}}\right)$};
\node at (6.5,-0.5) {$\OO\left(\frac{\ln N M}{N}\right)$};
\end{tikzpicture}
\end{minipage}
\hspace{1cm}
\begin{minipage}{0.4\textwidth}
\begin{tikzpicture}[scale=0.9]
\centering
\draw (0,0)--(7,0)--(7,2)--(6.5,2);
\draw[purple] (1,0) to [out=30, in=180] (7,1);
\draw (0.5,0) to [out=40,in=-110] (6.5,2);
\node at (1,-0.5) {$\OO\left(\ln N\left(\frac{M}{L}\right)^{\frac{1}{3}}\left(\frac{M}{N}\right)^{\frac{2}{3}}\right)$};
\node at (6.5,-0.5) {$\OO\left(\frac{\ln N M}{N}\right)$};
\end{tikzpicture}
\end{minipage}
\end{figure}

\begin{proposition}\label{prop:edgeeq}[Estimates on the density when $K=0$]
We assume that the particle configuration $(\ell_1, \ell_2,\cdots, \ell_K, \ell_{K+L+1}, \ell_{K+L+2}, \cdots, \ell_N)\in \cG_{M,L,K}$, i.e. it is in the set of ``good" boundary conditions. We consider the left edge case, i.e. $K=0$ and $I=[a,b]$, where $a=a(N)/N$ and $b=(\ell_{L+1}-\theta)/N$. Let $\kappa_I=\dist(b, A)$, then $\kappa_I\asymp (L/N)^{2/3}$. For any $x\in I$ such that $x\geq \min\{A,\alpha\}$ and $b-x\asymp \kappa_I$, the following holds
\begin{enumerate}
\item if $\alpha\geq A$,
\begin{align}\label{e:erroraA}
\left|\rho_W(x)-\frac{N}{L}\rho_V(x)\right|\lesssim
\left\{\begin{array}{cc}
\frac{N}{L}\sqrt{\alpha-A},& 0\leq x-\alpha\leq \alpha-A,\\
\frac{N}{L}\frac{\alpha-A}{\sqrt{x-\alpha}}+\frac{\ln N M}{L}\frac{ N}{L}\sqrt{x-\alpha},& x-\alpha\geq \alpha-A,\\
\end{array}
\right.
\end{align}
\item if $\alpha \leq A$, 
\begin{align}\label{e:errorAa}
\left|\rho_W(x)-\frac{N}{L}\rho_V(x)\right|\lesssim
\left\{\begin{array}{cc}
\frac{N}{L}\ln\frac{A-\alpha}{x-\alpha}\sqrt{A-\alpha}, & \alpha-A\leq x-A\leq A-\alpha,\\
\frac{N}{L}\frac{A-\alpha}{\sqrt{x-\alpha}}+\frac{\ln N M}{L}\frac{N}{L}\sqrt{x-\alpha},& x-A\geq A-\alpha.
\end{array}
\right.
\end{align}
\end{enumerate}
\end{proposition}

\begin{proof}[Proof of Theorem \ref{t:lowerscale}]

By our assumption, with probability at least $1-e^{cM}$, the rigidity estimates  \eqref{e:LDPM2} hold on the scale $M/N$ in the spectral domain $\cD_{M,r_0}\cup \cD_*$. It follows from Proposition \ref{prop:rigiditylocalmeasure}, that with probability at least $1-e^{cM/2}$ with respect to the marginal measure, we have $(\ell_1, \ell_2, \cdots,\ell_K, \ell_{K+L+1}, \ell_{K+L+2}, \cdots, \ell_N)\in \cG_{M,L,K}$, i.e. it is in the set of ``good" boundary conditions. In the following we fix a ``good" boundary condition $(\ell_1, \ell_2, \cdots,\ell_K, \ell_{K+L+1}, \ell_{K+L+2}, \cdots, \ell_N)\in \cG_{M,L,K}$ and prove that \eqref{e:Gconcentratebulk} and \eqref{e:Gconcentrateedge} hold with probability $1-e^{cL\ln N }$ with respect to the conditioned measure $\bP_L$ (as defined in \eqref{e:condensemble}). Since by our choice of parameters, $M\ll L$, this implies Theorem \ref{t:lowerscale}. 

For any $z=E+\ri \eta\in \cD_{M,r_0+r}$ as defined in \eqref{def:domainD}, we denote 
$$
i_0=\argmin_{1\leq i\leq N} |E-\gamma_i|,
$$
where $\gamma_i$ are classical particle locations of the equilibrium measure $\mu$ as in \eqref{e:defloc}. Then $i_0\leq N_0+1$, where $N_0=\max\{i: \gamma_i\leq B-r_0-r\}$. There are bulk cases and edge cases,
\begin{enumerate}
\item The bulk cases: if $i_0>\lfloor3L/4\rfloor$, we take $K=i_0-\lfloor L/4\rfloor$, then $K\leq N_0$.  Let $I=[a,b]$, where $a=(\ell_{K}+\theta)/N$ and$b=(\ell_{K+L+1}-\theta)/N$, and $\kappa_I=\dist(I, \{A,B\})$, then $\kappa_I\gtrsim (L/M)^{2/3}$. We take a bump function $\chi_0:\bR\rightarrow \bR$, such that $\chi_0(x)=1$ on $[\gamma_{K+\lfloor L/5\rfloor}, \gamma_{K+\lfloor 4L/5\rfloor}]$, and $\chi_0(x)$ vanishes outside $[\gamma_{K+\lfloor L/6\rfloor}, \gamma_{K+\lfloor 5L/6\rfloor}]$. Moreover, $\chi_0'(x)=\OO(|I|^{-1})$ and $\chi_0'(x)=\OO(|I|^{-2})$.
\item The left edge cases: if $i_0\leq \lfloor3L/4\rfloor$, we take $K=0$. Let $I=[a(N)/N, b]$,  where $b=(\ell_{L+1}-\theta)/N$, and $\kappa_I=\dist(b, A)$, then $\kappa_I\asymp (L/N)^{2/3}$. We take a bump function $\chi_0:\bR\rightarrow \bR$, such that $\chi_0(x)=1$ on $[\hat a_N, \gamma_{\lfloor 4L/5\rfloor}]$ (where $\hat a_N=\min\{a(N)/N, \hat a\}$), and $\chi_0(x)$ vanishes for $x\geq \gamma_{\lfloor 5L/6\rfloor}$. Moreover, $\chi_0'(x)=\OO(\kappa_I^{-1})$ and $\chi_0'(x)=\OO(\kappa_I^{-2})$.
\end{enumerate}
We decompose the difference $G_N(z)-G_\mu(z)$,
\begin{align}\label{e:decompG}
G_N(z)-G_{\mu}(z)
=\int \frac{\chi_0(x)}{z-x}\left(\rd \mu_N(x)-\rd \mu(x)\right)
+\int \frac{1-\chi_0(x)}{z-x}\left(\rd\mu_N(x)-\rd\mu(x)\right).
\end{align}

In the following we estimate the second term in \eqref{e:decompG}. We first discuss the bulk case. Since $\chi_0(x)=1$ on $[\gamma_{K+\lfloor L/5\rfloor}, \gamma_{K+\lfloor 4L/5\rfloor}]$, we have
\begin{align}\label{e:decompRegion}
\int \frac{1-\chi_0(x)}{z-x}\left(\rd\mu_N(x)-\rd\mu(x)\right)
=\int_{x\geq \gamma_{K+\lfloor 4L/5\rfloor}}
+\int_{x\leq \gamma_{K+\lfloor L/5\rfloor}}.
\end{align}   
We estimate the first term on the righthand side of \eqref{e:decompRegion}. The estimate for the second term is analogous, which is similar to the argument for \eqref{e:inth5}. Similar to the construction in Remark \ref{r:decompint}, we take a dyadic partition of unity $\chi_1, \chi_2, \cdots, \chi_{n'}$ with scale $\tilde \eta =L/N$, such that $1=\chi_0(x)+\chi_1(x)+\chi_2(x)+\chi_3(x)+\cdots+\chi_{n'}(x)$ for $[\gamma_{K+\lfloor 4L/5\rfloor}, \hat b_N]$. $\chi_k(y)$ is supported on interval $I_k$, with $|I_k|\asymp 2^kL/N\sqrt{\kappa_I}$. It follows
$$
\left|\int_{x\geq \gamma_{K+\lfloor 4L/5\rfloor}} \frac{1-\chi_0(x)}{z-x}\left(\rd\mu_N(x)-\rd \mu(x)\right)\right|
= 
\left|\sum_{k=1}^{n'}\int \frac{\chi_k(x)}{z-x}\left(\rd\mu_N(x)-\rd\mu(x)\right)\right|.
$$
Moreover, from our choice of the index $K$, $\gamma_{K+\lfloor 4L/5\rfloor}-E\asymp L/N\sqrt{\kappa_I}$ and $\dist(I_k, E)\asymp 2^k L/N\sqrt{\kappa_I}$. Let 
$$
h_k(x)=\frac{\chi_k(y)|I_k|}{z-x}.
$$
Then we have, $h_k(x)=\OO(1)$, $h_k'(x)=\OO(|I_k|^{-1})$ and $h_k''(x)=\OO(|I_k|^{-2})$.
Moreover, either $I_k\subset [\gamma_{K+\lfloor 4L/5\rfloor}, B-r_0]$, or $I_k$ is not completely contained in $[\gamma_{K+\lfloor 4L/5\rfloor},B-r_0]$, in which case $|I_k|\gtrsim 1$. By Proposition \ref{prop:rig} and \ref{prop:rig2}, we have
\begin{eqnarray*}
\sum_{k=1}^{n'}\left|\int \frac{\chi_k(x)}{z-x}\left(\rd\mu_N(x)-\rd\mu(x)\right)\right|
&=&\sum_{k=1}^{n'}|I_k|^{-1}\left|\int h_k(x)\left(\rd\mu_N(x)-\rd \mu(x)\right)\right|\\
&\lesssim& \sum_{k=1}^{n'}|I_k|^{-1}\frac{\ln N M}{N}
\lesssim \frac{\ln N M}{N|I_1|}\lesssim \frac{\ln N M}{L}\sqrt{\kappa_I}.
\end{eqnarray*}
The same estimate holds for the second term in \eqref{e:decompRegion}, and it follows 
\begin{align}\label{e:secondtermest}
\left|\int \frac{1-\chi_0(x)}{z-x}\left(\rd\mu_N(x)-\rd\mu(x)\right)\right|\lesssim  \frac{\ln N M}{L}\sqrt{\kappa_I}.
\end{align}
For the left edge case, we have
$$
\int \frac{1-\chi_0(x)}{z-x}\left(\rd\mu_N(x)-\rd\mu(x)\right)
=\int_{x\geq \lfloor4L/5\rfloor} \frac{1-\chi_0(x)}{z-x}\left(\rd\mu_N(x)-\rd\mu(x)\right).
$$
Exactly the same argument leads to the estimate \eqref{e:secondtermest}.

For the first term in \eqref{e:decompG}, it depends only on particles in the interval $I$, we can rewrite it as
\begin{align}\label{e:local_diff}
\frac{L}{N}\int \frac{\chi_0(x)}{z-x}\left(\rd\mu_L(x)-\rho_W(x)\rd x\right)+
\int \frac{\chi_0(x)}{z-x}\left(\frac{L}{N}\rho_W(x)-\rho_V(x)\right)\rd x.
\end{align}
In the following we prove with probability at least $1-e^{-\ln N L}$, with respect to the conditioned measure $\bP_L$ (as defined in \eqref{e:condensemble}) \eqref{e:local_diff} is small. 

An estimate of the first term in \eqref{e:local_diff} follows from Proposition      \ref{prop:ldb}. We take $g(x)=\chi_0(x)/(z-x)$ in \eqref{e:integ}. Then $\|g\|_{1/2}\lesssim (\ln N)^{1/2}/\eta$, and it follows, with probability at least $1-e^{-\ln N L}$, with respect to the conditional measure \eqref{e:condensemble}, 
\begin{align}\label{e:decompG2}
\left|\frac{L}{N}\int \frac{\chi_0(x)}{z-x}\left(\rd\mu_L(x)-\rho_W(x)\rd x\right)\right|\leq \frac{C\ln N\sqrt{ L}}{N\eta}.
\end{align}

For the second term in \eqref{e:local_diff}, we first discuss the bulk cases. $\chi_0(x)$ is supported on $[\gamma_{K+\lfloor L/6\rfloor}, \gamma_{K+\lfloor 5L/6\rfloor}]$. From our choice of the index $K$, for any $x\in [\gamma_{K+\lfloor L/6\rfloor}, \gamma_{K+\lfloor 5L/6\rfloor}]$, we have $b-x\asymp x-a\asymp |I|$. Thanks to Proposition \ref{prop:bulkeq}, we have
\begin{align}\begin{split}\label{e:decompG3}
\left|\int \frac{\chi_0(x)}{z-x}\left(\frac{L}{N}\rho_W(x)-\rho_V(x)\right)\rd x\right|
\lesssim\int \frac{\chi_0(x)}{|z-x|}\frac{\ln N M}{L}\sqrt{\kappa_I}\rd x
\lesssim \frac{(\ln N)^2M}{L}\sqrt{\kappa_I}.
\end{split}\end{align}
It follows by combining \eqref{e:decompG}, \eqref{e:secondtermest}, \eqref{e:decompG2} and \eqref{e:decompG3}, with probability at least $1-e^{-\ln N L}$, with respect to the conditional measure \eqref{e:condensemble},
$$
|G_N(z)-G_\mu(z)|\leq C(\ln N)^2\left(\frac{\sqrt{L}}{N\eta}+\frac{M}{L}\sqrt{\kappa_I}\right).
$$

For the second term in \eqref{e:local_diff}, the analysis for the left edge case is more delicate.  As in Proposition \ref{prop:edgeeq}, there are two cases. If  $\alpha\leq A$, by using \eqref{e:errorAa}, we have 
\begin{align}\begin{split}\label{e:decompG4}
\left|\int \frac{\chi_0(x)}{z-x}\left(\frac{L}{N}\rho_W(x)-\rho_V(x)\right)\rd x\right|
\lesssim\int_{\alpha\leq x\leq 2A-\alpha} \frac{\chi_0(x)}{|z-x|}\left|\frac{L}{N}\rho_W(x)-\rho_V(x)\right|\rd x\\
+\int_{x\geq 2A-\alpha} \frac{\chi_0(x)}{|z-x|}\frac{A-\alpha}{\sqrt{x-\alpha}}\rd x
+\frac{\ln N M}{L}\int_{x\geq 2A-\alpha} \frac{\chi_0(x)}{|z-x|}\sqrt{x-\alpha}\rd x.
\end{split}\end{align}
Let $\kappa_{E}=\dist(E, \{A,B\})$.
For the first term in \eqref{e:decompG4}, we have by \eqref{e:errorAa}, 
\begin{align}\begin{split}\label{e:1decompG4}
&\int_{x\leq 2A-\alpha} \frac{\chi_0(x)}{|z-x|}\left|\frac{L}{N}\rho_W(x)-\rho_V(x)\right|\rd x
\lesssim \sqrt{\alpha-A} \int_{\alpha \leq x\leq 2A-\alpha}\frac{\chi_0(x)}{|z-x|}\ln \frac{A-\alpha}{x-\alpha}\rd x.
\end{split}\end{align}
If $\eta+\kappa_E\geq 2(A-\alpha)$, then 
\begin{align}\label{e:side1}
\int_{\alpha \leq x\leq 2A-\alpha}\frac{\chi_0(x)}{|z-x|}\ln \frac{A-\alpha}{x-\alpha}\rd x
\lesssim \int_{\alpha \leq x\leq 2A-\alpha}\frac{1}{\eta+\kappa_E}\ln \frac{A-\alpha}{x-\alpha}\rd x\lesssim \frac{\ln N(A-\alpha)}{\eta+\kappa_E}.
\end{align}
If $\eta+\kappa_E\leq 2(A-\alpha)$, we trivially bound
\begin{align}\label{e:side2}
\int_{\alpha \leq x\leq 2A-\alpha}\frac{\chi_0(x)}{|z-x|}\ln \frac{A-\alpha}{x-\alpha}\rd x
\lesssim \ln N.
\end{align}
By combining \eqref{e:1decompG4}, \eqref{e:side1} and \eqref{e:side2}, we have 
\begin{align}\begin{split}\label{e:1decompG4.5}
&\int_{x\leq 2A-\alpha} \frac{\chi_0(x)}{|z-x|}\left|\frac{L}{N}\rho_W(x)-\rho_V(x)\right|\rd x\lesssim \ln N \left(\frac{(A-\alpha)^{3/2}}{\eta+\kappa_E}\wedge (A-\alpha)^{1/2}\right).
\end{split}\end{align}
For the second term in \eqref{e:decompG4}, by a similar argument, we have
\begin{align}\label{e:2decompG4}
\int_{x\geq 2A-\alpha} \frac{\chi_0(x)}{|z-x|}\frac{A-\alpha}{\sqrt{x-\alpha}}\rd x\lesssim \ln N\left(\frac{A-\alpha}{\sqrt{\eta+\kappa_E}}\wedge (A-\alpha)^{1/2}\right).
\end{align}
For the last term in \eqref{e:decompG4}, we have
\begin{align}\label{e:3decompG4}
\frac{\ln N M}{L}\int_{x\geq 2A-\alpha} \frac{\chi_0(x)}{|z-x|}\sqrt{x-\alpha}\rd x\lesssim \frac{(\ln N)^2M}{L}
\left(\frac{L}{N}\right)^{1/3},
\end{align}
where we used that $x-\alpha\leq b-\alpha \leq (L/N)^{2/3}$.
By combining \eqref{e:decompG4}, \eqref{e:1decompG4.5}, \eqref{e:2decompG4} and \eqref{e:3decompG4}, we have
\begin{align}\begin{split}\label{e:decompG5}
\left|\int \frac{\chi_0(x)}{z-x}\left(\frac{L}{N}\rho_W(x)-\rho_V(x)\right)\rd x\right|
\leq C\left(\ln N\left(\frac{A-\alpha}{\sqrt{\kappa_E+\eta}}\wedge (A-\alpha)^{1/2}\right)+\frac{(\ln N)^2M}{L}\left(\frac{L}{N}\right)^{1/3}\right)\\
\leq C(\ln N)^2\left(\left(\frac{M}{L}\right)^{1/2}\left(\frac{L}{N}\right)^{1/3}\wedge\frac{M}{L}\left(\frac{L}{N}\right)^{2/3}\frac{1}{\sqrt{\kappa_{E}+\eta}}+\frac{M}{L}\left(\frac{L}{N}\right)^{1/3}\right).
\end{split}\end{align}
A similar argument implies that \eqref{e:decompG5} holds for the case $\alpha\geq A$.
It follows by combining \eqref{e:decompG}, \eqref{e:secondtermest}, \eqref{e:decompG2} and \eqref{e:decompG5}, with probability at least $1-e^{-\ln N L}$, with respect to the conditional measure \eqref{e:condensemble},
\begin{align}\nonumber
|G_N(z)-G_\mu(z)|\leq C(\ln N)^2\left(\frac{\sqrt{L}}{N\eta}+\left(\frac{M}{L}\right)^{1/2}\left(\frac{L}{N}\right)^{1/3}\wedge\frac{M}{L}\left(\frac{L}{N}\right)^{2/3}\frac{1}{\sqrt{\kappa_E+\eta}}+\frac{M}{L}\left(\frac{L}{N}\right)^{1/3}\right).
\end{align}
This finishes the proof of Theorem \ref{t:lowerscale}.
\end{proof}

\section{Edge Universality}

In this section we prove the edge universality of the discrete $\beta$-ensembles for any $\beta\geq 1$, i.e. the distributions of extreme particles are the same as those of continuous $\beta$-ensembles, governed by Tracy-Widom $\beta$ distributions. Heuristically, the edge universality states that extreme particles after scaling by $N^{2/3}$ converge in law to Tracy-Widom $\beta$ distribution. We can view discrete $\beta$-ensembles as ``discretized" continuous $\beta$-ensembles on the scale $1/N$. Since the discretization scale is much smaller than the scale of the fluctuations of extreme particles, the discretization does not affect the distribution of extreme particles, and the edge universality of the discrete $\beta$-ensembles follows. To do this, we study the conditioned discrete $\beta$-ensemble \eqref{e:condensemble2} of the first few particles, and compare it with its continuous analogue \eqref{e:contbeta}. We show that these two laws are the same up to negligible error. It is proven in \cite{MR3253704} that, under mild conditions, the distribution  of extreme particles of \eqref{e:contbeta} converges to Tracy-Widom $\beta$ distribution. 
The edge universality of the discrete $\beta$-ensembles follows.

We fix small constant $r>0$, small exponents $\fa, \fb>0$, and let $L=N^{\fb}$. 
For any $(\ell_{L+1}, \ell_{L+2},\cdots \ell_{N})\in \bW_N^{\theta}$, we define the conditioned discrete $\beta$-ensemble,
\begin{align}\label{e:condensemble2}
\bP^{\loc}_L(\ell_1, \ell_2, \cdots, \ell_L)=\frac{1}{Z_{L}^{\loc}}\prod_{1\leq i<j\leq L}\frac{\Gamma(\ell_j-\ell_i+1)\Gamma(\ell_j-\ell_i+\theta)}{\Gamma(\ell_j-\ell_i)\Gamma(\ell_j-\ell_i+1-\theta)}
\prod_{i=1}^{L} \hat w(\ell_i; L),
\end{align}
where
\begin{align*}
\hat w(Nu; L)=e^{-NV_N(u)}e^{-\beta\Theta(N^{2/3-\fa}(u-A))}\prod_{j>L}\frac{\Gamma(\ell_j-Nu+1)\Gamma(\ell_j-Nu+\theta)}{\Gamma(\ell_j-Nu)\Gamma(\ell_j-Nu+1-\theta)} ,
\end{align*}
where \eqref{e:condensemble2} is slightly different from \eqref{e:condensemble}, since we introduced an additional quadratic potential $\Theta(u)\deq (u+1)^2{\bf{1}}_{u\leq -1}$, or $\Theta(u)\deq 2(u+1)^2{\bf{1}}_{u\leq -1}$ to prevent the particles $\ell_i$ from deviating far in the left direction.

We define the continuous version of $\bP^{\loc}_L(\ell_1, \ell_2, \cdots, \ell_L)$, let $\beta=2\theta$,
\begin{align}\label{e:contbeta}
\rd\bP^{\cont}_L(\rd x_1, \rd x_2, \cdots, \rd x_L)=\frac{1}{Z_{L}^{\cont}}\prod_{1\leq i<j\leq L}|x_j-x_i|^{\beta}
\prod_{i=1}^{L} \hat w(x_i; L)\rd x_i,
\end{align}
where $\bP_L^{\cont}$ is a measure on $x_1<x_2<\cdots< x_L\leq \ell_{L+1}-\theta$.

\begin{definition}\label{def:good2}
We define the set of \emph{``good" boundary conditions}, $\cR_L^*=\cR_L^*(\fa,r)$, as particle configurations
\begin{align*}
(\ell_{L+1}, \ell_{L+2},\cdots \ell_{N})\in \bW_N^{\theta},
\end{align*}
 such that with probability $1-\exp(-c(\ln N)^2)$ with respect to the conditioned discrete $\beta$-ensemble $\bP_L^{\loc}$:
 \begin{enumerate}
 \item The Stieltjes transform
satisfies 
\begin{align*}
|G_N(z)-G_\mu(z)|\leq \frac{(\ln N)^2}{N\eta},
\end{align*}
uniformly for any $z\in \cD_r\cup \cD_*$.
\item The extreme eigenvalue satisfies $\ell_1/N\geq A-N^{-2/3+\fa}$.
\end{enumerate}
\end{definition}
If $(\ell_{L+1}, \ell_{L+2},\cdots \ell_{N})$ is in the set of ``good" boundary conditions, then there exists some $r'>0$, such that 
\begin{align*}
|\ell_k/N-\gamma_k|\leq N^{-2/3+\fa}k^{-1/3},\quad k\in \qq{L+1, r'N},
\end{align*}
and with high probability with respect to $\bP_L^{\loc}$, 
\begin{align*}
|\ell_k/N-\gamma_k|\leq N^{-2/3+\fa}k^{-1/3},\quad k\in \qq{1, L}.
\end{align*}

It follows from Theorem \ref{t:rigidity} and Corollary \ref{c:initialstep}, that with probability at least $1-\exp(c(\ln N)^2)$ with respect to $\bP_N$, we have $(\ell_{L+1}, \ell_{L+2},\cdots \ell_{N})\in \cR_L^*(\fa,r)$.

Thanks to the additional quadratic potential $\Theta$, similar to the proof of Proposition \ref{prop:ldb}, we have the following estimate on the extreme particles of $\bP_L^{\cont}$,
\begin{proposition}\label{p:LDPextreme}
For any ``good" boundary condition $(\ell_{L+1},\ell_{L+2},\cdots, \ell_{N})\in \cR_L^*(\fa,r)$, there exists universal constant $C>0$ such that
\begin{align}\label{e:edgeldp}
&\bP_L^{\cont}\left(x_1/N\leq A-C(L\ln N)^{1/2}N^{-2/3+\fa}\right)\leq e^{-L\ln N}.
\end{align}
\end{proposition}
\begin{proof}
We recall that $W$ is defined in \eqref{e:expW}. We denote  by $\rho^{\cont}_W$ the unconstrained equilibrium measure of the potential $W$, which is characterized by
\begin{align}\label{e:expW24}
W(x)-2\theta \int_{a}^b \ln|x-y|\rho_W^{\cont}(y)\rd y-f_W=F_W(x)=\left\{
\begin{array}{lll}
=0,  &\text{on }  \rho_W^{\cont}(x)>0,\\
\geq 0, &\text{on }  \rho_W^{\cont}(x)=0.
\end{array}
\right.
\end{align}
The same argument as  in Proposition \eqref{prop:ldb} gives,
\begin{align*}
\bP_L\left(\sum_{i=1}^{L}\Theta(N^{2/3-\fa}(\ell_i/N-A))+L^2\cD^2 [\rho_W^{\cont}-\rho_L]\geq \gamma^2 L^2\right)\leq \exp\left(C\ln N L-\gamma^2 L^2\right),
\end{align*}
where $\rho_L$ is the convolution of the empirical particle density $1/L\sum_{i=1}^L \delta_{\ell_i/N}$ with the uniform measure on the interval $[-N^{-p}/2, N^{-p}/2]$ for some $p>2$, and the energy functional $\cD$ is defined in \eqref{def:energyfunctional}. In particular, we have
\begin{align*}
\bP_L\left(\Theta(N^{2/3-\fa}(\ell_1/N-A))\geq CL \ln N \right)\leq \exp\left(-\ln N L\right),
\end{align*}
which leads to \eqref{e:edgeldp} by rearranging.
\end{proof}

Following \cite{MR3253704}, we further define the following modification of $\cR_L^*$. %
We define $\cR_L^{\#}$ which adds the condition of a level repulsion near the boundary, 
\begin{align*}
\cR_L^{\#}=\cR_L^{\#}(\fa,r)\deq \{(\ell_{L+1}, \cdots, \ell_{N})\in \cR_L^*(\fa/3,r): |\ell_{L+2}-\ell_{L+1}|\geq N^{1/3-\fa}L^{-1/3}\}.
\end{align*}

The proof of edge universality of the discrete $\beta$-ensembles for any $\beta\geq 1$ follows from a direct comparison between $\bP_L^{\loc}$ and $\bP_L^{\cont}$. The following weak level repulsion estimates for $\bP_L^{\loc}$ and $\bP_L^{\cont}$ follows the same arguments as those in \cite[(D.1), (D.2)]{MR3253704}. We postpone their proofs to Appendix \ref{a:levelrep}.

\begin{proposition}\label{p:levelrep}
Let $\beta>0$. There are constants $C$ such that for $(\ell_{L+1}, \cdots, \ell_{N})\in \cR_L^*(\fa,r)$, $k\in \qq{1, L}$ and any $L^{4/3}N^{-1/3}\ll s\ll \min\{L^{-2},N^{-\fa}\}$, we have
\begin{align}
\label{e:dis}&\bP_L^{\loc}\left(|\ell_{k+1}-\ell_{k}|\leq sL^{-1/3}N^{1/3}\right)\leq CsL^2,\\
\label{e:cont}&\bP_L^{\cont}\left(|x_{k+1}-x_{k}|\leq sL^{-1/3}N^{1/3}\right)\leq CsL^2.
\end{align}
Moreover, thanks to the rigidity estimates for the local measure $\bP_L^{\loc}$, we have better estimate, for $k\in \qq{1,L}$ and any $L^{4/3}N^{-1/3}\ll s\ll N^{-\fa}$ 
\label{e:dis2}\begin{align}
\bP_L^{\loc}\left(|\ell_{k+1}-\ell_{k}|\leq sk^{-1/3}N^{1/3}\right)\leq C\left(sN^{\fa}+\exp(-c(\ln N)^2)\right).
\end{align}
\end{proposition}

In the following we fix a small number $s$ which will be chosen later. We define $\cG^{\loc}$ to be the set of  particle configurations $(\ell_1, \ell_2, \cdots, \ell_L)\in \bW_N^\theta$, such that  
\begin{align*}
\ell_1/N\geq A-C(L\ln N)^{1/2}N^{-2/3+\fa}, \quad |\ell_{k+1}-\ell_{k}|\geq sL^{-1/3}N^{1/3}, \quad k\in \qq{1, L}.
\end{align*}

We define the continuous version of $\cG^{\loc}$ as 
\begin{align*}
\cG^{\cont}=\{(x_1, x_2,\cdots, x_L): (\lceil x_1\rceil_\theta, \lceil x_2\rceil_\theta,\cdots, \lceil x_L\rceil_\theta)\in \cG^{\loc}\},
\end{align*}
where $\lceil x_k\rceil_\theta=\inf\{\ell\geq x_k: \ell\in a(N)+ k\theta+\bZ \}$. We remark that $\lceil x_k\rceil_\theta-1<x_k\leq \lceil x_k\rceil_\theta$.
Thanks to Propositions \ref{p:LDPextreme}, \ref{p:levelrep} and a union bound, we have 
\begin{align}\label{e:probG}
\bP^{\loc}_L(\cG^{\loc})\geq 1-CsL^3,\quad \bP^{\cont}_L(\cG^{\cont})\geq 1-CsL^3.
\end{align}

\begin{proposition}\label{p:same}
For any ``good" boundary condition $(\ell_{L+1},\ell_{L+2},\cdots, \ell_{N})\in \cR_L^*(\fa,r)$, and particle configurations $(x_1,x_2,\cdots,x_L)\in \cG^{\cont}$, we have
\begin{align}\label{e:same}
Z^{\loc}_L\bP_L^{\loc}(\ell_1, \ell_2,\cdots, \ell_L)=\exp\left\{O\left(\frac{L^{4/3}N^\fa}{sN^{1/3}}\right)\right\}Z^{\cont}_L\bP_L^{\cont}(x_1, x_2,\cdots, x_L),
\end{align}
where $\ell_k=\lceil x_k\rceil_\theta=\inf\{\ell\geq x_k: \ell\in a(N)+ k\theta+\bZ \}$, for $k\in \qq{1, L}$. Moreover
\begin{align}\label{e:same2}
\bP_L^{\loc}(\ell_1, \ell_2,\cdots, \ell_L)=\left(1+\OO\left(sL^3+\frac{L^{4/3}N^\fa}{sN^{1/3}}\right)\right)\bP_L^{\cont}(x_1, x_2,\cdots, x_L).
\end{align}
In the above estimates, we assume $L^{4/3}N^{\fa-1/3}\ll s\ll L^{-3}$.
\end{proposition}
\begin{proof}
For any $1\leq i<j\leq L$, since $\ell_j-\ell_i\geq sL^{-1/3}N^{1/3}(j-i)\gg 1$, we can estimate the interaction terms in \eqref{e:condensemble2} 
\begin{align*}
\frac{\Gamma(\ell_j-\ell_i+1)\Gamma(\ell_j-\ell_i+\theta)}{\Gamma(\ell_j-\ell_i)\Gamma(\ell_j-\ell_i+1-\theta)}
=|\ell_j-\ell_i|^{2\theta}e^{\OO\left(1/(\ell_j-\ell_i)\right)}
=|x_j-x_i|^{2\theta}e^{\OO\left(1/(\ell_j-\ell_i)\right)}.
\end{align*}
Therefore we have 
\begin{align}\begin{split}\label{e:inter}
\prod_{1\leq i<j\leq L}\frac{\Gamma(\ell_j-\ell_i+1)\Gamma(\ell_j-\ell_i+\theta)}{\Gamma(\ell_j-\ell_i)\Gamma(\ell_j-\ell_i+1-\theta)}
&=e^{\sum_{1\leq i<j\leq L}\OO\left(1/(\ell_j-\ell_i)\right)}\prod_{1\leq i<j\leq L}|x_j-x_i|^{2\theta}\\
&=\exp\left\{\OO\left(\frac{L^{4/3}\ln N}{sN^{1/3}}\right)\right\}\prod_{1\leq i<j\leq L}|x_j-x_i|^{2\theta}.
\end{split}\end{align}

In the following we estimate, $\hat w(\ell_k; L)/\hat w(x_k; L)$. Notice that we have the following asymptotics,
\begin{align*}
\del_z \ln \Gamma(z)=\ln z-\frac{1}{2z}+\OO\left(\frac{1}{z^2}\right).
\end{align*}
Moreover, the derivative of the potential $V_N$ can be written in terms of the equilibrium measure,
\begin{align}\label{e:ders4}
V_N'(u)=V'(u)+\OO(N^{-1/3})=2\theta P.V.\int_A^B \frac{\rho_V(y)}{u-y}\rd x +1_{u\not\in [A,B]}\OO\left(\sqrt{(u-A)(u-B)}\right) +\OO(N^{-1/3}).
\end{align}
where we used assumption \ref{a:moreVN}.
Therefore, for any $x\in [N(A-C(L\ln N)^{1/2}N^{-2/3+\fa}), \ell_{L+1}-sL^{-1/3}N^{1/3}]$, we have
\begin{align}\begin{split}\label{e:ders5}
&\del_x \ln \hat w(x;L)=-V_N'\left(\frac{x}{N}\right)+\OO\left(\frac{(L\ln N)^{1/2}}{N^{1/3+\fa}}\right)
+2\theta\sum_{j\geq L+1}\frac{1}{x-\ell_j}+\OO\left(\frac{1}{(x-\ell_j)^2}\right).
\end{split}\end{align}
We deduce  from \cite[Lemma C.1]{MR3253704} that 
\begin{align}\label{e:ders6}
-2\theta P.V.\int_A^B \frac{\rho_V(y)}{x/N-y}
+2\theta\sum_{j\geq L+CN^\fa}\frac{1}{x-\ell_j}
=\OO\left(\frac{L^{1/3}N^{\fa}}{N^{1/3}}\right),
\end{align}
Moreover, since $x\leq \ell_{L+1}-sL^{-1/3}N^{1/3}$, we have
\begin{align}\label{e:ders7}
2\theta\sum_{j= L+1}^{L+CN^{\fa}}\frac{1}{x-\ell_j}+2\theta\sum_{j\geq L+1}\OO\left(\frac{1}{(x-\ell_j)^2}\right)
=\OO\left( \frac{L^{1/3}\ln N}{sN^{1/3}}\right).
\end{align}
We get the following estimate by combining \eqref{e:ders4}, \eqref{e:ders5}, \eqref{e:ders6} and \eqref{e:ders7},
\begin{align*}
\del_x \ln \hat w(x;L)=\OO\left( \frac{L^{1/3}N^{\fa}}{sN^{1/3}}\right).
\end{align*}
Therefore, we deduce that 
\begin{align}\label{e:hatw2}
\frac{\hat w(\ell_k; L)}{\hat w(x_k; L)}=
e^{C\sup_{x_k\leq x\leq \ell_k}|\del_x\ln \hat w(x;L)||\ell_k-x_k|}=\exp\left\{\OO\left(\frac{L^{1/3}N^{\fa}}{sN^{1/3}}\right)\right\},
\end{align}
where we used that $|\ell_k-x_k|=\OO(1)$.

It follows by combining \eqref{e:inter} and \eqref{e:hatw2} together,
\begin{align*}
Z^{\loc}_L\bP_L^{\loc}(\ell_1, \ell_2,\cdots, \ell_L)=\exp\left\{\OO\left(\frac{L^{4/3}\ln N}{sN^{1/3}}\right)+\OO\left(\frac{L^{4/3}N^{\fa}}{sN^{1/3}}\right)\right\}Z^{\cont}_L\bP_L^{\cont}(x_1, x_2,\cdots, x_L).
\end{align*}
and \eqref{e:same} follows. 

For \eqref{e:same2}, we need to show $Z_L^{\loc}\approx Z_L^{\cont}$. Thanks to \eqref{e:probG}, it follows
\begin{align*}\begin{split}
&(1+\OO(sL^3))Z_L^{\loc}=Z_L^{\loc}\bP_L^{\loc}(\cG^{\loc})
=\sum_{(\ell_1, \ell_2,\cdots, \ell_L)\in \cG^{\loc}} Z_L^{\loc}\bP_L^{\loc}(\ell_1, \ell_2,\cdots, \ell_L)\\
=&\sum_{(\ell_1, \ell_2,\cdots, \ell_L)\in \cG^{\loc}} \exp\left\{\OO\left(\frac{L^{4/3}N^\fa}{sN^{1/3}}\right)\right\}
Z^{\cont}_L\bP_L^{\cont}(\{x_1, x_2,\cdots, x_L: \lceil x_k\rceil_\theta=\ell_k, \quad k\in \qq{1,L}\})\\
=&\exp\left\{\OO\left(\frac{L^{4/3}N^\fa}{sN^{1/3}}\right)\right\}Z^{\cont}_L\bP_{L}^{\cont}(\cG^{\cont})=(1+\OO(sL^3))\exp\left\{\OO\left(\frac{L^{4/3}N^\fa}{sN^{1/3}}\right)\right\}Z_L^{\loc}.
\end{split}\end{align*}
By rearranging the above expression, we get
\begin{align}\label{e:ZL}
Z_L^{\loc}=\left(1+\OO\left(sL^3+\frac{L^{4/3}N^\fa}{sN^{1/3}}\right)\right)Z_L^{\cont},
\end{align}
and \eqref{e:same2} follows from \eqref{e:same} and \eqref{e:ZL}.

\end{proof}

The edge universality of discrete $\beta$-ensemble follows easily from Proposition \ref{p:same}.

\begin{proof}[Proof of Theorem \ref{t:edgeuniversality}]
Thanks to Theorems \ref{t:rigidity} and \eqref{e:dis2} in Proposition \ref{p:levelrep}, the event $\cR_L^*(\fa, r) \cR_L^{\#}(\fa, r)$ holds with probability $1-\OO(N^{-2\fa/3})$. \eqref{e:edgeuniv} follows from the following statement: there exists some $\chi>0$, for any boundary conditions $(\ell_{L+1}, \ell_{L+2}, \cdots, \ell_{N})\in \cR^{*}_L(\fa, r)\cap \cR^{\#}_L(\fa, r)$, it holds,
\begin{align}\label{e:edgeuniv2}
\left|\bE_{\bP^{\loc}_L}\left[O\left(\left(N^{2/3}k^{1/3}s_A^{2/3}(\ell_k/N-\gamma_k)\right)_{k\in \Lambda}\right)\right]-\bE_{G\beta E}\left[O\left(\left(N^{2/3}k^{1/3}(x_k/N-\tilde \gamma_k)\right)_{k\in \Lambda}\right)\right]\right|=\OO(N^{-\chi}).
\end{align}
In the following we prove \eqref{e:edgeuniv2}. Notice that 
\begin{align}\begin{split}\label{e:same3}
&\bE_{\bP^{\cont}_L}\left[O\left(\left(N^{2/3}k^{1/3}s_A^{2/3}(x_k/N-\gamma_k)\right)_{k\in \Lambda}\right)\right]\\
=&\bE_{\bP^{\cont}_L}\left[{\bf 1}_{\{\cG^{\cont}\}}O\left(\left(N^{2/3}k^{1/3}s_A^{2/3}(\lceil x_k\rceil_{\theta}/N-\gamma_k)\right)_{k\in \Lambda}\right)\right]+\OO\left((L/N)^{1/3}+sL^3\right)\\
=&\bE_{\bP^{\loc}_L}\left[{\bf 1}_{\{\cG^{\loc}\}}O\left(\left(N^{2/3}k^{1/3}s_A^{2/3}(\ell_k/N-\gamma_k)\right)_{k\in \Lambda}\right)\right]+\OO\left(sL^3+\frac{L^{4/3}N^\fa}{sN^{1/3}}\right)\\
=&\bE_{\bP^{\loc}_L}\left[O\left(\left(N^{2/3}k^{1/3}s_A^{2/3}(\ell_k/N-\gamma_k)\right)_{k\in \Lambda}\right)\right]+\OO\left(sL^3+\frac{L^{4/3}N^\fa}{sN^{1/3}}\right),
\end{split}\end{align}
where we used \eqref{e:probG} and $|x_k-\lceil x_k \rceil_\theta|<1$ for the second and fourth lines, and \eqref{e:same2} for the third line. Recall that $L=N^{\fb}$ and $\fb<1/13$. We can take $\chi= 1/6-\fa/2-13\fb/6$, $s=N^{-1/6+\fa/2-5\fb/6}$, and $\fa$ small enough such that $\chi>0$. Then the error in \eqref{e:same3} is bounded by $\OO(N^{-\chi})$.

In particular  \eqref{e:same3} implies that
\begin{align*}
\bP_L^{\cont}(x_1\geq \gamma_1-2N^{-2/3+\fa})\geq 1/2,\quad |\bE_{\bP_L^{\cont}} [x_k-\gamma_k]|\leq 2N^{-2/3+\fa}k^{-1/3},\quad k\in \qq{1, L}.
\end{align*}
Therefore the boundary condition $(\ell_{L+1}, \ell_{L+2}, \cdots, \ell_{N})$ satisfies the assumptions in \cite[Theorem 3.3]{MR3253704}. As a result, there exits a small $\chi>0$,
\begin{align}\label{e:edgeuniv3}
\left|\bE_{\bP^{\cont}_L}\left[O\left(\left(N^{2/3}k^{1/3}s_A^{2/3}(x_k/N-\gamma_k)\right)_{k\in \Lambda}\right)\right]
-\bE_{G\beta E}\left[O\left(\left(N^{2/3}k^{1/3}(x_k/N-\tilde \gamma_k)\right)_{k\in \Lambda}\right)\right]\right|=
\OO(N^{-\chi}).
\end{align}
\eqref{e:edgeuniv2} follows from combining \eqref{e:same3} and \eqref{e:edgeuniv3}.

\end{proof}

\section{Multi-cut Case}\label{s:multi-cut}
In this section, we indicate how the arguments used so far in the article can be carried to the multi-cut stochastic systems with fixed filling fractions, which is the full model studied in \cite{Borodin2016}. We recall the definition of the multi-cut model below.

 We fix an integer
$k>0$, whose meaning is the number of segments in the support of the measure ($k=1$ corresponds to the one-cut case). For each
$N$ we take $k$ integers $n_1(N),\dots,n_k(N)$, such that $\sum_{i=1}^k n_i(N)=N$ and $k$
disjoint intervals $[a_1(N),b_1(N)]$, \dots, $[a_k(N),b_k(N)]$ of the real line ordered from left
to right.

 We assume that $b_i(N)+\theta\leq a_{i+1}(N)$ for $i\in \qq{1,k-1}$. The numbers $a_i(N)$,
$b_i(N)$ must also satisfy the conditions
 $b_i(N)-\theta n_i(N)-a_i(N)\in\mathbb Z_{> 0}$.
Further, the number $n_i(N)$ counts the number of the particles in the $i$-th interval; to make this
statement precise we define the sets of indices $I_j\subset\qq{1, N}$ for  $j\in\qq{1,k}$, via
\begin{align*}
 I_j=\qq{1+ \sum_{m=1}^{j-1} n_m(N) , \sum_{m=1}^j n_m(N)}.
\end{align*}
We also set $I_j^+$ and $I_j^-$ to be the maximal and minimal elements of $I_j$,
respectively.

\begin{definition} \label{Def_state_space}The state space $\bW_N^{\theta}$ consists of $N$--tuples $\ell_1<\ell_2<\dots<\ell_N$
such that for each $j=1,\dots,k$:
\begin{enumerate}
 \item If $i=I^-_j$, then $\ell_i-a_i(N)\in \mathbb Z_{> 0}$.
 \item If $i=I^+_j$, then $b_i(N)-\ell_i\in\mathbb Z_{> 0}$.
 \item If $i\in I_j$, but $i\ne I^+_j$, then $\ell_{i+1}-\ell_{i}-\theta\in\bZ_{\geq 0}$.
\end{enumerate}
\end{definition}

We also take a weight function $w(x;N)>0$ in the interior of the intervals $(a_i(N),b_i(N))$, and $w(a_i(N); N)=w(b_i(N); N)=0$ for $i\in \qq{1, k}$. The multi-cut discrete $\beta$-ensemble is given by the probability measure $\bP_N$ on $\bW_N^{\theta}$
\begin{equation}\label{eq_distribution_form}
 \bP_N(\ell_1,\dots,\ell_N)= \frac{1}{Z_N} \prod_{1\leq i<j \leq N}
 \frac{\Gamma(\ell_j-\ell_i+1)\Gamma(\ell_j-\ell_i+\theta)}{\Gamma(\ell_j-\ell_i)\Gamma(\ell_j-\ell_i+1-\theta)}
 \prod_{i=1}^N w(\ell_i;N).
\end{equation}

For the study of the multi-cut model, the Assumptions \ref{a:VN}, \ref{a:wx} and \ref{a:Hz} and \ref{a:moreVN} need to be replaced by their counterparts from \cite[Section 3.2]{Borodin2016}, as listed below

\begin{assumption}\label{a:mVN}
We require that as $N\rightarrow \infty$, for each $i\in\qq{1, k}$,
\begin{align*}
a_i(N)=N\hat a_i+\OO(\ln (N)),\quad b_i(N)=N\hat b_i+\OO(\ln(N)), \quad w(a_i(N); N)=w(b_i(N); N)=0,
\end{align*}
and 
\begin{align*}\nonumber
\hat a_1<\hat b_1<\hat a_2<\cdots<\hat a_k<\hat b_k.
\end{align*}
We require the weight $w(x;N)$ in the intervals $[a_i(N)+1, b_i(N)-1]$ for $i\in \qq{1,k}$, has the form
\begin{align*}\nonumber
w(x; N)=\exp\left(-NV_N\left(\frac{x}{N}\right)\right),
\end{align*}
for a function $V_N$ that is continuous in the intervals $[a_i(N)+1,b_i(N)-1]$, and such that 
\begin{align*}\nonumber
V_N(u)=V(u)+\OO\left(\frac{\ln(N)}{N}\right), 
\end{align*}
uniformly over $u\in [(a_i(N)+1)/N, (b_i(N)-1)/N]$. The function $V(u)$ is differentiable and the following bound holds for a constant $C>0$,
\begin{align*}
|V'(u)|\leq C\left(1+\sum_{i=1}^k\left(|\ln(u-\hat a_i)|+|\ln(u-\hat b_i)|\right)\right).
\end{align*}
\end{assumption}

\begin{assumption}\label{a:mfilling}
There exits a constant $\fd>0$ such that for any $i\in \qq{1, k}$ and $N$ large enough, the filling fractions satisfy
\begin{align*}\nonumber
\fd\leq n_i(N)/N\leq \theta^{-1}(\hat b_i-\hat a_i)-\fd.
\end{align*}
\end{assumption}

As in the one-cut case, we denote the empirical particle density as $\mu_N=N^{-1}\sum_{i=1}^N \delta_{\ell_i/N}$, and let $\mu(\rd x)=\rho_V(x)\rd x$ be the constrained equilibrium measure of the discrete $\beta$-ensemble $\bP_N$, which is given by the minimizer of the following functional 
\begin{align}\label{e:defIV}
-\theta \int_{x\neq y}\ln |x-y|\rho(x)\rho(y)\rd x\rd y+\int V(x)\rho(x)\rd x,
\end{align}
over all the densities $\rho(x)$, supported on $\cup_{i=1}^k[\hat a_i, \hat b_i]$, with $0\leq \rho(x)\leq \theta^{-1}$, and 
\begin{align}\nonumber
\int_{\hat a_i}^{\hat b_i}\rho(x)\rd x=\hat n_i, \quad i\in \qq{1,k}.
\end{align}
We remark that $\rho_V$ depends on $\hat n_i$, and the latter depend on $N$, so $\rho_V$ depends on $N$. However, we will hide this dependence from our notations.
We denote by $F_V$ the (non-centered) effective potential given by
 \begin{align}\nonumber
F_V(x)\deq -2\theta\int_{-\infty}^\infty \ln|x-y|\rho_V(y)\rd y+V(x),
\end{align}

On each single cut $[\hat a_i, \hat b_i]$, there exists a constant $f_{i}$ so that $F_V$ satisfies:
\begin{enumerate}
\item $F_V (x)-f_i  \geq 0$, for all $x$ in voids in $[\hat a_i, \hat b_i]$, i.e. maximal closed connected  intervals where $\rho_V(x)=0$;
\item $F_V (x)-f_i \leq  0$, for all $x$ in saturated regions in $[\hat a_i, \hat b_i]$, i.e. maximal closed connected intervals where $\rho_V(x)=\theta^{-1}$;
\item $F_V (x) -f_i = 0$, for all $x$ in bands in $[\hat a_i, \hat b_i]$, i.e. maximal open connected intervals where $0<\rho_V(x)<\theta^{-1}$.
\end{enumerate}

We denote the classical particle locations $\gamma_1, \gamma_2, \cdots, \gamma_N$ corresponding to the constrained equilibrium measure $\mu$, as
\begin{align}\label{e:mdefloc}
\frac{i-1/2}{N}=\int_{-\infty}^{\gamma_i}\rho_V(x)\rd x, \quad i\in \qq{1, N}.
\end{align}

\begin{assumption}\label{a:mwx}
For $i\in\qq{1,k}$, there exist open sets $\cM_i$, which contains the interval $[\hat a_i, \hat b_i]$, such that the functions $\psi^{\pm}_N(x)$ defined through
\begin{align}\nonumber
\frac{w(x;N)}{w(x-1;N)}=\frac{\psi_N^+(x)}{\psi_N^-(x)},
\end{align}
can be chosen in such a way that
\begin{align}
\psi_N^{\pm}(x)=\phi^{\pm}\left(\frac{x}{N}\right)+\OO\left(\frac{1}{N}\right),\nonumber
\end{align}
uniformly over $x/N$ in compact subsets of $\cM\deq \cup_{i=1}^k\cM_i$. All the aforementioned functions are analytic. 
\end{assumption}

We recall the Stieltjes transform $G_\mu$ of the equilibrium measure $\mu$ as in \eqref{e:defGmu}, 
and the two functions $R_\mu, Q_\mu$ as in \eqref{e:defRQ},
\begin{align}\begin{split}\label{e:mdefRQ}
R_\mu(z) \deq \phi^-(z)e^{-\theta G_\mu(z)}+\phi^+(z)e^{\theta G_\mu(z)},\\
Q_\mu(z) \deq \phi^-(z)e^{-\theta G_\mu(z)}-\phi^+(z)e^{\theta G_\mu(z)}.
\end{split}\end{align}
In the multi-cut case, they all depend on $N$. Moreover, under Assumptions \ref{a:mVN}, \ref{a:mfilling} and  \ref{a:mwx}, $R_\mu(z)$ is analytic on $\cal M$ \cite[Proposition 5.11]{Borodin2016}.

\begin{assumption}\label{a:mHz}
We assume there exists a function $H(z)$ analytic in $\cM$ and numbers $\{A_i, B_i\}_{i=1}^k$ such that 
\begin{itemize}
\item {$\hat a_i< A_i<B_i< \hat b_i$} for $i\in\qq{1, k}$;
\item $Q_\mu(z)=H(z)\prod_{i=1}^k\sqrt{(z-A_i)(z-B_i)}$, where the branch of square root is such that $\sqrt{(z-A_i)(z-B_i)}\sim z$ as $z\rightarrow \infty$; 
\item $H(z)\neq 0$ for all $z$ in a neighborhood of $\cup_{i=1}^k[\hat a_i, \hat b_i]$.
\end{itemize}
\end{assumption}

\begin{assumption}\label{a:mmoreVN}
For any $u$ in a small neighborhood of $\cup_{i=1}^k[A_i, B_i]$, { $V_N(u)$ is analytic and} the following holds
\begin{align*}
V_N'(u)=V'(u)+\OO\left(N^{ -1/3}\right).
\end{align*}
\end{assumption}

Under Assumptions \ref{a:mVN}, \ref{a:mfilling}, \ref{a:mwx} and \ref{a:mHz}, when restricted on a single cut $[\hat a_i, \hat b_i]$, the equilibrium measure $\rho_V|_{[\hat a_i, \hat b_i]}$ behaves as in one-cut case. Especially, it has  a single band, and the square root behaviors at $A_1, B_1, \cdots, A_k, B_k$.

To state the rigidity theorem in the multi-cut case, we need some more definitions. We fix a small parameter $\fa>0$, and an index $i\in \qq{1, k}$, and define the spectral domains
\begin{align}\begin{split}\label{def:mdomainD}
\cD_{i,r}&\deq \cD_{i,r}^{\rm int}\cup \cD^{\ext}_i,\\
\cD_{i,r}^{\rm int}&\deq \{E+\ri \eta \in \cM_i\cap \bC_+: E\in[A_i, B_i-r], \eta\sqrt{\kappa_E+\eta}\geq N^{\fa}/N\},\\
\cD^{\ext}_i&\deq \{E+\ri \eta \in \cM_i\cap \bC_+: E\leq A_i, \eta\geq (N^{\fa}/N)^{2/3}\},
\end{split}
\end{align}
and
\begin{align}\begin{split}\label{def:mD_0}
\cD_*&\deq \cD_*^{\rm int}\cup \cD_*^{\ext},\\
\cD_*^{\rm int}&\deq \{E+\ri \eta \in \cM\cap \bC_+: E\in\cup_{i=1}^k [A_i, B_i], \eta\sqrt{\kappa_E+\eta}\geq N^{-(1-\fa)/2}\},\\
\cD_*^{\ext}&\deq \{E+\ri \eta \in \cM\cap \bC_+: E\notin \cup_{i=1}^k [A_i, B_i], \eta\geq N^{-(1-\fa)/3}\}.
\end{split}
\end{align}
where $\kappa_E=\dist(E, \{A_1,B_1,\cdots, A_k,B_k\})$.

As in the one-cut case, we define the dual equilibrium measure $\mu^{\dual}$, 
\begin{align}\label{e:mdualmu}
\mu^{\dual}(\rd x)=\rho_V^{\dual}(x)\rd x, \quad \rho_V^{\dual}(x)=\theta^{-1}\sum_{i=1}^k{\bf 1}_{[\hat a_i, \hat b_i]}-\rho_V(x),
\end{align}
and the Stieltjes transform $G_\mu^{\dual}$ of the dual equilibrium measure $\mu^{\dual}$.
For any $z=E+\ri \eta\in \bC_+$, let $\kappa_E=\dist(E,\{A_1, B_1,\cdots,A_k,B_k\})$, $\hat\kappa_E=\dist(E, \{\hat a_1,\hat b_1,\cdots,\hat a_k, \hat b_k\})$, we define the control parameter:
\begin{align}\begin{split}\label{e:mdefcontrolmu}
\Theta_\mu(z)=\min\left\{\frac{|\Im[G_\mu(z)]|}{N\eta},\frac{ |\Im[G_\mu^{\dual}(z)]|}{N\eta}+\frac{1}{N(\eta+\hat \kappa_E)}\right\}.
\end{split}\end{align}
Thanks to the square root behavior of the equilibrium density $\rho_V(x)$ at $A_1, B_1,\cdots, A_k, B_k$, the asymptotic behavior of $\Theta_\mu$ depends on $\eta$, $\kappa_E$ and $\hat \kappa_E$,
\begin{align}\label{e:mThetaest}
\Theta_\mu(z)\asymp\left\{\begin{array}{ll}
\sqrt{\kappa_E+\eta}/N\eta, &\quad E\in \cup_{i=1}^k[A_i, B_i],\\
1/N\sqrt{\kappa_E+\eta}, & \quad E\not\in \cup_{i=1}^k[A_i, B_i] \text{ on voids}, \\
1/N\sqrt{\kappa_E+\eta}+1/N(\hat \kappa_E+\eta), & \quad E\not\in \cup_{i=1}^k[A_i, B_i] \text{ on saturated regions},\\
\end{array}\right.
\end{align}

For any $v\in \bC_+$, $0\leq t\leq 1$, $\al=\pm1,\pm \ri$, and large number $K=K(N,\Im[v])\leq N$, we introduce the deformed probability measure
\begin{align}\begin{split}\label{e:mdefdeformm}
\bP_N^{K,t,v,\al}=\frac{Z_N}{Z_N^{K,t,v,\al}}\bP_N e^{ 2K\Re\left[\sum_{i=1}^N \ln \left(1+\frac{\al t}{N(v-\ell_i/N)}\right)\right]}.
\end{split}\end{align}

As in the one-cut case, we can use the loop equation to improve the estimates of the Green's function. For the proof of the following Theorem, we need to { introduce the hyperelliptic integrals to invert the loop equation},
which is slightly different from that of Theorem \ref{t:bootstrap}. We will sketch the proof at the end of  this section.

\begin{theorem}\label{t:mbootstrap}
Let $v=E+\ri \eta\in \cM\cap \bC_+$, $\kappa_E=\dist(E,\{A_1, B_1,\cdots,A_k,B_k\})$, $\hat\kappa_E=\dist(E, \{\hat a_1,\hat b_1,\cdots,\hat a_k, \hat b_k\})$ and $\Theta_\mu$ as defined in \eqref{e:mdefcontrolmu}. We assume Assumptions \ref{a:mVN}, \ref{a:mfilling}, \ref{a:mwx}, \ref{a:mHz}, and that $N\eta\sqrt{\kappa_E+\eta}\gg 1$. For any $0\leq t\leq 1$, $\alpha=\pm 1, \pm\ri$, and $K\ll (N\eta)^2\sqrt{\kappa_E+\eta}$, we assume that with high probability w.r.t. $\bP_N^{K,t,v,\al}$ 
\begin{align}\label{e:moldbound}
G_N(v)=G_\mu(v)+\OO(\varepsilon),
\end{align}
then we have
\begin{align}\label{e:mbootstrap}
\bE_{\bP_N^{K,t,v,\al}}\left[G_N(v)-G_\mu(v)\right]=\OO\left(\frac{\varepsilon^2}{\sqrt{\kappa_E+\eta}}+\varepsilon_0\right),
\end{align}
where,
\begin{align}\label{e:defepsilon0}
\varepsilon_0=\frac{1}{\sqrt{\kappa_E+\eta}}\left(\Theta_\mu(v)+\frac{K\min\{|\Im[G_\mu(v)]|,|\Im[G_\mu^{\dual}(v)]|\}}{(N\eta)^2}+\frac{\ln N}{(N\eta)^2}+\frac{\ln N}{N}+\frac{\ln NK}{(N\eta)^3}+\frac{K^2}{(N\eta)^4}\right).
\end{align}
 Let $\tilde \varepsilon_0 =\max \{\veps_0, (\ln N)^{2}K^{-1}\}$. If  there exists some $\fc>0$, such that $\varepsilon \leq N^{-\fc}\sqrt{\kappa_E+\eta}$, then we have
\begin{align}\label{e:mLDP1}
\bP_N\left(|G_N(v)-G_\mu(v)|\lesssim s\tilde\varepsilon_0\right)\geq 1-e^{-csK\tilde \veps_0},
\end{align}
for any $s\geq 1$.
\end{theorem}

Using the large deviation estimate \cite[Proposition 5.6]{Borodin2016} as input for Theorem \ref{t:mrigidity}, we have the same statement as Corollary \ref{c:initialstep} for the multi-cut case: 
\begin{corollary}\label{c:minitialstep} We assume Assumptions \ref{a:mVN}, \ref{a:mfilling}, \ref{a:mwx} and \ref{a:mHz}.
Fix $s\geq (\ln N)^2$. With probability at least $1-e^{-cs}$ with respect to $\bP_N$, it holds uniformly for any $z=E+\ri \eta\in \cD_{*}$ (as defined in \eqref{def:mD_0}),
\begin{align}\label{e:LDPGNGmu}
\left|G_N(z)-G_\mu(z)\right|\lesssim \frac{s}{N\eta}.
\end{align}
\end{corollary}
With the estimate \eqref{e:LDPGNGmu}, and Proposition \ref{prop:rig2}, we can localize at a single cut $[\hat a_i, \hat b_i]$, the effect of the other cuts  is equivalent to a smooth potential of size $\OO(s/N)$. The same argument as in one-cut case leads to the following rigidity and edge universality statement.

\begin{theorem}\label{t:mrigidity}
We assume Assumptions \ref{a:mVN}, \ref{a:mfilling}, \ref{a:mwx}, \ref{a:mHz} and that $[\hat a_i, A_i]$ is a void region, i.e. $\rho_V(x)=0$ on $[\hat a_i, A_i]$, then the following holds:
\begin{enumerate}
\item
Fix $s\geq (\ln N)^2$ and small $r>0$. With probability at least $1-e^{-cs}$, it holds uniformly for any $z=E+\ri \eta\in \cD_{i,r}$, 
\begin{align}\label{e:mbulkrig}
\left|G_N(z)-G_\mu(z)\right|\leq \frac{s}{N\eta}.
\end{align}

\item 
Fix a small $\fc$ such that $0<\fc<\fa/4$. With probability at least $1-\exp(-c(\ln N)^{2})$, it holds uniformly for $z=E+\ri (N^{\fa}/N)^{2/3}\in  \cM_i$ with $E\leq A_i-N^{2\fc}(N^{\fa}/N)^{2/3}$, 
\begin{align}\label{e:medgerig}
\left|G_N(z)-G_\mu(z)\right|\ll \frac{1}{N\eta}.
\end{align}
\end{enumerate}
A similar statement holds if $[B_i, \hat b_i]$ is a void region.
\end{theorem}

\begin{theorem}\label{t:medgeuniversality}
Fix $\beta\geq 1$, $\theta=\beta/2$,  $0<\fb<1/13$ and $L=N^{\fb}$. We assume Assumptions \ref{a:mVN}, \ref{a:mfilling}, \ref{a:mwx}, \ref{a:mHz}, \ref{a:mmoreVN} and that $[\hat a_i, A_i]$ is a void region, $\rho_V(u)=(1+\OO(u-A_i))s_{A_i}\sqrt{u-A_i}/\pi$.  
Take any fixed $m\geq 1$ and a continuously differentiable compactly supported function $O:\bR^m\rightarrow \bR$. For any index set $\Lambda\in \qq{1, L}$ with $|\Lambda|=m$, we denote the shifted index set $\Lambda_i=\{k+I_i^--1: k\in \Lambda\}$, then
\begin{align}\label{e:medgeuniv}
\left|\bE_{\bP_N}\left[O\left(\left(N^{2/3}k^{1/3}s_{A_i}^{2/3}(\ell_k/N-\gamma_k)\right)_{k\in \Lambda_i}\right)\right]-\bE_{G\beta E}\left[O\left(\left(N^{2/3}k^{1/3}(x_k/N-\tilde \gamma_k)\right)_{k\in \Lambda}\right)\right]\right|=\OO(N^{-\chi}),
\end{align}
where $\gamma_k$ are classical eigenvalue locations of $\rho_V$, and $\tilde \gamma_k$ are classical eigenvalue locations of semi-circle distribution. A similar statement holds if $[B_i, \hat b_i]$ is a void region.
\end{theorem}

\begin{proof}[Proof of Theorem \ref{t:mrigidity}]
We recall the notations $\phi_N^+(x),\phi_N^-(x), R_N(zN), \cE(z)$ from \eqref{e:defPhi+-},\eqref{e:defRNPhi} and \eqref{e:defE}. In the multi-cut setting, $R_N(zN)$ is also an analytic function in $\cM$ as proven in \cite[Theorem 4.1]{Borodin2016}. We take two sets of oriented contours: $\cC_1$ consists of two clockwise oriented circles, one is centered at $v$ with radius $\eta/2$, and the other is centered at $\bar v$ with radius $\eta/2$.  $\cC_2$ consists of $k$ counterclockwise oriented contours: $\cC_2^1, \cC_2^2,\cdots, \cC_2^k$, such that each $\cC_2^i\subset \cM_i$ ($1\leq i\leq k$) encloses a small neighborhood of $[\hat a_i,\hat b_i]$. By dividing \eqref{e:perturbedRN} by $2\pi \ri (z-v) H(z)$ (as defined in Assumption \ref{a:mHz}), and integrating along the union of contours $\cC_1\cup \cC_2$, we get
\begin{align}\label{e:relationIs}
0=I_1+I_2+I_3+I_4,
\end{align}
where $I_1, I_2, I_3,I_4$ are defined in \eqref{e:defIs}. We have the same estimate for $I_1, I_3,I_4$ as in the proof of Theorem \ref{t:bootstrap}:
\begin{align}
\begin{split}\label{e:estimateIs}
I_1
&=\theta \prod_{i=1}^k\sqrt{(v-A_i)(v-B_i)}
\bE\left[G_N(v)-G_\mu(v)\right]\\
&+\OO\left(\varepsilon^2
+\frac{\left|\bE\left[\Im[G_N(v)-G_\mu(v)]\right]\right|}{N\eta}+\Theta_\mu(v)+\frac{1}{N}+\frac{\ln N}{(N\eta)^2}\right)\\
I_3&=\OO\left(\frac{\ln N}{N}\right)\\
I_4&\lesssim \frac{K\min\{\bE[|\Im[G_N(v)]|],\bE[|\Im[G_N^{\dual}(v)]|]\}}{(N\eta)^2}+\OO\left(\frac{1}{N}+\frac{K}{N^2\eta}+\frac{\ln N K}{(N\eta)^3}+\frac{K^2}{(N\eta)^4}\right).
\end{split}
\end{align}

For the estimate of $I_2$ in the multi-cut setting, we need to introduce certain notations from the theory of hyperelliptic integrals to invert the loop equation \cite[Section 6]{Borodin2016}.

Let $P(z)=p_0+p_1z+\cdots+p_{k-2}z^{k-2}$ be a polynomial of degree $k-2$, and consider the map
\begin{align*}
\Omega: P(z)\mapsto \left(\frac{1}{2\pi \ri}\int_{\cC_2^1}\frac{P(z)\rd z}{\prod_{i=1}^k\sqrt{(z-A_i)(z-B_i)}},\cdots,\frac{1}{2\pi \ri}\int_{\cC_2^k}\frac{P(z)\rd z}{\prod_{i=1}^k\sqrt{(z-A_i)(z-B_i)}}
\right).
\end{align*}
Note that the sum of the integrals in the definition of $\Omega$ equals zero. And it turns out that $\Omega$ is an isomorphism between $(k-1)$-dimensional vector spaces for any $k\geq 2$. Using $\Omega$ we can now define a more complicated map $\Upsilon$. Given a function $f(z)$ defined on the contours $
\cC_2^i$ and such that the sum of its integrals over these contours is zero, we define a function $\Upsilon_z[f]$ through
\begin{align*}
\Upsilon_z[f]=f(z)+\frac{P(z)}{\prod_{i=1}^k\sqrt{(z-A_i)(z-B_i)}},
\end{align*}
where $P(z)$ is the unique polynomial of degree at most $k-2$, such that for each $i\in \qq{1,k}$,
\begin{align*}
\frac{1}{2\pi \ri}\int_{\cC_2^i}\Upsilon_z[f]\rd z=0.
\end{align*}
The polynomial $P(z)$ can be evaluated in terms of the map $\Omega$ via
\begin{align*}
P=\Omega^{-1}\left(-\frac{1}{2\pi\ri}\int_{\cC_2^1}f(z)\rd z,\cdots, -\frac{1}{2\pi\ri}\int_{\cC_2^k}f(z)\rd z\right).
\end{align*}
The same as in \eqref{e:I2}, we have
\begin{align}\begin{split}\label{e:mI2_1}
I_2=\frac{\theta}{2\pi \ri}\int_{\cC_2}\frac{\bE[G_N(z)-G_\mu(z)]\prod_{i=1}^k\sqrt{(z-A_i)(z-B_i)}}{z-v}\rd z+\OO\left(\frac{\ln N}{N}\right).
\end{split}\end{align}
Since $\bE[G_{N}(z)-G_\mu(z)]$ is analytic outside the contours of integration and decays as $1/z^2$ as $z\rightarrow \infty$, the integral is a polynomial $P_N(v)$ of degree at most $k-2$ (to see that one uses $(z-v)^{-1}=z^{-1}\sum_{n\geq 0}(v/z)^m$). We conclude,
\begin{align}\label{e:mI2}
I_2=P_N(v)+\OO\left(\frac{\ln N}{N}\right).
\end{align}
Combining the estimates \eqref{e:relationIs}, \eqref{e:estimateIs}, and \eqref{e:mI2} all together, we deduce that
\begin{align}\begin{split}\label{e:mbootstrap0}
&\theta \prod_{i=1}^k\sqrt{(v-A_i)(v-B_i)}
\bE\left[G_N(v)-G_\mu(v)\right]+P_N(v)
=\OO\left(\frac{1}{N\eta}+\frac{K}{(N\eta)^2}\right)\left|\bE[\Im[G_N(v)-G_\mu(v)]]\right|\\
&+\OO\left(\varepsilon^2+\Theta_\mu(v)+\frac{K\min\{|\Im[G_\mu(v)]|,|\Im[G_\mu^{\dual}(v)]|\}}{(N\eta)^2}+\frac{\ln N}{(N\eta)^2}+\frac{\ln N}{N}+\frac{\ln N K}{(N\eta)^3}+\frac{K^2}{(N\eta)^4}\right).
\end{split}\end{align}
Notice that $\prod_{i=1}^k\sqrt{(v-A_i)(v-B_i)}\asymp \sqrt{\kappa_E+\eta}$, and recall $\varepsilon_0$ from \eqref{e:defepsilon0},
\begin{align}\begin{split}
&\theta
\bE\left[G_N(v)-G_\mu(v)\right]+\frac{P_N(v)}{ \prod_{i=1}^k\sqrt{(v-A_i)(v-B_i)}}\\
&=\OO\left(\frac{1}{N\eta\sqrt{\kappa_E+\eta}}+\frac{K}{(N\eta)^2\sqrt{\kappa_E+\eta}}\right)\left|\bE[\Im[G_N(v)-G_\mu(v)]]\right|
+\OO\left(\frac{\varepsilon^2}{\sqrt{\kappa_E+\eta}}+\varepsilon_0\right).\nonumber
\end{split}\end{align}
Now we are in a position to apply the map $\Upsilon_v$. Indeed, the integral of $G_N(z)$ around $\cC_2^i$ is deterministic and equals $n_i(N)/N$. On the other hand, the integral of $G_\mu(z)$ around $\cC_2^i$ equals the total mass of $\mu(x)$ inside $\cC_2^i$, which is $\hat n_i=n_i(N)/N$. Thus the integral of $\bE[G_N(z)-G_\mu(z)]$ around each loop $\cC_2^i$ vanishes. So $\bE[G_N(z)-G_\mu(z)]$ is given by
\begin{align}\begin{split}\label{e:invertloop}
&\bE[G_N(z)-G_\mu(z)]\\
&=\Upsilon_v\left[\OO\left(\frac{1}{N\eta\sqrt{\kappa_E+\eta}}+\frac{K}{(N\eta)^2\sqrt{\kappa_E+\eta}}\right)\left|\bE[\Im[G_N(v)-G_\mu(v)]]\right|
+\OO\left(\frac{\varepsilon^2}{\sqrt{\kappa_E+\eta}}+\varepsilon_0\right)\right]\\
&=\OO\left(\frac{1}{N\eta\sqrt{\kappa_E+\eta}}+\frac{K}{(N\eta)^2\sqrt{\kappa_E+\eta}}\right)\left|\bE[\Im[G_N(v)-G_\mu(v)]]\right|
+\OO\left(\frac{\varepsilon^2}{\sqrt{\kappa_E+\eta}}+\varepsilon_0\right),
\end{split}\end{align}
where in the last line we used that the map $\Upsilon_v$ depends only on the contours $\cC_2$, and is Lipschitz.

By our assumption $N\eta\sqrt{\kappa_E+\eta}\gg 1$, and $(N\eta)^2\sqrt{\kappa_E+\eta}\gg K$, it follows by rearranging \eqref{e:invertloop}, 
\begin{align*}\begin{split}
\bE[G_N(z)-G_\mu(z)]
=\OO\left(\frac{\varepsilon^2}{\sqrt{\kappa_E+\eta}}+\varepsilon_0\right)\,.\nonumber
\end{split}\end{align*}
This finishes the proof of \eqref{e:mbootstrap}. The remaining of the Theorem \ref{t:mbootstrap} can be proven in the same way as in the one-cut case.

\end{proof}

\section{Examples and Generalizations}

In this section, we discuss the generalizations of Theorems \ref{t:rigidity} and \ref{t:edgeuniversality} in the following four senses:
\begin{enumerate}
\item In the special case when $\theta=1$, thanks to the duality between particles and holes (see \cite{MR1952523}), our results imply the rigidity and edge universality for both void and \emph{saturated} regions.

\item Thanks to certain large deviation estimates for the largest and smallest particles, our results imply the rigidity and edge universality of stochastic systems supported on \emph{infinite} intervals.

\item By conditioning on filling fractions, our results imply the edge universality of multi-cut stochastic systems \emph{without fixed filling fractions}.

\item Thanks to certain large deviation estimates for the number of total particles, our results imply the edge universality of stochastic systems \emph{without fixed number of particles}.

\end{enumerate}

We demonstrate how the generalizations of Theorems \ref{t:rigidity} and \ref{t:edgeuniversality} imply the rigidity and edge universality of certain stochastic systems, i.e., Binomial ensembles, discretized $\beta$-ensemble with convex potentials,  Lozenge Tilings and Jack deformation of Plancherel measures.

\subsection{$\theta=1$ Binomial Weight}\label{s:bw}

The first example is the binomial ensemble, which is a special case of the discrete $\beta$-ensemble defined in Section \ref{s:notation}. This ensemble is also
known as the \emph{Krawtchouk orthogonal polynomial ensemble}. 
We demonstrate how to use the duality between the particles and holes to prove the rigidity and edge universality for both void and saturated regions.

Let $\bP_N$ be a discrete $\beta$-ensemble as defined in Section \ref{s:notation} with $\theta=1$, i.e.,
\begin{align*}
\bP_N(\ell_1, \ell_2, \cdots, \ell_N)=\frac{1}{Z_N} \prod_{1\leq i<j\leq N}(\ell_i-\ell_j)^2
\prod_{i=1}^N w(\ell_i; N),
\end{align*}
on $\bW_N$, where $\bW_N$ is the set of ordered $N$-tuples of integers
\begin{align}\label{e:defWN1}
 0\leq \ell_1<\ell_2<\dots<\ell_N\leq M.
\end{align}
We can consider its dual particle locations. More precisely, for any given particle configuration $\ell_1, \ell_2, \cdots, \ell_N$, we denote the locations of holes, $0\leq x_1<x_2<\cdots<x_{M-N+1}\leq M$. The distribution $\bP_N$ induces a distribution on the configuration of the holes, 
\begin{align}\label{e:dualdensity0}
\bP_N^{\dual} (x_1, \cdots, x_{M-N+1})
=&\bP_N(\ell_1, \ell_2, \cdots, \ell_{N})
=\frac{1}{Z_N} \prod_{1\leq i<j \leq N}
 (\ell_i-\ell_j)^2 \prod_{i=1}^{N} w(\ell_i; N).
\end{align}
Up to some constant, we can rewrite the righthand side of \eqref{e:dualdensity0} in terms of $x_1,x_2,\cdots,x_{M-N+1}$,
\begin{align*}\begin{split}
\prod_{1\leq i<j \leq N}
 (\ell_i-\ell_j)^2 \prod_{i=1}^{N} w(\ell_i; N)
 &=\frac{\prod_{1\leq i\leq N,\atop0\leq k\leq M, k\neq \ell_i}|\ell_i-k|}{\prod_{1\leq i\leq M-N+1,\atop 0\leq k\leq M, k\neq x_i}|x_i-k|}\prod_{1\leq i<j \leq N}
 (x_i-x_j)^2 \prod_{i=1}^{N} w(\ell_i; N)\\
 &\propto \frac{1}{\prod_{1\leq i\leq M-N+1,\atop 0\leq k\leq M, k\neq x_i}(x_i-k)^2}\prod_{1\leq i<j \leq N}
 (x_i-x_j)^2\prod_{i=1}^{N-M+1}w(x_i;N)^{-1}.
\end{split}
\end{align*}
where we used that 
\begin{align*}
\prod_{1\leq i\leq N,\atop0\leq k\leq M, k\neq \ell_i}|\ell_i-k|
\prod_{1\leq i\leq M-N+1,\atop 0\leq k\leq M, k\neq x_i}|x_i-k|&=\prod_{0\leq j\neq k\leq M}|j-k|\\
\prod_{i=1}^{N} w(\ell_i; N)\prod_{i=1}^{N-M+1}w(x_i;N)&=\prod_{j=1}^M w(j;N),
\end{align*}
are both constants depending on $M,N$.

Therefore, there exists some constant $Z_N^{\dual}$ so that  the distribution of the configuration of the holes takes the following form
\begin{align}\begin{split}\label{e:dualdensity}
\bP_N^{\dual} (x_1, \cdots, x_{M-N+1})
=& \frac{1}{Z_N^{\dual}} \prod_{1\leq i<j \leq M-N+1}
 (x_i-x_j)^2 \prod_{i=1}^{N-M+1}w(x_i;N)^{-1}\prod_{0\leq k\leq M,\atop k\neq x_i}(x_i-k)^{-2},
\end{split}\end{align}
which is again a discrete $\beta$-ensemble with $\theta=1$ and weight,
\begin{align*}\nonumber
w^{\dual}(x;N)=w(x;N)^{-1}\prod_{0\leq k\leq M,\atop k\neq x}(x-k)^{-2}.
\end{align*}
This is called the dual ensemble as studied in \cite[Section 3.2]{MR1952523}.
The saturated regions (void regions) of the equilibrium measure of $\bP_N$ become the void regions (saturated regions) of the equilibrium measure of $\bP_N^{\dual}$.  In the following we use the Krawtchouk orthogonal polynomial ensemble as an example to illustrate how to use the duality between the particles and holes to prove the rigidity and edge universality for both void and saturated regions.

Fix two integers $0\leq N\leq M$, with $M=\lceil \fm N\rceil$ and $\fm>1$, and consider the space $\bW_N$ of $N$-tuples of integers as defined in \eqref{e:defWN1}. The Krawtchouk orthogonal polynomial ensemble is the probability distribution $\bP_{N}$ on $\bW_N$, 
\begin{equation}
\label{eq_Binomial_ensemble}
 \bP_{N}(\ell_1,\dots,\ell_N)=\frac{1}{Z(N,M)} \prod_{1\leq i<j \leq N}
 (\ell_i-\ell_j)^2 \prod_{i=1}^{N} {M \choose \ell_i}.
\end{equation}
We remark that the partition function $Z(N,M)$ is explicitly known in this case:
\begin{align*}\nonumber
 Z(N,M)= 2^{N(M-N+1)} (M!)^N  \prod_{j=0}^{N-1} \frac{ j!}{(M-j)!}.
\end{align*}
We denote the empirical particle density $\mu_N=N^{-1}\sum_{i=1}^N \delta_{\ell_i/N}$. The empirical particle denisty $\mu_N$ converges to the equilibrium measure, which is absolutely continuous. For $\fm> 2$, the density of the
 measure $\mu_\fa$ is given by
 \begin{align*}
  \mu_\fm(x)=\left\{\begin{array}{cc}
  \dfrac{1}{\pi}{\mathrm{arccot}}\left(\dfrac{\fm-2}{2\sqrt{\fm-1-\left(x-\fm/2\right)^2}}\right),&
  \left|x-\frac{\fm}{2}\right|<\sqrt{\fm-1},\\
  0,& \left|x-\frac{\fm}{2}\right|\geq \sqrt{\fm-1},
  \end{array}\right.\nonumber
 \end{align*}
 and for $1<\fm< 2$,
\begin{align*}
  \mu_\fm(x)=\left\{\begin{array}{cc}
  \dfrac{1}{\pi}{\mathrm{arccot}}\left(\dfrac{\fm-2}{2\sqrt{\fm-1-\left(x-\fm/2\right)^2}}\right),&
  \left|x-\frac{\fm}{2}\right|<\sqrt{\fm-1},\\
  1,& \frac{\fm}{2}\geq \left|x -\frac{\fm}{2}\right|\geq \sqrt{\fm-1},\\
  0, & \left|x-\frac{\fm}{2}\right|\geq \frac{\fm}{2},
  \end{array}\right.\nonumber
 \end{align*}
where $\mathrm{arccot}(y)$ is the inverse cotangent function. We use the notations from Section \ref{s:notation}, 
\begin{align*}\nonumber
\hat a=0, \quad \hat b=\fm,\quad A=\frac{\fm}{2}-\sqrt{\fm-1}, \quad B=\frac{\fm}{2}+\sqrt{\fm-1}.
\end{align*}

The Assumptions \ref{a:VN}, \ref{a:wx} and \ref{a:Hz} are verified in \cite[Section 2]{Borodin2016}. For the case $\fm\geq 2$, since both $[\hat a, A]$ and $[B, \hat b]$ are void regions, Theorem \ref{t:rigidity} and \ref{t:edgeuniversality} imply that the rigidity and edge universality hold. For the case $1<\fm< 2$, unfortunately, since both $[\hat a, A]$ and $[B, \hat b]$ are saturated regions, Theorem \ref{t:rigidity} and \ref{t:edgeuniversality} do not say anything. However, in this case, the rigidity and edge universality follow by considering the dual particle locations. As in \eqref{e:dualdensity}, for any given particle configuration $\ell_1, \ell_2, \cdots, \ell_N$, the distribution on the configuration of the holes, $0\leq x_1<x_2<\cdots<x_{M-N+1}\leq M$ is given by 
\begin{align*}\begin{split}
\bP_N^{\dual} (x_1, \cdots, x_{M-N+1})
= \frac{1}{Z^{\dual}(N,M)} \prod_{1\leq i<j \leq M-N+1}
 (x_i-x_j)^2 \prod_{i=1}^{N-M+1}w^{\dual}(x_i;N), 
\end{split}\end{align*}
which is again a discrete $\beta$-ensemble as defined in Section \ref{s:notation} with $\theta=1$, and weight,
\begin{align}\nonumber
w(x;N)^{\dual}={M \choose x}^{-1}\prod_{0\leq j\leq M,\atop j\neq x}(x-j)^{-2}.
\end{align}
One can check that for the dual measure $\bP_N^{\dual}$, the Assumptions \ref{a:VN}, \ref{a:wx} and \ref{a:Hz} hold, and its equilibrium measure is given by 
\begin{align*}
  \mu^{\dual}_\fm(x)=\left\{\begin{array}{cc}
  \frac{1}{\fm-1}\left(1-\dfrac{1}{\pi}{\mathrm{arccot}}\left(\dfrac{\fm-2}{2\sqrt{\fm-1-\left(x-\fm/2\right)^2}}\right)\right)&
  \left|x-\frac{\fm}{2}\right|<\sqrt{\fm-1},\\
  0& \left|x -\frac{\fm}{2}\right|\geq \sqrt{\fm-1}.
  \end{array}\right.\nonumber
 \end{align*}
Since $[\hat a, A]$ and $[B,\hat b]$ are both void regions for the dual equilibrium measure, $\mu_{\fm}^{\dual}$, Theorem \ref{t:rigidity} and \ref{t:edgeuniversality} imply the rigidity and edge universality for the configuration of holes. Since the particle configuration is determined by the hole configuration, the rigidity of the hole configuration implies the rigidity of the particle configuration.

In general for the discrete $\beta$-ensemble with $\theta=1$ and satisfying Assumptions \ref{a:VN}, \ref{a:wx}, \ref{a:Hz} and \ref{a:moreVN}, by considering the duality between particles and holes, we have
\begin{enumerate}
\item The optimal rigidity holds, i.e. any particle is close to its classical location given by the equilibrium measure.
\item The distributions of extreme particles or extreme holes after proper scaling are the same as the distributions of extreme eigenvalues of $GUE$. 
\end{enumerate}

\subsection{Arbitrary Convex Potential on $\bR$ With No Saturation}\label{s:cp}
The second example is the discrete $\beta$-ensemble supported on infinite intervals. We take a real convex analytic function $V(x)$, i.e. such that $V''(x)>0$ for all $x\in \bR$. Fix a constant $\kappa>0$ such that
\begin{equation}
\label{eq_potential_behavior}
 \liminf_{x\to\infty} \frac{\kappa V(x)}{2\theta \ln |x|}>1,
\end{equation}
  and consider the probability distribution given by
\begin{equation}
\label{eq_convex_distribution}
 \bP(\ell_1,\dots,\ell_N)= \frac{1}{Z_N}
 \prod_{1\le i<j \le N}
 \frac{\Gamma(\ell_j-\ell_i+1)\Gamma(\ell_j-\ell_i+\theta)}{\Gamma(\ell_j-\ell_i)\Gamma(\ell_j-\ell_i+1-\theta)}
 \prod_{i=1}^N \exp\left( - \kappa N  \cdot V\left(\frac{\ell_i}{N}\right)\right),
\end{equation}
on $N$-tuples $\ell_1<\ell_2<\dots<\ell_N$ such that $\ell_j=\lambda_j+\theta j$ and $\lambda_1\leq \la_2\leq \cdots\leq \la_N$ are integers. We are in the framework of Section \ref{s:notation} except that now the particle configuration is supported on an infinite interval, i.e., $\bR$. We have $a(N)=\hat a=-\infty$ and $b(N)=\hat b=\infty$. We hereafter assume that
the equilibrium measure has no saturated regions. Then the constrained equilibrium measure $\mu$ is the same as the unconstrained equilibrium and $\mu$ has only one band.
This is the case when 
 $\kappa$ is small enough. By taking 
\begin{align}\label{e:defpsiphi}
\psi_N^{+}(x)=\exp\{\kappa N\left(V((x-1)/N)-V(x/N)\right)\}, \quad \psi_N^-(x)=1,\quad \phi^+(z)=\exp\{-\kappa V'(z)\}, \quad \phi^-(z)=1,
\end{align}
we also have the following Nekrasov's equation,
\begin{align}\label{e:nekrasov2}
R_N(\xi)=\psi^-_N(\xi)\bE_{\bP_N}\left[\prod_{i=1}^{N}\left(1-\frac{\theta}{\xi-\ell_i}\right)\right]+\psi^+_N(\xi)\bE_{\bP_N}\left[\prod_{i=1}^{N}\left(1+\frac{\theta}{\xi-\ell_i-1}\right)\right].
\end{align}
where $R_N(\xi)$ is analytic over $\bC$. In the argument of the proof of Theorem \ref{t:bootstrap}, we take a contour integral which encloses a small neighborhood of $[\hat a, \hat b]$. For this to work, we need to localize the probability measure $\bP_N$ onto a finite interval. Thanks to \cite[Proposition 9.3]{Borodin2016}, with exponentially high probability the empirical particle density of \eqref{eq_convex_distribution} is supported on a finite interval, i.e. there exist constants $L, c$ such that
\begin{align}\label{e:extremeeiLDP}
\bP_{N}\left(-L\leq \ell_1/N\leq \ell_N/N\leq L\right)\geq 1-e^{-cN},
\end{align}
for $N$ large enough. As in \cite{Borodin2016}, we define the new measure $\hat \bP_N$ as $\bP_N$ conditioned on the event that $\{\sup_{1\leq i\leq N}|\ell_i/N|\leq L\}$. For this new density $\hat\bP_N$, we take 
\begin{align*}
\hat a=-L, \quad \hat b=L,\quad a(N)=\max\{\la+\theta< -NL:\la\in \bZ\}, \quad b(N)=\min\{\la+N\theta>NL: \la \in \bZ\},
\end{align*}
and $\psi^{\pm}_N(z)$ and $\psi^\pm(z)$ are as given in \eqref{e:defpsiphi}.
The Assumptions \ref{a:VN}, \ref{a:wx}, and \ref{a:Hz} are verified in \cite[Section 9.3]{Borodin2016}.  However, since we no longer have the vanishing boundary condition $w(a(N); N)=w(b(N); N)=0$, we don't have the precise Nekrasov's equation. The righthand side of \eqref{e:nekrasov2}, after replacing $\bP_N$ by $\hat\bP_N$
\begin{align*}
\psi^-_N(\xi)\bE_{\hat\bP_N}\left[\prod_{i=1}^{N}\left(1-\frac{\theta}{\xi-\ell_i}\right)\right]+\psi^+_N(\xi)\bE_{\hat\bP_N}\left[\prod_{i=1}^{N}\left(1+\frac{\theta}{\xi-\ell_i-1}\right)\right],
\end{align*}
has two poles at $a(N)+1$ and $b(N)$. Fortunately, thanks to the large deviation estimate \eqref{e:extremeeiLDP}, the residues at $a(N)+1$ and $b(N)$ are exponentially small. Therefore, we have the following approximated Nekrasov's equation
\begin{align}\nonumber
R_N(\xi)+\OO(e^{-cN})=\psi^-_N(\xi)\bE_{\hat\bP_N}\left[\prod_{i=1}^{N}\left(1-\frac{\theta}{\xi-\ell_i}\right)\right]+\psi^+_N(\xi)\bE_{\hat\bP_N}\left[\prod_{i=1}^{N}\left(1+\frac{\theta}{\xi-\ell_i-1}\right)\right],
\end{align}
where $R_N(\xi)$ is an analytic function, and the error $\OO(e^{-cN})$ is uniform provided that $\dist(\xi/N, [-L,L])\geq 2/N$.  The proof of Theorem \ref{t:bootstrap} can be carried through without any change. For the remaining proofs, we do not use Nekrasov's equation. Therefore, the rigidity and edge universality hold for both $\hat \bP_N$ and $\bP_N$,
\begin{corollary}
We fix a constant $\kappa>0$,  take a real convex analytic function $V(x)$ satisfying \eqref{eq_potential_behavior}, and consider the probability $\bP_N$ as given in \eqref{eq_convex_distribution}. Suppose that  the equilibrium measure $\mu(x)=\rho_V(x)\rd x$ is supported on $[A,B]$, has  no saturated region, and therefore only one band. Then the rigidity and edge universality hold for $\bP_N$:
\begin{enumerate}
\item For any $\fa>0$, with probability at least $1-\exp(-c(\ln N)^{2})$, we have
\begin{align*}
\left|\frac{\ell_i}{N}-\gamma_i\right|\leq \frac{N^{\fa}}{N^{2/3}\min\{i, N-i\}^{1/3}},\quad i\in \qq{1,N},
\end{align*}
where $\gamma_i$ are the classical particle locations of the equilibrium measure $\mu$.
\item Fix $\beta\geq 1$, $0<\fb<1/13$ and $L=N^{\fb}$. Let $\rho_V(u)=(1+\OO(u-A))s_{A}\sqrt{u-A}/\pi$.  
Take any fixed $m\geq 1$ and a continuously differentiable compactly supported function $O:\bR^m\rightarrow \bR$. For any index set $\Lambda\in \qq{1, L}$ with $|\Lambda|=m$, 
\begin{align*}
\left|\bE_{\bP_N}\left[O\left(\left(N^{2/3}k^{1/3}s_{A}^{2/3}(\ell_k/N-\gamma_k)\right)_{k\in \Lambda}\right)\right]-\bE_{G\beta E}\left[O\left(\left(N^{2/3}k^{1/3}(x_k/N-\tilde \gamma_k)\right)_{k\in \Lambda}\right)\right]\right|=\OO(N^{-\chi}),
\end{align*}
where $\tilde \gamma_i$ are the classical particle locations of the semi-circle distribution. The same statement holds for particles close to the right edge.
\end{enumerate}
\end{corollary}

\subsection{Lozenge Tilings}\label{s:lt}
The third example is the uniformly random lozenge tilings of a hexagon with a hole. We demonstrate that by conditioning on the filling fractions, our results imply the edge universality of multi-cut stochastic systems without filling fractions.

Consider an $A\times B\times C$ hexagon drawn on the regular triangular lattice. We cut a rhombic $D\times D$ hole in the hexagon, where the bottom point of the hole is at distance $t$ from the left side of the hexagon and at height $H$ (counted from the bottom of the hexagon along $t$-th vertical line). Let $\bP_N$ be the probability distribution of the horizontal lozenges (outside the hole) on $t$-th vertical line induced by the uniform measure on all tilings. Assuming $t>\max(B,C)$, which yields $N=B+C-D-t$, and introducing the coordinate system such that the lowest possible position for horizontal lozenge on the $t$-th vertical line is $1$ and the highest one is $A+B+C-t$. The probability distribution $\bP_N$ is given by
\begin{align*}
\bP_N(\ell_1, \ell_2, \cdots, \ell_N)
=\frac{1}{Z_N}\prod_{1\leq i<j\leq N}(\ell_i-\ell_j)^2
\prod_{i=1}^N w(\ell_i;N),
\end{align*}
where 
\begin{align*}
w(\ell_i;N)=(A+B+C+1-t-\ell_i)_{t-B}(\ell_i)_{t-C}(H-\ell_i)_D(H-\ell_i)_D,
\end{align*}
and $(a)_n$ is the Pochhammer symbol, $(a)_n=a(a+1)\cdots(a+n-1)$.
This is a two-cut discrete $\beta$-ensemble with $k=2$ and $\theta=1$. We also need two filling fractions $n_1$ and $n_2$: we consider such tilings that there are $n_1(N)$ horizontal lozenges (on $t$-th vertical line) below the hole and $n_2(N)$ lozenges above. We use $N$ as a large parameter, and suppose that
\begin{align*}\begin{split}
A&=\hat A N+\OO(1),\quad B=\hat B N+\OO(1),
\quad C=\hat C+\OO(1),\quad t={\hat{t}} N+\OO(1),\\
H&=\hat H N+\OO(1), \quad D=\hat DN +\OO(1), \quad n_1(N)=\hat n_1 N, \quad n_2(N)=\hat n_2 N.
\end{split}
\end{align*}
Using the notations from Section \ref{s:multi-cut}, we have 
\begin{align*}
\hat a_1 =0,\quad \hat b_1 =\hat H, \quad \hat a_2=\hat H+\hat D,\quad \hat b_2=\hat A+\hat B+\hat C-\hat t.
\end{align*}
We further assume that $\hat t>\max\{\hat B, \hat C\}$, and there exists a constant $\fd>0$,
\begin{align*}
\fd<\hat n_1 <\hat H-\fd,\quad \fd<\hat n_2 <\hat A+\hat B+\hat C-\hat t-\hat H-\hat D-\fd.
\end{align*}
For the density $\bP_N$, the potential $V(x)$ has the form,
\begin{align*}
\begin{split}
V(x)&=(\hat A+\hat B+\hat C-\hat t-x)\ln(\hat A+\hat B+\hat C-\hat t-x)-(\hat A+\hat C-x)\ln(\hat A+\hat C-x)\\
&+x\ln x-(\hat t-\hat C-x)\ln (\hat t-\hat C-x)
+2(\hat H-x)\ln (\hat H-x)-2(\hat H+\hat D-x)\ln (\hat H+\hat D-x).
\end{split}
\end{align*}
We denote the equilibrium measure with filling fraction  $\hat n=(\hat n_1, \hat n_2)$ by $\mu^{\hat n}=\rho_V^{\hat n}(x)\rd x$, which is given by the minimizer of the following functional 
\begin{align}\label{e:equilibrium_fixed_filling}
-\int_{x\neq y}\ln |x-y|\rho(x)\rho(y)\rd x\rd y+\int V(x)\rho(x)\rd x,
\end{align}
over all the densities $\rho(x)$, supported on $[\hat a_1, \hat b_1]\cup [\hat a_2,\hat b_2]$, with $0\leq \rho(x)\leq 1$, and 
\begin{align*}
\int_{a_i}^{b_i}\rho(x)\rd x=n_i,\quad i=1,2.
\end{align*} 
The function $Q_{\mu(\hat n)}$ is a square root of a polynomial of degree $6$ \cite[Section 9.2]{Borodin2016} 
\begin{align*}
Q_{\mu^{\hat n}}(z)=c^{\hat n}(z-C^{\hat n})\sqrt{(z-A_1^{\hat n})(z-B_1^{\hat n})(z-A_2^{\hat n})(z-B_2^{\hat n})},
\end{align*}
where $c^{\hat n}$ is a constant, and 
\begin{align*}
a_1<A_1^{\hat n}<B_1^{\hat n}<b_1<C^{\hat n}<a_2<A_2^{\hat n}<B_2^{\hat n}<b_2.
\end{align*}

The Assumptions \ref{a:mVN}, \ref{a:mfilling}, \ref{a:mwx}, \ref{a:mHz} and \ref{a:mmoreVN}  are verified in Section \cite[Section 9.2]{Borodin2016}.  Theorem \ref{t:mrigidity} and \ref{t:medgeuniversality}, combining with a duality argument as in Section \ref{s:bw}, imply the rigidity and edge universality for lozenge tilings of a hexagon with a hole (with fixed filling fraction above and below the hole).

In the following, we consider the uniformly random lozenge tilings without fixed filling fraction. 
We denote the empirical particle density $\mu_N=N^{-1}\sum_{i=1}^N \delta_{\ell_i/N}$, and $\mu^*=\rho_V^*(x)\rd x$ the equilibrium measure, which is given by the minimizer of the following functional 
\begin{align}\label{e:defIV}
-\int_{x\neq y}\ln |x-y|\rho(x)\rho(y)\rd x\rd y+\int V(x)\rho(x)\rd x,
\end{align}
over all the densities $\rho(x)$, supported on $[\hat a_1, \hat b_1]\cup [\hat a_2,\hat b_2]$, with $0\leq \rho(x)\leq 1$. We denote 
\begin{align*}
\hat n_1=\mu_N([\hat a_1,\hat b_1]),\quad \hat n_1^*=\mu^*([\hat a_1,\hat b_1]),\quad
\hat n_2=\mu_N([\hat a_2,\hat b_2]),\quad
 \hat n_2^*=\mu^*([\hat a_2,\hat b_2]),
 \end{align*}
and write $\hat n^*=(\hat n_1^*,\hat n_2^*)$ and $\hat n=(\hat n_1,\hat n_2)$. We recall that $\mu^{\hat n}=\rho_V^{\hat n}(x)\rd x$, the equilibrium measure with filling fraction $\hat n$ is given by the minimizer of the functional
\eqref{e:equilibrium_fixed_filling}. Since the equilibrium measure $\rho_V^{\hat n}$ depends on the filling fraction $\hat n$, we cannot expect the rigidity and edge universality if the filling fraction $\hat n$ fluctuates too much. 
We make the following assumption: with probability going to one
\begin{align}\label{e:fractionfluc}
\|\hat n^*-\hat n\|_{\infty}\deq \max\{|\hat n^*_1-\hat n_1|, |\hat n^*_2-\hat n_2|\}=o( N^{-2/3}).
\end{align}
\begin{remark}
Under mild assumptions, statements like \eqref{e:fractionfluc} are proven in \cite{GGG} for multi-cut discrete $\beta$-ensembles without fixed filling fraction in a more general setting. In particular, for the tiling model studied in this section, \eqref{e:fractionfluc} always holds.
\end{remark}

As proven in \cite[Proposition 5.2]{Borodin2016}, the equilibrium measure $\mu^{\hat n}$ is Lipschitz with respect to the filling fraction, i.e. if $\hat n$ and $\hat n^*$ are close enough
\begin{align}\label{e:measureLipschitz}
\cD(\mu^{\hat n}, \mu^{\hat n^*})\lesssim \|\hat n-\hat n^*\|_{\infty},
\end{align}
where $\cD(\cdot, \cdot)$ is defined in \eqref{eq_quadratic_main}.
As a consequence,  $Q_{\mu^{\hat n}}$ is Lipschitz in $\hat n$ in a neighborhood $B_\varepsilon=\{\hat n: \|\hat n-\hat n^*\|\leq \varepsilon$\} of $\hat n^*$, 
and  $c^{\hat n}, C^{\hat n}, A_1^{\hat n}, B_1^{\hat n}, A_2^{\hat n}, B_2^{\hat n}$ are all Lipschitz with respect to the filling fractions in a neighborhood of $\hat n^*$ ($c^{\hat n^*}\neq 0$ and the $A_1^{\hat n^*}, B_1^{\hat n^*}, A_2^{\hat n^*}, B_2^{\hat n^*}$ are all distinct so that the derivatives of $Q_{\mu^{\hat n}}$ with respect to these parameters do not vanish). We observe here the large deviation estimates similar to Proposition \ref{p:LDP0} insures that $B_\varepsilon^c$ has probability bounded by $e^{-c N^2}$ for some $c>0$ for both continuous and discrete models. 

Fix $i\in \qq{1,2}$. We assume that $[\hat a_i, A_i^{\hat n^*}]$ is a void region for the equilibrium measure $\mu^*$, and  $\rho_V^{*}=(1+\OO(u-A_i^{\hat n^*}))s_{A_i^{\hat n^*}}\sqrt{u-A_i^{\hat n^*}}/\pi$. In the following, we show that under assumption \ref{e:fractionfluc}, the left few particles on the cut $[\hat a_i, \hat b_i]$ have Tracy-Widom fluctuations.

We first show that if we condition on the filling fraction $\hat n$, satisfying \eqref{e:fractionfluc}, the left few particles on the cut $[\hat a_i, \hat b_i]$ have Tracy-Widom fluctuations. Thanks to \eqref{e:measureLipschitz}, $[\hat a_i, A_i^{\hat n}]$ is also a void region for the equilibrium measure $\mu^{\hat n}$, and $\rho_V^{\hat n}=(1+\OO(u-A_i^{\hat n}))s_{A_i^{\hat n}}\sqrt{u-A_i^{\hat n}}/\pi$. The following is a consequence of Theorem \ref{t:medgeuniversality}. Take any fixed $m\geq 1$ and a continuously differentiable compactly supported function $O:\bR^m\rightarrow \bR$. For any index set $\Lambda\in \bZ_{\geq 1}$ with $|\Lambda|=m$, we denote the shifted index set $\Lambda_i=\{k+n_11_{i=2}: k\in \Lambda\}$, then
\begin{align}\begin{split}\label{e:condition_exp}
&\bE_{G\beta E}\left[O\left(\left(N^{2/3}(x_k/N+2)\right)_{k\in \Lambda}\right)\right]+\OO(N^{-\chi})\\
=&\bE\left[1_{\hat n\in B_\varepsilon}\bE_{\bP_N}\left[\left.O\left(\left(N^{2/3}s_{A_i^{\hat n}}^{2/3}(\ell_k/N-A_i^{\hat n})\right)_{k\in \Lambda_i}\right)\right|\hat n\right]\right] +\OO(e^{-c N^2})\\
=&\bE\left[\bE_{\bP_N}\left[\left.O\left(\left(N^{2/3}s_{A_i^{\hat n^*}}^{2/3}(\ell_k/N-A_i^{\hat n^*})\right)_{k\in \Lambda_i}\right)\right|\hat n\right]\right]+\mathbb E[1_{\hat n\in B_\varepsilon}\OO(N^{2/3}\|\hat n-\hat n^*\|_{\infty})]+\OO(e^{-c N^2})\\
=&\bE\left[\bE_{\bP_N}\left[\left.O\left(\left(N^{2/3}s_{A_i^{\hat n^*}}^{2/3}(\ell_k/N-A_i^{\hat n^*})\right)_{k\in \Lambda^*}\right)\right|\hat n\right]\right]+o(1),
\end{split}\end{align}
where $\Lambda^*=\{k+j-1: k\in \Lambda\}$, and  $j$ is the index of the left most particle on the cut $[a_i,b_i]$, and in the third line we used that $s_{A_i^{\hat n}}$ and $A_i^{\hat n}$ are Lipschitz with respect to the filling fraction.
Combining with a duality argument as in Section \ref{s:bw}, the above argument implies, under assumption \ref{e:fractionfluc}, extreme particles or holes for lozenge tilings of a hexagon with a hole without fixed filling fraction have Tracy-Widom distribution.

\subsection{Jack Deformation of Plancherel Measure}\label{s:JM}
The last example is the Jack Deformation of Plancherel Measure. We prove that for $\beta\geq 1$, the distribution of the lengths of the first few rows in Young diagrams under the Jack deformation of Plancherel measure, appropriately scaled, converges to the Tracy-Widom $\beta$ distribution. This answers an open problem in \cite[Section 1.6]{MR3498866}.

Young diagrams of size $n$ is a collection of $n$ boxes, arranged in left-justified rows, with the row lengths in non-increasing order, We denote a Young diagram of size $n$ by listing the number of boxes in each row, $\la=(\la_1,\la_2,\cdots, \la_{\ell_{\la}})$, where $\ell_\la$ is the length of $\la$, $\la_i$ is the number of boxes in the $i$-th row of $\la$ and $\la_1\geq \la_2\cdots\geq \la_{\ell_{\la}}$.
In this section we study the following probability measure on Young diagrams of size $n$,
\begin{align}\label{e:defMn}
\bM^{(n)}(\la)=\frac{n! \theta^n}{H(\la,\theta)H'(\la,\theta)},
\end{align}
where 
\begin{align}\begin{split}\label{e:defHla}
H(\la, \theta)&=\prod_{(i,j)\in \la}\left((\la_i-j)+(\la_j'-i)\theta+1\right),\\
H'(\la,\theta)&=\prod_{(i,j)\in \la}\left((\la_i-j)+(\la_j'-i)\theta+\theta\right),
\end{split}\end{align}
and $(i,j)\in \la$ stands for the box in the $i$-th row and $j$-th column of the Young diagram $\la$, and $\la'$ is the transposed Young diagram. In the expressions \eqref{e:defHla}, $\la_i-j$ is the arm length and $\la_j'-i$ is the leg length of the box $(i,j)\in \la$.

When $\theta=1$, the measure $\bM^{(n)}$ specializes to the well-known Plancherel measure for the symmetric groups. In general, it is called the Jack deformation of Plancherel measure, because of its connection to Jack polynomials. We refer to \cite{MR2143199, MR2098845,MR2607329,MR3498866,MR2465773,MR2837720} for recent researches on the Jack measure.
We prove that the distribution of the lengths of the first few rows in Young diagrams under the Jack deformation of Plancherel measure, after proper scaling, converges to the Tracy-Widom $\beta$ distribution.
\begin{theorem}\label{t:Jackmeasure}
Fix $\beta\geq 1$, $\theta=\beta/2$ and integer $m>0$. For any $(t_1,t_2,\cdots,t_m)\in \bR^{m}$, we have
\begin{align}
\lim_{M\rightarrow\infty}\bM^{(M)}\left(\theta^{-5/6}M^{1/3}\left(\la_k/\sqrt{M}-2\sqrt{\theta}\right)\leq t_k, 1\leq k\leq m\right)=
F_{\theta}(t_1,t_2,\cdots, t_m),
\end{align}
where 
\begin{align}\label{e:defFtheta}
F_{\theta}(t_1,t_2,\cdots, t_m)\deq\lim_{N\rightarrow\infty}\bP_{G\beta E}\left(N^{2/3}\left(x_k/N-2\right)\leq t_k, 1\leq k\leq m\right).
\end{align}
\end{theorem}

In the special case $\theta=1$, Theorem \ref{t:Jackmeasure} specializes to the well-known Baik-Deift-Johansson theorem, which was first proven in \cite{MR1682248} for the first row, and generalized to the first few rows in \cite{MR1758751,MR1826414,MR1802530}. For the cases $\theta=1/2$ and $\theta=2$, the above theorem was proven in \cite{MR1845180}.

We can rewrite the quantities $H(\la, \theta)$ and $H'(\la,\theta)$ in \eqref{e:defHla} in a more explicit way, in terms of row lengths of $\la$ (see \cite[Lemma 4.1]{MR2465773}):
\begin{align}\begin{split}\label{e:simplifyH}
H(\la, \theta)&=\prod_{1\leq i<j\leq \ell_\la} \frac{\Gamma(\la_i-\la_j+(j-i-1)\theta+1)}{\Gamma(\la_i-\la_j+(j-i)\theta+1)}\prod_{i=1}^{\ell_\la}\Gamma(\la_i+(\ell_\la-i)\theta+1)\\
H'(\la, \theta)&=\prod_{1\leq i<j\leq \ell_\la} \frac{\Gamma(\la_i-\la_j+(j-i)\theta)}{\Gamma(\la_i-\la_j+(j-i+1)\theta)}\prod_{i=1}^{\ell_\la}\frac{\Gamma(\la_i+(\ell_\la-i)\theta+\theta)}{\Gamma(\theta)}
\end{split}\end{align}
For $m\geq 1$ and any Young diagram $\la=(\la_1,\la_2,\cdots, \la_{\ell_\la})$, we define
\begin{align*}
V_m(\la)=\prod_{1\leq i<j\leq m}\frac{\Gamma(\la_i-\la_j+(j-i)\theta+1)\Gamma(\la_i-\la_j+(j-i+1)\theta)}{\Gamma(\la_i-\la_j+(j-i-1)\theta+1)\Gamma(\la_i-\la_j+(j-i)\theta)}
\end{align*}
and
\begin{align*}
W_m(\la)=\prod_{i=1}^{m}\frac{\Gamma(\theta)}{\Gamma(\la_i+(m-i)\theta+1)\Gamma(\la_i+(m-i)\theta+\theta)},
\end{align*}
where by our convention, $\la_i=0$ for $i>\ell_\la$. According to the formula \eqref{e:simplifyH}, the measure $\bM^{(n)}$ as in \eqref{e:defMn} can be rewritten as 
\begin{align}\label{e:bMexp}
\bM^{(n)}(\la)=\frac{n! \theta^n}{H(\la,\theta)H'(\la,\theta)}=n!\theta^n V_{\ell_\la}(\la)W_{\ell_\la}(\la).
\end{align}

The following lemma proven in \cite[Lemma 6.6]{MR2098845} will allow us to conclude that the lengths of the Young diagrams under the measure $\bM^{(n)}$ are of size $\OO(\sqrt{n})$ with high probability.
\begin{lemma}\label{l:lengthLDP}
Suppose that $\theta>0$, then we have
\begin{align*}\begin{split}
\bM^{(n)}\left(\la_1\geq 2e\sqrt{n\theta}\right)&\leq \theta^{-1}n^24^{-e\sqrt{n\theta}},\\
\bM^{(n)}\left(\ell_\la\geq 2e\sqrt{n/\theta}\right)&\leq \theta n^2 4^{-e\sqrt{n/\theta}}.
\end{split}\end{align*}
\end{lemma}

Although that the number of rows of the Young diagrams under the measure $\bM^{(n)}$ is random, the following simple but important lemma allows us to view $\bM^{(n)}$ as a measure of Young diagrams with fixed number of rows.
\begin{lemma}\label{l:VWVW}
If $N\geq \ell_\la$, then 
\begin{align*}
V_N(\la)W_N(\la)=V_{\ell_\la}(\la)W_{\ell_\la}(\la).
\end{align*}
\end{lemma}
\begin{proof}
We may assume that $N>\ell_\la$. By our convention, $\la_i=0$ for $i>\ell_\la$. Hence,
\begin{align*}
V_N(\la)
&=
V_{\ell_\la}(\la)\prod_{1\leq i\leq \ell_\la,\atop \ell_\la<j\leq N}\frac{\Gamma(\la_i+(j-i)\theta+1)\Gamma(\la_i+(j-i+1)\theta)}{\Gamma(\la_i+(j-i-1)\theta+1)\Gamma(\la_i+(j-i)\theta)}
\prod_{\ell_\la< i<j\leq N}\frac{\Gamma((j-i)\theta+1)\Gamma((j-i+1)\theta)}{\Gamma((j-i-1)\theta+1)\Gamma((j-i)\theta)}\\
&=V_{\ell_\la}(\la)\prod_{1\leq i\leq \ell_\la}\frac{\Gamma(\la_i+(N-i)\theta+1)\Gamma(\la_i+(N-i+1)\theta)}{\Gamma(\la_i+(\ell_\la-i)\theta+1)\Gamma(\la_i+(\ell_\la-i+1)\theta)}
\prod_{\ell_\la< i\leq N}\frac{\Gamma((N-i)\theta+1)\Gamma((N-i+1)\theta)}{\Gamma(\theta)}\\
&=V_{\ell_\la}(\la)\frac{W_{\ell_\la}(\la)}{W_N(\la)}.
\end{align*}
This finishes the proof.
\end{proof}

The proof of Theorem \ref{t:Jackmeasure} uses the Poissonization and de-Poissonization scheme of \cite{MR1682248, MR1758751,MR1826414}. We consider the Poissonized Jack measure with rate $M$, 
\begin{align*}
\bP^{(M)}(\la)=e^{-M}\sum_{n=0}^{\infty}\bM^{(n)}(\la)\frac{M^n}{n!}=e^{-M}V_{\ell_\la}(\la)W_{\ell_\la}(\la)\prod_{i=1}^{\ell_\la}(\theta M)^{\la_i},
\end{align*}
on the set of all partitions, where $\bM^{(n)}(\la)=0$ if $\sum_{i}\la_i\neq n$, and we used \eqref{e:bMexp}. Thanks to the Poisson concentration of measure and Lemma \ref{l:lengthLDP}, for $M$ large enough, there exists $\fc\geq 2e\max\{\theta^{1/2},\theta^{-1/2}\}$ and $c>0$ such that 
\begin{align}\label{e:rowConcentrate}
\bP^{(M)}(\ell_\la>\fc \sqrt{M})+\bP^{(M)}( \la_1 > \fc\sqrt{M})\leq e^{-cM^{1/2}}.
\end{align}

In the following we take  $N$ large and $M=N^2/\fc^2$, and define the measure on the Young diagrams with at most $N$ rows
\begin{align*}
\tilde \bP^{(M)}(\la)=\frac{{\bf 1}_{\{\ell_\la\leq N\}}}{Z_{M}}e^{-M}\sum_{n=0}^{\infty}\bM^{(n)}(\la)\frac{M^n}{n!}.
\end{align*}
where $Z_M$ is a normalizing constant. It follows from \eqref{e:rowConcentrate}, $Z_M=1+\OO(\exp(-cM^{1/2}))$. Hence for any bounded test function $O:\bR^m\mapsto\bR$, we have
\begin{align}\label{e:tP=P}
\bE_{ \bP^{(M)}}[O(\la_1,\la_2,\cdots, \la_m)]=\bE_{\tilde \bP^{(M)}}[O(\la_1,\la_2,\cdots, \la_m)]+\OO\left(\exp(-cM^{1/2})\right).
\end{align}

In the following we show that the measure $\tilde \bP^{(M)}$ falls into the category of discrete $\beta$ ensemble. Hence, the law of the first few rows of Young diagrams under $\tilde \bP^{(M)}$, after scaling, converges to the Tracy-Widom $\beta$ distribution. For any partition $\la=(\la_1,\la_2,\cdots,\la_N)$ with at most $N$ rows, where $\la_i=0$ if $i>\ell_\la$, we define $\ell_i=\la_{N-i+1}+(i-1)\theta$, for $1\leq i\leq N$.  Thanks to Lemma \ref{l:VWVW}, we can view $\tilde\bP^{(M)}$ as a measure on the particle configurations $(\ell_1,\ell_2,\cdots, \ell_N)\in \bW^\theta_N$ with $a(N)=-1$ and $b(N)=\infty$ as in Definition \ref{Def_state_space_1cut}, and 
\begin{align}\begin{split}\label{def:PNMeixner}
\tilde\bP^{(M)}(\lambda)&=\frac{{\bf 1}_{\{\ell_\la\leq N\}}}{Z_M}e^{-M}V_N(\la)W_N(\la)\prod_{i=1}^N(\theta M)^{\la_i}\\
&=\frac{1}{Z_N}\prod_{1\leq i<j\leq N}\frac{\Gamma(\ell_i-\ell_j+1)\Gamma(\ell_i-\ell_j+\theta)}{\Gamma(\ell_i-\ell_j)\Gamma(\ell_i-\ell_j+1-\theta)}\prod_{i=1}^N\frac{(\theta M)^{\ell_i}}{\Gamma(\ell_i+1)\Gamma(\ell_i+\theta)}=:\bP_N(\ell_1,\ell_2,\cdots, \ell_N),
\end{split}\end{align}
which is a discrete $\beta$ ensemble with weight
\begin{align*}
w(x;N)=\frac{(\theta M)^{x}}{\Gamma(x+1)\Gamma(x+\theta)}.
\end{align*}
In the following, we check that the measure $\bP_N$ in \eqref{def:PNMeixner} satisfies the Assumptions  \ref{a:VN}, \ref{a:wx}, \ref{a:Hz}, \ref{a:moreVN}.
\begin{enumerate}
\item For Assumption \ref{a:VN}, we have $a(N)=-1$, $b(N)=\infty$, $\hat a=0$ and $\hat b=\infty$. Although the particle configuration is supported on an infinite interval, i.e. $[0,\infty)$. Thanks to \eqref{e:rowConcentrate}, we have
\begin{align*}
\bP_N(\ell_N\leq (1+\theta)N)> \tilde \bP^{(M)}(\la_1> N)=\OO(\exp(-cM^{1/2})).
\end{align*}
We can condition on the event that $\ell_N\leq (1+\theta)N$ as in Section \ref{s:cp}. For those particle configurations, one can check that
\begin{align*}
V_N(x)=\frac{1}{N}\ln \left(\Gamma(Nx+1)\Gamma(Nx+\theta)\right)-x\ln (\theta M), 
\end{align*}
and
\begin{align*}
V(x)=2x\ln x -2x-x\ln(\theta/\fc^2),\quad V'(x)=2\ln x-\ln (\theta/\fc^2).
\end{align*}
\item For Assumption \ref{a:wx}, we have
\begin{align*}
\psi^+_N(x)=\theta/\fc^2,\quad \psi_N^-(x)=\frac{x(x+\theta-1)}{N^2},\quad \phi^+(x)=\theta/\fc^2,\quad \phi^-(x)=x^2.
\end{align*}
\item For Assumption \ref{a:Hz}, we look for an equilibrium measure $\mu(\rd x)=\rho_V(x)\rd x$ so  that there exists $0<A<B$, such that it has a unique band $[A,B]$ and $[0,A]$ is a saturated region. Therefore, by taking derivative on both sides of \eqref{e:defFV}, 
\begin{align*}
2\theta\int_{0}^B\frac{\rho_V(y)\rd y}{x-y}=V'(x),
\end{align*}
for any $x\in [A,B]$. We notice that the total mass of $\mu$ on $[A,B]$ is $1-A/\theta$. We can solve for $\rho_V$ using the inverse Hilbert transform formula \eqref{e:invHil}, 
\begin{align}\label{e:eqmJM}
A=\theta-\frac{2\sqrt{\theta}}{\fc},\quad B=\theta+\frac{2\sqrt{\theta}}{\fc},\quad 
\rho_V(x)=\left\{
\begin{array}{cc}
\frac{1}{\theta},& 0\leq x\leq A,\\
\frac{2}{\pi \theta}\arctan\sqrt{\frac{B-x}{x-A}}, &A\leq x\leq B.
\end{array}
\right.
\end{align}
$\rho_V(x) dx$ is indeed our unique equilibrium measure as soon  as $0<A<B$, that is $\theta\ge (2/\fc)^2$.
The Stieltjes transform of the equilibrium measure $\mu$ is given by 
\begin{align*}
G_\mu(z)=\frac{1}{\theta}\ln\left(\frac{(z^2-\theta z)-z\sqrt{(z-A)(z-B)}}{2\theta/\fc^2}\right).
\end{align*}
where the branch of square root is such that $\sqrt{(z-A)(z-B)}\sim z$ as $z\rightarrow \infty$. The quantities $R_\mu(z)$ and $Q_\mu(z)$ are given by
\begin{align*}
R_\mu(z)=z^2-\theta z,\quad Q_\mu(z)=H(z)\sqrt{(z-A)(z-B)},\quad H(z)=z.
\end{align*}
Although $H(z)$ vanishes at $z=0$, $R_\mu(z)$, $\psi_N^-(zN)$, $\phi^-(z)$ and $e^{\theta G_\mu(z)}$ all vanish at $z=0$, which cancels the zero of $H(z)$. One can check that the Proof of Theorem \ref{t:bootstrap} can be carried through without any change.  
\item For Assumption \ref{a:moreVN}, we have that $V(x)$ is analytic in a neighborhood of $[A,B]$, and 
\begin{align*}
V'_N(x)=V'(x)+\OO(1/N).
\end{align*}
\end{enumerate}
Therefore, it follows from Theorem \ref{t:edgeuniversality} and \eqref{e:tP=P}, the distribution of the lengths of the first few rows in Young diagrams under the Poissonized Jack measure converges to the Tracy-Widom $\beta$ distribution.

\begin{theorem}\label{t:PJM}
Fix $\beta\geq 1$, $\theta=\beta/2$, and an index set $\Lambda\in \bZ_{\geq 1}$ with $|\Lambda|=m$. For any continuously differentiable compactly supported function $O:\bR^m\rightarrow \bR$, there exists a small $\chi>0$, such that 
\begin{align}\label{e:PJM}
\bE_{\bP^{(M)}}\left[O\left(\theta^{-5/6}M^{1/3}\left(\la_k/\sqrt{M}-2\sqrt{\theta}\right), k\in \Lambda\right)\right]-\bE_{G\beta E}\left[O\left(N^{2/3}\left(x_k/N-2\right), k\in \Lambda\right)\right]=\OO(M^{-\chi}),
\end{align}
where $N=\fc \sqrt{M}$, provided $M$ is large enough and $\fc$ is chosen greater than $2e\max\{\theta^{1/2},\theta^{-1/2}\}$.
\end{theorem}
\begin{proof}
From \eqref{e:eqmJM}, the equilibrium measure of $\bP_N$ is 
\begin{align*}
\rho_V(x)=\frac{\fc^{1/2}}{\theta^{5/4}\pi}\left(1+\OO\left(B-x\right)\right)\sqrt{B-x},\quad x\rightarrow B-,
\end{align*}
where $B=\theta+2\sqrt{\theta}/\fc$.
Theorem \ref{t:edgeuniversality} implies that 
\begin{align*}
\bE_{\bP_N}\left[O\left(\fc^{1/3}\theta^{-5/6}N^{2/3}\left(\ell_{N-k+1}/N-B\right), k\in \Lambda\right)\right]-\bE_{G\beta E}\left[O\left(N^{2/3}\left(x_k/N-2\right), k\in \Lambda\right)\right]=\OO(M^{-\chi}).
\end{align*}
Since $N=\fc \sqrt{M}$ and $\ell_{N-k+1}=\la_{k}+(N-k+1)\theta$, and $O$ is continuous differentiable, we have
\begin{align*}
O\left(\fc^{1/3}\theta^{-5/6}N^{2/3}\left(\ell_{N-k+1}/N-B\right), k\in \Lambda\right)
=O\left(\theta^{-5/6}M^{1/3}\left(\la_k/\sqrt{M}-2\sqrt{\theta}\right), k\in \Lambda\right)+\OO(N^{-1/3}).
\end{align*}
Hence,
\begin{align}\label{e:trunckPJM}
\bE_{\tilde\bP^{(M)}}\left[O\left(\theta^{-5/6}M^{1/3}\left(\la_k/\sqrt{M}-2\sqrt{\theta}\right), k\in \Lambda\right)\right]-\bE_{G\beta E}\left[O\left(N^{2/3}\left(x_k/N-2\right), k\in \Lambda\right)\right]=\OO(M^{-\chi}).
\end{align}
\eqref{e:PJM} follows from combining \eqref{e:trunckPJM} and \eqref{e:tP=P}.
\end{proof}

In the following we prove Theorem \ref{t:Jackmeasure} using the de-Poissonization argument as in \cite{MR1682248, MR1758751,MR1826414}.
\begin{proof}[Proof of Theorem \ref{t:Jackmeasure}]
We denote $F_n$ the distribution function of $\la_1,\la_2,\cdots, \la_m$ under the Jack deformation of Plancherel measure $\bM^{(n)}$,
\begin{align*}
F_n(x_1,x_2,\cdots,x_m)=\bM^{(n)}(\la_k\leq x_k, 1\leq i\leq m).
\end{align*}
We also denote the distribution function of the Poissonized Jack deformation of the Plancherel measure $\bP^{(M)}$,
\begin{align*}
F(M; x_1,x_2,\cdots,x_m)=e^{-M}\sum_{n=0}^{\infty}\frac{M^n}{n!}F_n(x_1,x_2,\cdots,x_m).
\end{align*}
The measures $\bM^{(n)}$ can be obtained as distribution at time $n$ of a
certain random growth process of  Young diagrams  \cite[Section 1]{MR2143199}, i.e. there exist $q_{\theta}(\mu,\la)\geq 0$, such that
\begin{align*}
\bM^{(n-1)}(\mu)=\sum_{\la:\la\searrow\mu}q_{\theta}(\mu,\la)\bM^{(n)}(\la),
\end{align*}
where $\la\searrow\mu$ means that $\mu$ can be obtained from $\la$ by removing one box. This implies that 
\begin{align*}
F_{n}(x_1,x_2,\cdots,x_m)\leq F_{n-1}(x_1,x_2,\cdots,x_m).
\end{align*}
Moreover, by construction, $F_{n}(x_1,x_2,\cdots,x_m)$ and $F(M;x_1,x_2,\cdots,x_m)$ are monotone in $x_k$ for any $1\leq k\leq m$. Hence, let $M_{\pm}=M\pm4\sqrt{M\ln M}$,  \cite[Lemma 3.8]{MR1826414} (see also  \cite[Lemma 4.7]{BOO}) implies that 
\begin{align*}
F(M_+; x_1,x_2,\cdots,x_m)-\frac{C}{M^2}\leq F_M(x_1,x_2,\cdot,x_m)\leq F(M_-;x_1,x_2,\cdots,x_m)+\frac{C}{M^2}.
\end{align*}
Especially, for any $(t_1,t_2,\cdots, t_m)\in \bR^m$, by taking $x_k=2\sqrt{\theta M}+\theta^{5/6}M^{1/6}t_k$, we have
\begin{align}\begin{split}\label{e:sandwich1}
&\phantom{{}\leq{}}F(M_+; 2\sqrt{\theta M}+\theta^{5/6}M^{1/6}t_k, 1\leq k\leq m)-\frac{C}{M^2}
\leq F_M(2\sqrt{\theta M}+\theta^{5/6}M^{1/6}t_k, 1\leq k\leq m)\\
&\leq F(M_-;2\sqrt{\theta M}+\theta^{5/6}M^{1/6}t_k, 1\leq k\leq m)+\frac{C}{M^2}.
\end{split}\end{align}
For any $1\leq k\leq m$, $2\sqrt{\theta M_{\pm}}+\theta^{5/6}M_{\pm}^{1/6}t_k=2\sqrt{\theta M}+\theta^{5/6}M^{1/6}t_k+\OO(\sqrt{\ln M})$, so for any $\varepsilon>0$ and $M$ sufficiently large, 
\begin{align}\label{e:monotone}
2\sqrt{\theta M_+}+\theta^{5/6}M_+^{1/6}(t_k-\varepsilon)\leq 2\sqrt{\theta M}+\theta^{5/6}M^{1/6}t_k\leq 2\sqrt{\theta M_-}+\theta^{5/6}M_-^{1/6}(t_k+\varepsilon).
\end{align} 
Thanks to the monotonicity property of $F(M_{\pm}, x_1,x_2,\cdot, x_m)$, \eqref{e:monotone} implies that $F(M_+; 2\sqrt{\theta M_+}+\theta^{5/6}M_+^{1/6}(t_k-\varepsilon), 1\leq k\leq m)\leq F(M_+; 2\sqrt{\theta M}+\theta^{5/6}M^{1/6}t_k, 1\leq k\leq m)$
and $F(M_-;2\sqrt{\theta M}+\theta^{5/6}M^{1/6}t_k, 1\leq k\leq m)\leq F(M_-;2\sqrt{\theta M_-}+\theta^{5/6}M_-^{1/6}(t_k+\varepsilon), 1\leq k\leq m)$. Hence,
\eqref{e:sandwich1} leads to
\begin{align}\begin{split}\label{sandwich2}
&\phantom{{}\leq{}}F(M_+; 2\sqrt{\theta M_+}+\theta^{5/6}M_+^{1/6}(t_k-\varepsilon), 1\leq k\leq m)-\frac{C}{M^2}
\leq F_M(2\sqrt{\theta M}+\theta^{5/6}M^{1/6}t_k, 1\leq k\leq m)\\
&\leq F(M_-;2\sqrt{\theta M_-}+\theta^{5/6}M_-^{1/6}(t_k+\varepsilon), 1\leq k\leq m)+\frac{C}{M^2}.
\end{split}\end{align}
From this and Theorem \ref{t:PJM}, we conclude that
\begin{align*}
F_{\theta}(t_k-\varepsilon, 1\leq k\leq m)-o(1)
\leq F_M(2\sqrt{\theta M}+\theta^{5/6}M^{1/6}t_k, 1\leq k\leq m)\leq 
F_{\theta}(t_k+\varepsilon, 1\leq k\leq m)+o(1)
\end{align*}
as $M\rightarrow \infty$, where $F_\theta$ is as defined in \eqref{e:defFtheta}. Since $\varepsilon>0$ is arbitrary and $F_\theta$ is continuous, this finishes the proof of Theorem \ref{t:Jackmeasure}.
\end{proof}

\appendix

\section{Hilbert Transform on a finite interval $[a,b]$.}
In this section, we collects some useful facts of the Hilbert transform on a finite interval $[a,b]$. A comprehensive study of this topic can be found in \cite[Chapter 4.3]{MR809184}.

It follows by directly computation:
\begin{align*}
P.V.\int_a^b \frac{1}{\sqrt{(y-a)(b-y)}}\frac{\rd y}{x-y}=
\left\{
\begin{array}{cc}
0 & x\in [a,b],\\
\frac{\pi}{\sqrt{(x-a)(x-b)}} & x\notin [a,b].
\end{array}
\right.
\end{align*}
and
\begin{align*}
\int_a^b \frac{\sqrt{(y-a)(b-y)}\rd y}{x-y}=
\left\{
\begin{array}{cc}
\pi \left(x-\frac{a+b}{2}\right)& x\in [a,b],\\
\pi \left(x-\frac{a+b}{2}\right)-\pi\sqrt{(x-a)(x-b)} & x\notin [a,b],
\end{array}
\right.
\end{align*}
where the branch of $\sqrt{(x-a)(x-b)}$ is chosen such that $\sqrt{(x-a)(x-b)}\sim x$, as $x\rightarrow \infty$.

Similar to Hilbert transform, there is an inverse Hilbert transform formula for Hilbert transform on finite intervals: if $g$ is the Hilbert transform of $f$,
\begin{align*}
P.V.\int_a^b \frac{f(y)}{x-y} \rd y =g(x),\quad x\in [a,b],
\end{align*}
then $f$ is given by the inverse Hilbert transform,
\begin{align}\label{e:invHil}
f(x)=\frac{1}{\pi^2\sqrt{(x-a)(b-x)}} \int_a^b \frac{\sqrt{(y-a)(b-y)}}{y-x}g(y)\rd y +\frac{C}{\pi\sqrt{(x-a)(b-x)}}, \quad x\in (a,b),
\end{align}
where 
\begin{align*}
C=\int_{a}^b f(x)\rd x.
\end{align*}

\section{Rigidity and Proof of Proposition \ref{prop:rig} and \ref{prop:rig2}}\label{ap:proofrig}
In this section we prove Proposition \ref{prop:rig} and \ref{prop:rig2}. The proofs follow a similar argument to \cite[Lemma B.1]{MR2639734}, which expresses an integral in terms of the Stieltjes transform using the Helffer-Sj{\"o}strand functional calculus.

\begin{proof}[Proof of Proposition \ref{prop:rig}]
Let $S(x+\ri y)= G_N(x+\ri y)-G_\mu(x+\ri y)$. We take a cutoff function $\chi: \bR\rightarrow \bR$, such that $\chi(y)=1$ on $[-\eta, \eta]$, $\chi(y)$ vanishes outside $[-2\eta, 2\eta]$, and $\chi'(y)=\OO(1/\eta)$.
By the Helffer-Sj{\"o}strand formula we have,
\begin{align}\begin{split}\label{e:eigHF}
\left|\int_{-\infty}^{\infty} f(x)( \rd  \mu_N(x)-\rd\mu(x))\right|
&\leq
C\int_{\bC_+} ( |f(x)|+|y||f'(x)|)|\chi'(y)||S(x+\ri y)|\rd x\rd y\\
&+C\left|\int_{\bC_+} y\chi(y)f''(x)\Im[S(x+\ri y)]\rd x\rd y\right|.
\end{split}\end{align} 
The first term in \eqref{e:eigHF} is nonzero, only if $x\in I$ and $y\in [\eta, 2\eta]$, which implies $x+\ri y\in \cD_{M,r_0}$. Thanks to the rigidity estimate on $\cD_{M,r_0}$, we have $|S(x+\ri y)|\leq \tilde \eta/y$. The first term is easily bounded by
\beq\label{e:fterm}
\int_{\bC_+} ( |f(x)|+|y||f'(x)|)|\chi'(y)||S(x+\ri y)|\rd x\rd y
\lesssim \tilde \eta.
\eeq 
For the second term in \eqref{e:eigHF}, we define $\eta_0\deq\tilde \eta/\sqrt{\kappa_I}$ in the bulk case; or $\eta_0\deq\tilde \eta^{2/3}$ in the edge case. From our assumptions, we have $\eta\geq \eta_0$. We split the second term into two integrals $\{x+\ri y\in \bC_+: x\in I, y\geq \eta_0\}\cup \{x+\ri y\in \bC_+: x\in I, y\leq \eta_0\}$. 
For the integral over the first region we integrate by part, 
\begin{align}\begin{split}\label{exp3}
&\int_{x\in I}\int_{y\geq \eta_0} y\chi(y)f''(x)\Im[S(x+\ri y)]\rd x\rd y 
=
-\int_{x\in I} f'(x)\eta_0 \Re[S(x+\eta_0)]\rd x\\
-&\int_{x\in I}\int_{y\geq \eta_0} f'(x)\del_y(y\chi(y))\Re[S(x+\ri y)]\rd x\rd y.
\end{split}\end{align}
By the definition of $\eta_0$, the rigidity estimates hold for $y\geq \eta_0$. Hence, both terms are easily estimated by $C  \tilde \eta\ln(N) $.

On the region $\{x+\ri y\in \bC_+: x\in I,y\leq \eta_0\}$, let $\kappa_x\deq \dist(x,\{A,B\})$. In the bulk case, by our assumption, we have $\kappa_x\asymp \kappa_I$, and
\begin{align}\label{e:Gbehaviorbulk1}
|\Im[G_{\mu}(x+\ri y)]|\asymp \sqrt{\kappa_x+y}\lesssim \sqrt{\kappa_I+\eta_0}\lesssim \sqrt{\kappa_I}.
\end{align} 
In the edge case, if $x\in [A,B]$, we have
\begin{align}\label{e:Gbehaviorbulk2}
|\Im[G_{\mu}(x+\ri y)]|\asymp \sqrt{\kappa_x+y}\lesssim  \sqrt{\eta+\eta_0}\lesssim\eta^{1/2}+\tilde \eta^{1/3}.
\end{align} 
If $x\not\in [A,B]$, we have
\begin{align}\label{e:Gbehaviorbulk3}
|\Im[G_{\mu}(x+\ri y)]|\asymp y/\sqrt{\kappa_x+y}\lesssim y^{1/2}\lesssim \tilde \eta^{1/3}.
\end{align} 
It follows for the bulk case, 
\begin{align}\begin{split}\label{e:controlS01}
| \Im [ S ( x + \ri y ) ] |
\lesssim& |\Im[G_N(x+\ri y)]|+|\Im[G_\mu(x+\ri y)]|\\
\lesssim& \frac{\eta_0}{y}|\Im[G_N(x+\ri\eta_0)]|+|\Im[G_\mu(x+\ri y)]|\\
\lesssim &\frac{\eta_0}{y}|\Im[S(x+\ri\eta_0)]|+\frac{\eta_0}{y}|\Im[G_\mu(x+\ri \eta_0)]|+|\Im[G_\mu(x+\ri y)]|,
\end{split}\end{align}
where we used that $y|\Im[G_N(x+\ri y)]|$ is monotonic increasing as functions of $y$.
It follows by combining \eqref{e:Gbehaviorbulk1} and \eqref{e:controlS01} 
\begin{align*}
| \Im [ S ( x + \ri y ) ] |\lesssim \frac{\tilde \eta}{y}+\frac{\eta_0\sqrt{\kappa_I}}{y}+\sqrt{\kappa_I}\lesssim \frac{\tilde \eta}{y}.
\end{align*}
Similarly for the edge case, It follows by combining \eqref{e:Gbehaviorbulk2}, \eqref{e:Gbehaviorbulk3} and \eqref{e:controlS01}
\begin{align*}\begin{split}
| \Im [ S ( x + \ri y ) ] |
\lesssim \frac{\tilde \eta}{y}+\frac{\eta_0\eta^{1/2}}{y}.
\end{split}\end{align*}
Hence we have
\begin{align*}
\left|\int_{x\in I}\int_{y\leq \eta_0} y\chi(y)f''(x)\Im[S(x+\ri y)]\rd x\rd y\right| &\lesssim 
\left\{\begin{array}{cc}
\frac{\eta_0\tilde \eta}{\eta}\lesssim \tilde \eta, & \text{bulk case},\\
\frac{\eta_0\td \eta}{\eta}+\frac{\eta_0^2}{\eta^{1/2}}\lesssim \td \eta, & \text{edge case}.
\end{array}
\right.
\end{align*}
This finishes the proof of \eqref{e:rig}.

In the following, we prove that uniformly for any $E\in [a(N)/N, B-r_0-r/4]$: 
\begin{align}\label{e:location2}
|\#\{i: \ell_i/N\leq E\}-\#\{i: \gamma_i/N\leq E\}|\lesssim N\tilde \eta\ln N,
\end{align}
and \eqref{e:location} follows. For any $E\in [a(N)/N, B-r_0-r/4]$ with $\kappa_E=\dist(E, \{A,B\})$. We take two test functions $f_1(x)$ and $f_2(x)$ such that
\begin{enumerate}
\item if $E\leq A-\tilde \eta^{2/3}$, we take $f_1(x)=1$ on $x\leq E-\tilde\eta^{2/3}$ and vanishes on $x\geq E$; $f_2(x)=1$ on $x\leq E$ and vanishes  on $x\geq E+\tilde \eta^{2/3}$.
\item if $A-\tilde \eta^{2/3}\leq E\leq A+2\tilde \eta^{2/3}$, we take $f_1(x)=1$ on $x\leq A-2\tilde \eta^{2/3}$ and vanishes on $x\geq A-\tilde \eta^{2/3}$; $f_2(x)=1$ on $x\leq A+2\tilde\eta^{2/3}$ and vanishes on $x\geq A+4\tilde \eta^{2/3}$.
\item if $A+2\tilde \eta^{2/3}\leq E\leq B-r_0-r/4$, we take $f_1(x)=1$ on $x\leq E-\tilde\eta/\sqrt{\kappa_E/2}$ and vanishes on $x\geq E$; $f_2(x)=1$ on $x\leq E$ and vanishes  on $x\geq E+\tilde \eta/\sqrt{\kappa_E/2}$.
\end{enumerate} 
Let $I$ be the interval such that $f'(x)\neq 0$, we further assume that $f'(x)=\OO(1/|I|)$ and $f''(x)=\OO(1/|I|^2)$. One can check that $I$ satisfies the conditions in Proposition \ref{prop:rig}. The same argument as for \eqref{e:rig} gives that 
\begin{align}\label{e:location3}
\left|\int f_i(x)\left(\rd \mu_N(x)-\rd \mu(x)\right)\right|\lesssim
\td \eta \ln N,\quad i=1,2.
\end{align}
By our choices of the functions $f_1$ and $f_2$, they obey $f_1\leq 1_{a(N)/N\leq x\leq E}\leq f_2$. Therefore by \eqref{e:location3},
\begin{align*}\begin{split}
\frac{\#\{i: \ell_i/N\leq E\}}{N}\leq \int f_2(x)\rd\mu_N(x)
=\int_{x\leq E} f_2(x)\rd\mu(x)+\int_{x\geq E} f_2(x)\rd\mu(x)+\OO(\tilde \eta \ln N)\\
=\frac{\#\{i: \gamma_i/N\leq E\}}{N}+\int_{x\geq E} f_2(x)\rd\mu(x)+\OO(\tilde \eta \ln N).
\end{split}\end{align*}
For $x\in [A,B-r]$, $\rho_V(x)\asymp \sqrt{\dist(x, \{A,B\})}$. We have $\int_{x\geq E} f_2(x)\rd\mu(x)=\OO(\tilde \eta)$. Hence,  $\#\{i: \ell_i/N\leq E\}\leq \#\{i: \gamma_i\leq E\}+\OO(N\tilde \eta \ln N)$. The lower bound follows the same argument. 
\end{proof}

\begin{proof}[Proof of Proposition \ref{prop:rig2}]
Let $S(x+\ri y)= G_N(x+\ri y)-G_\mu(x+\ri y)$. Without loss of generality, we assume $\cal M$ contains a distance $\tau$ neighborhood of $[\hat a,\hat b]$, where $\tau\gtrsim 1$. We take a cutoff function $\chi: \bR\rightarrow \bR$, such that $\chi(y)=1$ on $[-\tau/2, \tau/2]$, $\chi(y)$ vanishes outside $[-\tau, \tau]$, and $\chi'(y)=\OO(1)$.
By the Helffer-Sj{\"o}strand formula we have,
\begin{align}\begin{split}\label{e:eigHF2}
\left|\int_{-\infty}^{\infty} f(x)( \rd  \mu_N(x)-\rd\mu(x))\right|
&\leq
C\int_{\bC_+} ( |f(x)|+|y||f'(x)|)|\chi'(y)||S(x+\ri y)|\rd x\rd y\\
&+C\left|\int_{\bC_+} y\chi(y)f''(x)\Im[S(x+\ri y)]\rd x\rd y\right|.
\end{split}\end{align} 
The first term in \eqref{e:eigHF2} is nonzero, only if $y\in [\tau/2, \tau]$, which implies $x+\ri y\in \cD_*$. Thanks to the rigidity estimate on $\cD_*$, we have $|S(x+\ri y)|\leq \tilde \eta/y$. The first term is easily bounded by
\beq\label{e:fterm2}
\int_{\bC_+} ( |f(x)|+|y||f'(x)|)|\chi'(y)||S(x+\ri y)|\rd x\rd y
\lesssim \tilde \eta.
\eeq 
For the second term in \eqref{e:eigHF2}, we denote $\kappa_x=\dist(x,\{A,B\})$, and we split the integral domain of the second term into three domains: $\cD'\cup \cD''\cup \cD'''$, where
\begin{align*}\begin{split}
\cD'&=\cD_*,\\
 \cD''&=\{x+\ri y\in \bC_+: x\in [A,B], y\leq N^{-(1-\fa)/2}/\sqrt{\kappa_x+y}\},\\
 \cD'''&=\{x+\ri y\in \bC_+: x\not\in[A,B], y\leq N^{-(1-\fa)/3}\}.
\end{split}\end{align*}
On $\cD'$, we have $|S(x+\ri y)|\leq \tilde\eta/y$, and
\begin{align*}
\left|\int_{\cD'} y\chi(y)f''(x)\Im[S(x+\ri y)]\rd x\rd y\right|\leq C\int_{\cD'}\tilde\eta\rd x\rd y\lesssim \tilde\eta.
\end{align*}
On $\cD''$, we have $x\in [A,B]$, and 
\begin{align}\label{e:controlS0}
\min\{|\Im[G_{\mu}(x+\ri y)]|, |\Im[G^{\dual}_\mu(x+\ri y)]|\}\asymp \sqrt{\kappa_x+y}.
\end{align}
We denote $\eta_x=N^{-(1-\fa)/2}/\sqrt{\kappa_x+\eta_x}$, it follows, 
\begin{align}\begin{split}\label{e:controlS1}
| \Im [ S ( x + \ri y ) ] |
\lesssim& |\Im[G_N(x+\ri y)]|+|\Im[G_\mu(x+\ri y)]|\\
\lesssim& \frac{\eta_x}{y}|\Im[G_N(x+\ri\eta_x)]|+|\Im[G_\mu(x+\ri y)]|\\
\lesssim &\frac{\eta_x}{y}|\Im[S(x+\ri\eta_x)]|+\frac{\eta_x}{y}|\Im[G_\mu(x+\ri \eta_x)]|+|\Im[G_\mu(x+\ri y)]|,
\end{split}\end{align}
where we used that $y|\Im[G_N(x+\ri y)]|$ is monotonic increasing as a function of $y$.
Thanks to the relation \eqref{e:Gdual}, there is another way to estimate $| \Im [ S ( x + \ri y ) ] |$,
\begin{align}\begin{split}\label{e:controlS2}
| \Im [ S ( x + \ri y ) ] |
\lesssim& |\Im[G^{\dual}_N(x+\ri y)]|+|\Im[G^{\dual}_\mu(x+\ri y)]|+\frac{1}{Ny}\\
\lesssim& \frac{\eta_x}{y}|\Im[G^{\dual}_N(x+\ri\eta_x)]|+|\Im[G^{\dual}_\mu(x+\ri y)]|+\frac{1}{Ny}\\
\lesssim &\frac{\eta_x}{y}|\Im[S(x+\ri\eta_x)]|+\frac{\eta_x}{y}|\Im[G^{\dual}_\mu(x+\ri \eta_x)]|+|\Im[G^{\dual}_\mu(x+\ri y)]|+\frac{1}{Ny}.
\end{split}\end{align}
Therefore, it follows from combining \eqref{e:controlS0}, \eqref{e:controlS1} and \eqref{e:controlS2}, 
\begin{align*}
| \Im [ S ( x + \ri y ) ] |\lesssim \frac{\tilde\eta}{y}
+\frac{\eta_x\sqrt{\kappa_x+y}}{y}+\frac{1}{Ny}
\lesssim\frac{\tilde\eta}{y}
+\frac{N^{-(1-\fa)/2}}{y}+\frac{1}{Ny},
\end{align*}
where we used that $|\Im[S(x+\ri\eta_x)]|\lesssim \tilde\eta/\eta_x$. The integral over $\cD''$ is bounded
\begin{align*}
\left|\int_{\cD''} y\chi(y)f''(x)\Im[S(x+\ri y)]\rd x\rd y\right|\leq C\left(\int_{\cD''}\tilde\eta+N^{-(1-\fa)/2}\rd x\rd y\right)\lesssim \tilde\eta+N^{-(1-\fa)}\lesssim \tilde\eta.
\end{align*}
On $\cD'''$, we denote $\eta_x=N^{-(1-\fa)/3}$. For $x\not\in [A,B]$, we have
\begin{align}\label{e:controlS4}
\min\{|\Im[G_{\mu}(x+\ri y)]|, |\Im[G^{\dual}_\mu(x+\ri y)]|\}\asymp \frac{y}{\sqrt{\kappa_x+y}}.
\end{align}
The same argument as above, for $x+\ri y\in \cD'''$, we have
\begin{align*}
| \Im [ S ( x + \ri y ) ] |\lesssim \frac{\tilde\eta}{y}
+\frac{\eta_x}{\sqrt{\kappa_x+y}}+\frac{1}{Ny}
\lesssim\frac{\tilde\eta}{y}
+\frac{N^{-(1-\fa)/3}}{\sqrt{\kappa_x+y}}+\frac{1}{Ny},
\end{align*}
and 
\begin{align*}
\left|\int_{\cD'''} y\chi(y)f''(x)\Im[S(x+\ri y)]\rd x\rd y\right|\leq C\left(\int_{\cD'''}\tilde\eta+\frac{yN^{-(1-\fa)/3}}{\sqrt{\kappa_x+y}}\rd x\rd y\right)\lesssim \tilde\eta+N^{-(1-\fa)}\lesssim \tilde\eta.
\end{align*}
This finishes the proof of \eqref{e:rigO1}.

\end{proof}

\section{Proof of Propositions \ref{prop:abestimate}, \ref{prop:bulkeq}, \ref{prop:alphabetaloc} and \ref{prop:edgeeq}}
\begin{proof}[Rigidity and Proof of Proposition \ref{prop:abestimate}]
Let $J\deq [a,\alpha]$. We prove the bulk case:  If $\rho_W(x)\equiv 0$ on $[a,\alpha]$, then $\alpha-a=\OO(\ln N M/N\sqrt{\kappa_I})$. If $\rho_W(x)\equiv N/L\theta$ on $[a,\alpha]$, then $\alpha-a=\OO(\ln N M/N)$. The rest statements follow from the same argument.

If $\rho_W(x)\equiv 0$ on $J$, and $\alpha-a\gtrsim \ln N M/N\sqrt{\kappa_I}$, we can take a bump function $\chi(x)$ supported on $J$, with $\chi'(x)\asymp 1/|J|$ and $\chi''(x)\asymp 1/|J|^2$. It follows from Proposition \ref{prop:ldb}, with high probability with respect to $\bP_L$
\begin{align}\label{e:rig1}
\left|\int \chi(x)(\rd \mu_L(x)-\rho_W(x)\rd x)\right|=\int \chi(x) \rd\mu_L(x)\leq C\sqrt{\frac{\ln N}{L}}.
\end{align}
Moreover, it follows from our definition $\cG_{M, L, K}$ of the set of ``good" boundary conditions, and proposition \ref{prop:rig}, with high probability with respect to $\bP_L$,
\begin{align}\label{e:rig2}
\left|\int \chi(x)(\rd \mu_N(x)-\rho_V(x)\rd x)\right|\lesssim \frac{\ln N M}{N}.
\end{align}
Since $\chi(x)$ is supported on $J\subset I$, 
\begin{align*}
\int \chi(x)\rd\mu_N(x)=\int_I \chi(x)\rd\mu_N(x)=\frac{L}{N}\int \chi(x)\rd\mu_L(x).
\end{align*}
By comparing\eqref{e:rig1} and \eqref{e:rig2}, it follows
\begin{align}\label{e:rig3}
\int \chi(x) \rho_V(x)\rd x\leq \frac{C\ln N M}{N},
\end{align}
provided that $M^2\gg L$. On $[a,\alpha]$, $\rho_V(x)\asymp \sqrt{\kappa_I}$.  We can take the bump function $\chi(x)$, such that $\int \chi(x) \rho_V(x)\rd x\asymp (\alpha-a)\sqrt{\kappa_I}$. It follows that $\alpha-a\lesssim \ln NM/N\sqrt{\kappa_I}$.

If $\rho_W(x)\equiv N/L\theta$ on $[a,\alpha]$, and $\alpha-a\gtrsim \ln NM/N$, we can take a bump function $\chi(x)$ supported on $J$, with $\chi'(x)\asymp 1/|J|$ and $\chi''(x)\asymp 1/|J|^2$. It follows from Proposition \ref{prop:ldb}, with high probability with respect to $\bP_L$
\begin{align}\label{e:rig4}
\left|\int \chi(x)(\rd \mu_L(x)-\rho_W(x)\rd x)\right|\leq C\sqrt{\frac{\ln N}{L}}.
\end{align}
Moreover, it follows from our definition $\cG_{M,L,K}$ of the set of ``good" boundary conditions, and proposition \ref{prop:rig}, with high probability with respect to $\bP_L$,
\begin{align}\label{e:rig5}
\left|\int \chi(x)(\rd \mu_N(x)-\rho_V(x)\rd x)\right|=\left|\int \chi(x)\left(\frac{L}{N}\rd \mu_L(x)-\rho_V(x)\rd x\right)\right|\lesssim \frac{\ln N M}{N}.
\end{align}
By taking difference of \eqref{e:rig4} and \eqref{e:rig5}, it follows
\begin{align}\label{e:rig6}
\frac{L}{N}\int \chi(x) \rho_W(x)\rd x-\int \chi(x)\rho_V(x)\rd x
=\int\chi(x)\left(\theta^{-1}-\rho_V(x)\right)\rd x\leq  \frac{C\ln N M}{N},
\end{align}
provided $M^2\gg L$. On $J$, there exists a constant $c>0$ such that $\theta^{-1}-\rho_V(x)>c$. We can take the bump function $\chi(x)$, such that $\int \chi(x) \left(\theta^{-1}-\rho_V(x)\right)\rd x\asymp (\alpha-a)$. It follows that $\alpha-a\asymp \ln NM/N$. This finishes the proof.
\end{proof}

\begin{proof}[Proof of Proposition \ref{prop:bulkeq}]

From Proposition \ref{prop:Iproperty}, $\kappa_I\gtrsim (L/N)^{2/3}$,  $|I|\asymp L/N\sqrt{\kappa_I}$ and $\rho_V(x)\asymp \sqrt{\kappa_I}$. We recall $\alpha, \beta$ from Proposition \ref{prop:equilibriumstructure} and the relevant estimates from Proposition \ref{prop:abestimate}.
We take derivative of both sides of \eqref{e:expW} at $x\in [\alpha, \beta]$,
\begin{align}\label{e:expW4}
W'(x)=\frac{2\theta N}{L}\left(P.V.\int_\alpha^\beta\frac{\rho_V(y)}{x-y}\rd y+\int_a^{\alpha}\frac{\rho_V(y)}{x-y}\rd y+\int_\beta^b\frac{\rho_V(y)}{x-y}\rd y+\int_{I^c}\frac{\rd\mu(y)-\rd \mu_N(y)}{x-y}\right),
\end{align}
where we have $F_V(x)=0$, since $x\in [\alpha, \beta]\subset [A,B]$. It follows by taking 
derivative of both sides of \eqref{e:expW2} at $x\in[\alpha, \beta]$
\begin{align}\label{e:expW5}
W'(x)=2\theta\left( P.V.\int_\alpha^\beta \frac{\rho_W(y)}{x-y}\rd y+ \int_{a}^\alpha\frac{\rho_W(y)}{x-y}\rd y+ \int_{\beta}^b\frac{\rho_W(y)}{x-y}\rd y\right),
\end{align}
where we have $F_W(x)=0$ on $[\alpha,\beta]$. By taking difference of \eqref{e:expW4} and \eqref{e:expW5}, we have
\begin{align*}\begin{split}
P.V.&\int_{\alpha}^\beta \frac{\rho_W(y)-\frac{N}{L}\rho_V(y)}{x-y}\rd y
=\frac{N}{L}\int_{I^c}\frac{\rd\mu(y)-\rd \mu_N(y)}{x-y}\\
-&\int_{a}^\alpha\frac{\rho_W(y)-\frac{N}{L}\rho_V(y)}{x-y}\rd y-\int_{\beta}^b\frac{\rho_W(y)-\frac{N}{L}\rho_V(y)}{x-y}\rd y.
\end{split}\end{align*}
Thanks to the inverse Hilbert transform \ref{e:invHil}, we get 
\begin{align}\begin{split}\label{e:defrW}
&\rho_W(x)-\frac{N}{L}\rho_V(x)=\frac{1}{\pi\sqrt{(x-\alpha)(\beta-x)}}\left(\frac{N}{L}\int_{I^c}h(y,x)\left(\rd\mu(y)-\rd\mu_N(y)\right)\right.-\\
-&\left.\int_{a}^\alpha h(y,x)\left(\rho_W(y)-\frac{N}{L}\rho_V(y)\right)\rd y-\int_{\beta}^b h(y,x)\left(\rho_W(y)-\frac{N}{L}\rho_V(y)\right)\rd y\right),
\end{split}\end{align}
for $x\in [\alpha, \beta]$, where 
\begin{align*}
h(y,x)=\frac{\sqrt{(y-\alpha)(y-\beta)}}{y-x}-1,
\end{align*}
and we choose the branch such that $\sqrt{(y-\alpha)(y-\beta)}\sim y$ as $y\rightarrow \infty$.

For the first term in \eqref{e:defrW}, it breaks into two terms,
\begin{align}\begin{split}\label{e:inth0}
\int_{b}^{\infty}h(y,x)\left(\rd\mu(y)-\rd\mu_N(y)\right)
+\int_{-\infty}^ah(y,x)\left(\rd\mu(y)-\rd\mu_N(y)\right).
\end{split}\end{align}
We estimate the first term in \eqref{e:inth0},
and the second part can be estimated similarly. 
We recall the construction from Remark \ref{r:decompint}. Notice that $\dist(b, \{A,B\})\asymp \kappa_I\gtrsim (L/N)^{2/3}$, we construct bump functions $\chi_1,\chi_2,\cdots, \chi_{n'}$ with scale $\tilde \eta=M/N$. $\chi_k(y)$ is supported on interval $I_k$, with $|I_k|\asymp 2^kM/N\sqrt{\kappa_I}$ and $\dist(I_k, \beta)\asymp 2^k M/N\sqrt{\kappa_I}$. 
\begin{align}\begin{split}\label{e:inth1}
\int_{y\geq b}h(y,x)\left(\rd\mu(y)-\rd\mu_N(y)\right)
&=\int_{b_0}^{b_1} (1-\chi_1(y))h(y,x)\left(\rd\mu(y)-\rd\mu_N(y)\right)\\
&+\sum_{k=1}^{n'}\int_{y\geq b}\chi_k(y)h(y,x)\left(\rd \mu(y)-\rd\mu_N(y)\right).
\end{split}\end{align}
Notice $1-\chi_1$ is supported on $[b_0,b_1]$ where $b_0=b$ and $b_1=b+\OO(M/N\sqrt{\kappa_I})$. For any $y\in [b_0,b_1]$, since $b-\beta\lesssim M/N\sqrt{\kappa_I}$, $y-\beta\lesssim M/N\sqrt{\kappa_I}$. Also, by our assumption $b-x\asymp |I|\asymp L/N\sqrt{\kappa_I}$, we have $y-x\asymp y-\alpha\asymp |I|\asymp L/N\sqrt{\kappa_I}$. It follows that $h(y,x)\lesssim (M/L)^{1/2}$. Moreover, the total mass of $\mu$ and $\mu_N$ on this interval is bounded by $\OO(\ln NM/N)$. Therefore,
\begin{align}\begin{split}\label{e:inth2}
\left|\int_{b_0}^{b_1} h(y,x)(1-\chi_1(y))\left(\rd\mu(y)-\rd\mu_N(y)\right)\right|
&\lesssim \frac{\ln N M}{N}\left(\frac{M}{L}\right)^{1/2}=\left(\frac{M}{L}\right)^{3/2}\frac{\ln N L}{N}.
\end{split}\end{align}
For $k\geq 1$, there are two cases, $|I_k|\leq |I|$ or $|I_k|\geq |I|$.
If $|I_k|\leq |I|$, since by our assumption $b-x\asymp |I|\asymp L/N\sqrt{\kappa_I}$, we have $y-x\asymp y-\alpha\asymp |I|$, $y-\beta\asymp |I_k|$, and $I_k\subset[b, B-r_0]$.
As a function of $y$, $h(y,x)$ and its derivatives satisfy
$$
|h(y,x)|\lesssim |I_k|^{1/2}|I|^{-1/2},\quad |\del_y h(y,x)|\lesssim|I_k|^{-1/2}|I|^{-1/2},\quad |\del_y^2 h(y,x)|\lesssim |I_k|^{-3/2}|I|^{-1/2}.
$$
Let $h_k(y,x)=(|I|/|I_k|)^{1/2}h(y,x)$,
then as a function of $y$, $h_k(y,x)=\OO(1)$, $\del_yh_k(y,x)=\OO(|I_k|^{-1})$ and $\del_y^{2}h_k(y,x)=\OO(|I_k|^{-2})$.
By Proposition \ref{prop:rig}, we have
\begin{align}\begin{split}
\sum_{k: |I_k|\leq |I|}\left|\int h(y,x)\chi_k(y)\left(\rd\mu(y)-\rd\mu_N(y)\right)\right|
&=\sum_{k: |I_k|\leq |I|}\left(\frac{|I_k|}{|I|}\right)^{1/2}\left|\int h_k(y,x)\left(\rd\mu(y)-\rd\mu_N(y)\right)\right|\\
&\lesssim \sum_{k: |I_k|\leq |I|}\left(\frac{|I_k|}{|I|}\right)^{1/2}\frac{\ln N M}{N}
\lesssim \frac{\ln N M}{N}.
\label{e:inth3}
\end{split}\end{align}
For $|I_k|\geq |I|$, we have $y-x\asymp y-\alpha\asymp y-\beta\asymp |I_k|$, and 
\begin{align}\begin{split}\label{e:inth3.5}
&\sum_{k: |I_k|\geq |I|}\int h(y,x)\chi_k(y)\left(\rd\mu(y)-\rd\mu_N(y)\right)\\
=&\sum_{k: |I_k|\geq |I|}\int (h(y,x)-1)\chi_k(y)\left(\rd\mu(y)-\rd\mu_N(y)\right)
+\sum_{k: |I_k|\geq |I|}\int \chi_k(y)\left(\rd\mu(y)-\rd\mu_N(y)\right)\\
=&\sum_{k: |I_k|\geq |I|}\int (h(y,x)-1)\chi_k(y)\left(\rd\mu(y)-\rd\mu_N(y)\right)
+\OO\left( \frac{\ln N M}{N}\right).
\end{split}\end{align}
As a function of $y$, $h(y,x)-1$ and its derivatives satisfy,
\begin{align*}
|h(y,x)-1|\lesssim |I|/|I_k|,\quad |\del_y (h(y,x)-1)|\lesssim|I|/|I_k|^2,\quad |\del_y^2 (h(y,x)-1)|\lesssim |I|/|I_k|^3.
\end{align*}
Let $h_k(y,x)=(|I_k|/|I|)(h(y,x)-1)$,
then as a function of $y$, $h_k(y,x)=\OO(1)$, $\del_yh_k(y,x)=\OO(|I_k|^{-1})$ and $\del_y^{2}h_k(y,x)=\OO(|I_k|^{-2})$. Moreover, either $I_k\subset [b, B-r_0]$, or $I_k$ is not completely contained in $[b,B-r_0]$, in which case $|I_k|\gtrsim 1$. By Proposition \ref{prop:rig} and \ref{prop:rig2}, we have
\begin{align}\begin{split}\label{e:inth4}
&\sum_{k: |I_k|\geq |I|}\left|\int (h(y,x)-1)\chi_k(y)\left(\rd\mu(y)-\rd\mu_N(y)\right)\right|
=\sum_{k: |I_k|\geq |I|}\frac{|I|}{|I_k|}\left|\int h_k(y,x)\left(\rd\mu(y)-\rd\mu_N(y)\right)\right|\\
\lesssim &\sum_{k: |I_k|\leq |I|}\frac{|I|}{|I_k|}\frac{\ln N M}{N}
\lesssim \frac{\ln N M}{N}.
\end{split}\end{align}

It follows by combining \eqref{e:inth1}, \eqref{e:inth2}, \eqref{e:inth3} \eqref{e:inth3.5} and \eqref{e:inth4}, 
\begin{align}\label{e:inth5}
\left|\int_{y\geq b}h(y,x)\left(\rd\mu(y)-\rd\mu_N(y)\right)\right|
\lesssim \frac{\ln N M}{N}.
\end{align}
Similarly, for the second term in \eqref{e:inth0}, 
\begin{align}\label{e:inth6}
\left|\int_{y\leq a}h(y,x)\left(\rd\mu(y)-\rd\mu_N(y)\right)\right|
\lesssim \frac{\ln N M}{N}.
\end{align}

For the second term in \eqref{e:defrW}, there are two cases. If $\rho_W(y)\equiv 0$ on $[a,\alpha]$, then $\rho_V(y)\asymp \sqrt{\kappa_I}$, $\alpha-a=\OO(\ln N M/N\sqrt{\kappa_I})$, $x-y\asymp \beta-y\asymp L/N\sqrt{\kappa_I}$, $h(y,x)\asymp (N\sqrt{\kappa_I}/L)^{1/2}\sqrt{y-\alpha}$ and 
\begin{align*}
\left|\int_{a}^\alpha h(y,x)\left(\rho_W(y)-\frac{N}{L}\rho_V(y)\right)\rd y\right|
\lesssim\left(\frac{N\sqrt{\kappa_I}}{L}\right)^{3/2}
\int_a^{\alpha} \sqrt{y-\alpha}\rd y\lesssim \left(\frac{\ln NM}{L}\right)^{3/2}.
\end{align*}
If $\rho_W(y)\equiv N/L\theta$ on $[a,\alpha]$, then $\alpha-a=\OO(\ln N M/N)$, $\rho_W(y)-N\rho_V(y)/L\asymp N/L$, $h(y,x)\asymp (N\sqrt{\kappa_I}/L)^{1/2}\sqrt{y-\alpha}$ and 
\begin{align*}
\left|\int_{a}^\alpha h(y,x)\left(\rho_W(y)-\frac{N}{L}\rho_V(y)\right)\rd y\right|
\lesssim\left(\frac{N}{L}\right)^{3/2}\kappa_I^{1/4}
\int_a^{\alpha} \sqrt{y-\alpha}\rd y\lesssim\left(\frac{\ln NM}{L}\right)^{3/2}\kappa_I^{1/4}.
\end{align*}
In both cases, the second term in \eqref{e:defrW} is bounded,
\begin{align}\label{e:bbetabounda}
\left|\int_{a}^\alpha h(y,x)\left(\rho_W(y)-\frac{N}{L}\rho_V(y)\right)\rd y\right|
\lesssim\left(\frac{\ln N M}{L}\right)^{3/2}.
\end{align}
Similarly, for the third term in \eqref{e:defrW}, we have
\begin{align}\label{e:bbetabound}
\left|\int_{\beta}^b h(y,x)\left(\rho_W(y)-\frac{N}{L}\rho_V(y)\right)\rd y\right|
\lesssim\left(\frac{\ln N M}{L}\right)^{3/2}.
\end{align}

By combining the estimates \eqref{e:defrW}, \eqref{e:inth5},\eqref{e:inth6}, \eqref{e:bbetabounda} and \eqref{e:bbetabound} all together, for $x\in[a,b]$ such that $b-x\asymp x-a\asymp |I|\asymp L/N\sqrt{\kappa_I}$, we have
\begin{align*}\begin{split}
\left|\rho_W(x)-\frac{N}{L}\rho_V(x)\right|\lesssim\frac{\pi}{\sqrt{(x-\alpha)(\beta-x)}}\left(\frac{\ln N M}{L}\right)\lesssim \left(\frac{\ln N M}{L}\right)\frac{N\sqrt{\kappa_I}}{L}.
\end{split}\end{align*}
This finishes the proof.
\end{proof}

\begin{proof}[Proof of Proposition \ref{prop:alphabetaloc}]

We prove $|a-\alpha|=\OO\left(\ln N (M/L)^{1/3}(M/N)^{2/3}\right)$ for the left edge case. The proof for the right edge case is exactly the same. 

We recall that for the left edge case, we have $I=[a,b]$, where $a=a(N)/N$ and $b=(\ell_{L+1}-\theta)/N$. We denote $\kappa_I=\dist(b, A)$, then $\kappa_I\asymp (L/N)^{2/3}$ and $\rho_V(x)\asymp \sqrt{[x-A]_+}$.
We take derivative of both sides of \eqref{e:expW} at $x\in [\alpha, \beta]$,
\begin{align}\label{e:expW7}
W'(x)=\frac{2\theta N}{L}\left(F_V'(x)+P.V.\int_\alpha^\beta\frac{\rho_V(y)}{x-y}\rd y+\int_A^{\alpha}\frac{\rho_V(y)}{x-y}\rd y+\int_\beta^b\frac{\rho_V(y)}{x-y}\rd y+\int_{I^c}\frac{\rd\mu(y)-\rd \mu_N(y)}{x-y}\right).
\end{align}
It follows by taking 
derivative of both sides of \eqref{e:expW2} at $x\in[\alpha, \beta]$
\begin{align}\label{e:expW8}
W'(x)=2\theta\left( P.V.\int_\alpha^\beta \frac{\rho_W(y)}{x-y}\rd y+ \int_{\beta}^b\frac{\rho_W(y)}{x-y}\rd y\right),
\end{align}
where we used that $F_W(x)=0$ on $[\alpha,\beta]$. 

By taking difference of \eqref{e:expW7} and \eqref{e:expW8}, we have
\begin{align*}\begin{split}
P.V.&\int_{\alpha}^\beta \frac{\rho_W(y)-\frac{N}{L}\rho_V(y)}{x-y}\rd y
=\frac{N}{L}\int_{I^c}\frac{\rd\mu(y)-\rd \mu_N(y)}{x-y}\\
-&\int_{\beta}^b\frac{\rho_W(y)-\frac{N}{L}\rho_V(y)}{x-y}\rd y+\frac{N}{L}\int_{A}^\alpha\frac{\rho_V(y)}{x-y}\rd y+\frac{N}{L}F_V'(x).
\end{split}\end{align*}
Thanks to the inverse Hilbert transform \ref{e:invHil}, we get 
\begin{align}\begin{split}\label{e:defrW2}
&\rho_W(x)-\frac{N}{L}\rho_V(x)=\frac{1}{\pi\sqrt{(x-\alpha)(\beta-x)}}\left(\frac{N}{L}\int_{b}^{\infty}h(y,x)\left(\rd\mu(y)-\rd\mu_N(y)\right)\right.\\
-&\left.\int_{\beta}^b h(y,x)\left(\rho_W(y)-\frac{N}{L}\rho_V(y)\right)\rd y
+\frac{N}{L}\int_{A}^\alpha h(y,x)\rho_V(y)\rd y+\frac{N}{\pi L}P.V.\int_\alpha^\beta \frac{\sqrt{(y-\alpha)(\beta-y)}}{y-x}F'_V(y)\rd y\right),
\end{split}\end{align}
for $x\in [\alpha, \beta]$, where 
\begin{align*}
h(y,x)=\frac{\sqrt{(y-\alpha)(y-\beta)}}{y-x}.
\end{align*}
and we choose the branch such that $\sqrt{(y-\alpha)(y-\beta)}\sim y$ as $y\rightarrow \infty$.
At $x=\alpha$, $\rho_W(x)-N\rho_V(x)/L$ is a finite number. Therefore, by plugging $x=\alpha$ in \eqref{e:defrW2}, we get
\begin{align}\begin{split}\label{e:pluga}
0=&\frac{N}{L}\int_{b}^{\infty} \sqrt{\frac{y-\beta}{y-\alpha}}(\rd \mu(y)-\rd\mu_N(y))
-\int_\beta^b \sqrt{\frac{y-\beta}{y-\alpha}}\left(\rho_W( y)-\frac{N}{L}\rho_V(y)\right)\rd y\\
+&\frac{N}{L}\int_A^\alpha \sqrt{\frac{y-\beta}{y-\alpha}}\rho_V(y)\rd y
+\frac{N}{\pi L}\int_\alpha^\beta \sqrt{\frac{\beta-y}{y-\alpha}}F_V'(y)\rd y.
\end{split}\end{align}
We first estimate the first term in \eqref{e:pluga}. For $y\geq b$, the behaviors of the function $(y-\beta)^{1/2}/(y-\alpha)^{1/2}$ and its first and second derivatives is the same as those of the function $h(y,x)$.
Therefore, by the same argument as for \eqref{e:inth5}, we have
\begin{align}\label{e:pluga1}
\left|\frac{N}{L}\int_{b}^{\infty} \sqrt{\frac{y-\beta}{y-\alpha}}(\rd \mu(y)-\rd\mu_N(y))\right|\lesssim \frac{\ln N M}{L}.
\end{align}
For the second term in \eqref{e:pluga}, by the same argument as for \eqref{e:bbetabound}, we have
\begin{align}\label{e:pluga2}
\left|\int_\beta^b \sqrt{\frac{y-\beta}{y-\alpha}}\left(\rho_W( y)-\frac{N}{L}\rho_V(y)\right)\rd y\right|\lesssim \left(\frac{\ln N M}{L}\right)^{3/2}.
\end{align}
For the last two terms in \eqref{e:pluga}, exact one of them is nonzero. More precisely, if $A\leq \alpha$, the last term in \eqref{e:pluga} vanishes. Otherwise the third term in \eqref{e:pluga} vanishes. If $A\leq \alpha$, we have the following estimate for the third term in \eqref{e:pluga},
\begin{align}\label{e:pluga3}
\frac{N}{L}\int_A^\alpha \sqrt{\frac{y-\beta}{y-\alpha}}\rho_V(y)\rd y
\asymp \frac{N}{L}\int_{A}^\alpha \left(\frac{L}{N}\right)^{1/3}\sqrt{\frac{y-A}{\alpha-y}}\rd y\asymp \left(\frac{N}{L}\right)^{2/3}(\alpha-A),
\end{align}
where we used $\beta-y\asymp \kappa_I \asymp (L/N)^{2/3}$ and $\rho_V(y)\asymp \sqrt{[y-A]_+}$. If $\alpha\leq A$, we have the following estimate for the last term in \eqref{e:pluga},
\begin{align}\label{e:pluga4}
\frac{N}{\pi L}\int_\alpha^A \sqrt{\frac{y-\beta}{y-\alpha}}F_V'(y)\rd y
\asymp \frac{N}{L}\int_\alpha^A \left(\frac{L}{N}\right)^{1/3}\sqrt{\frac{A-y}{y-\alpha}}\rd y\asymp \left(\frac{N}{L}\right)^{2/3}(A-\alpha),
\end{align}
where we used $\beta-y\asymp \kappa_I \asymp (L/N)^{2/3}$ and $F_V'(y)\asymp \sqrt{[A-y]_+}$.

By combining \eqref{e:pluga}, \eqref{e:pluga1}, \eqref{e:pluga2}, \eqref{e:pluga3}, and \eqref{e:pluga4} all together, we get 
\begin{align*}
\left(\frac{N}{L}\right)^{2/3}|\alpha-A|\lesssim \frac{\ln N M}{L},
\end{align*}
and the statement follows by rearranging. 
\end{proof}
 
\begin{proof}[Proof of Proposition \ref{prop:edgeeq}]

We prove the estimate for the left edge case. The proof for the right edge case is exactly the same.
It follows from taking differece of \eqref{e:defrW2} and \eqref{e:pluga}, 
\begin{align}\begin{split}\label{e:defrWedge}
&\rho_W(x)-\frac{N}{L}\rho_V(x)=\frac{1}{\pi\sqrt{(x-\alpha)(\beta-x)}}\left(\frac{N}{L}\int_{b}^\infty\left( h(y,x)-\sqrt{\frac{y-\beta}{y-\alpha}}\right)\left(\rd\mu(y)-\rd\mu_N(y)\right)\right.\\
-&\left.\int_{\beta}^b \left(h(y,x)-\sqrt{\frac{y-\beta}{y-\alpha}}\right)\left(\rho_W(y)-\frac{N}{L}\rho_V(y)\right)\rd y
+\frac{N}{L}\int_{A}^\alpha \left(h(y,x)-\sqrt{\frac{y-\beta}{y-\alpha}}\right)\rho_V(y)\rd y\right.\\
+&\left.\frac{N}{\pi L}P.V.\int_\alpha^\beta \left(\frac{\sqrt{(y-\alpha)(\beta-y)}}{y-x}-\sqrt{\frac{y-\beta}{y-\alpha}}\right)F'_V(y)\rd y\right)\\
=&
\frac{1}{\pi}\sqrt{\frac{x-\alpha}{\beta-x}}\left(\frac{N}{L}\int_{b}^\infty g(y,x)\left(\rd\mu(y)-\rd\mu_N(y)\right)
-\int_{\beta}^b g(y,x)\left(\rho_W(y)-\frac{N}{L}\rho_V(y)\right)\rd y\right.\\
+&\left.\frac{N}{L}\int_{A}^\alpha g(y,x)\rho_V(y)\rd y
+\frac{ N}{\pi L}P.V.\int_\alpha^\beta \frac{\sqrt{\beta-y}}{(y-x)\sqrt{y-\alpha}}F'_V(y)\rd y\right)
,
\end{split}\end{align}
where 
\begin{align*}
g(y,x)=\frac{\sqrt{y-\beta}}{(y-x)\sqrt{y-\alpha}}.
\end{align*}

We first estimate the first term on the righthand side of \eqref{e:defrWedge}, which is similar to the argument for \eqref{e:inth5}. 
We recall the construction from Remark \ref{r:decompint}. Notice that $\dist(b, A)= \kappa_I\gtrsim (L/N)^{2/3}$, we construct bump functions $\chi_1,\chi_2,\cdots, \chi_{n'}$ with scale $\tilde \eta=M/N$, from Remark \ref{r:decompint}. $\chi_k(y)$ is supported on interval $I_k$, with $|I_k|\asymp 2^kM/N\sqrt{\kappa_I}$ and $\dist(I_k, b)\asymp 2^k M/N\sqrt{\kappa_I}$. 
\begin{align}\begin{split}\label{e:intg1}
\int_{y\geq b}g(y,x)\left(\rd\mu(y)-\rd\mu_N(y)\right)
&=\int_{b_0}^{b_1} (1-\chi_1(y))g(y,x)\left(\rd\mu(y)-\rd\mu_N(y)\right)\\
&+\sum_{k=1}^{n'}\int_{y\geq b}\chi_k(y)g(y,x)\left(\rd \mu(y)-\rd\mu_N(y)\right).
\end{split}\end{align}
Notice $1-\chi_1$ is supported on $[b_0,b_1]$ where $b_0=b$ and $b_1=b+\OO(M/N\sqrt{\kappa_I})$. For any $y\in [b_0,b_1]$, since $b-\beta\lesssim M/N\sqrt{\kappa_I}$, $y-\beta\lesssim M/N\sqrt{\kappa_I}$. Also, by our assumption $b-x\asymp\kappa_I\asymp (L/N)^{2/3}$. It follows by combining Proposition \ref{prop:alphabetaloc} that $y-x\asymp y-\alpha\asymp\kappa_I\asymp (L/N)^{2/3}$. All these estimates together impliy $g(y,x)\lesssim (M/L)^{1/2}$ for $y\in[b_0,b_1]$. Moreover, the total mass of $\mu$ and $\mu_N$ on this interval is bounded by $\OO(\ln NM/N)$. Therefore,
\begin{align}\begin{split}\label{e:intg2}
\left|\int_{b_0}^{b_1} g(y,x)(1-\chi_1(y))\left(\rd\mu(y)-\rd\mu_N(y)\right)\right|
&\lesssim \ln N\left(\frac{M}{N}\right)^{1/3}\left(\frac{M}{L}\right)^{7/6}.
\end{split}\end{align}
For $k\geq 1$, there are two cases, $|I_k|\leq \kappa_I$ or $|I_k|\geq \kappa_I$.
In the first case that $|I_k|\leq \kappa_I$, we $I_k\subset[b, B-r_0]$. Moreover,
if $y\in I_k$ with $|I_k|\leq \kappa_I$, by our assumption $b-x\asymp \kappa_I\asymp L/N\sqrt{\kappa_I}$, and Proposition \ref{prop:alphabetaloc}, we have $y-x\asymp y-\alpha\asymp\kappa_I\asymp (L/N)^{2/3}$, and $y-\beta\asymp |I_k|$.
As a function of $y$, $g(y,x)$ and its derivatives satisfy
\begin{align*}
|g(y,x)|\lesssim |I_k|^{1/2}\kappa_I^{-3/2},\quad |\del_y g(y,x)|\lesssim|I_k|^{-1/2}\kappa_I^{-3/2},\quad |\del_y^2 g(y,x)|\lesssim |I_k|^{-3/2}\kappa_I^{-3/2}.
\end{align*}
Let $g_k(y,x)=(\kappa_I^{3/2}/|I_k|^{1/2})g(y,x)$,
then as a function of $y$, $g_k(y,x)=\OO(1)$, $\del_yg_k(y,x)=\OO(|I_k|^{-1})$ and $\del_y^{2}g_k(y,x)=\OO(|I_k|^{-2})$.
By Proposition \ref{prop:rig}, we have
\begin{align}\begin{split}\label{e:intg3}
&\sum_{k: |I_k|\leq \kappa_I}\left|\int g(y,x)\chi_k(y)\left(\rd\mu(y)-\rd\mu_N(y)\right)\right|
=\sum_{k: |I_k|\leq \kappa_I}\frac{|I_k|^{1/2}}{\kappa_I^{3/2}}\left|\int g_k(y,x)\left(\rd\mu(y)-\rd\mu_N(y)\right)\right|\\
\lesssim &\sum_{k: |I_k|\leq \kappa_I}\frac{|I_k|^{1/2}}{\kappa_I^{3/2}}\frac{\ln N M}{N}
\lesssim \frac{\ln N M}{N\kappa_I}\asymp \frac{\ln N M}{N}\left(\frac{N}{L}\right)^{2/3}.
\end{split}\end{align}
For $|I_k|\geq \kappa_I$, we have $y-x\asymp y-\alpha\asymp y-\beta\asymp |I_k|$, and 
\begin{align*}
|g(y,x)|\lesssim |I_k|^{-1},\quad |\del_y g(y,x)|\lesssim|I_k|^{-2},\quad |\del_y^2 g(y,x)|\lesssim |I_k|^{-3}.
\end{align*}
Let $g_k(y,x)=|I_k|g(y,x)$,
then as a function of $y$, $g_k(y,x)=\OO(1)$, $\del_yg_k(y,x)=\OO(|I_k|^{-1})$ and $\del_y^{2}g_k(y,x)=\OO(|I_k|^{-2})$.
Moreover, either $I_k\subset [b, B-r_0]$, or $I_k$ is not completely contained in $[b,B-r_0]$, in which case $|I_k|\gtrsim 1$. By Proposition \ref{prop:rig} and \ref{prop:rig2}, we have
\begin{align}\begin{split}\label{e:intg4}
&\sum_{k: |I_k|\geq \kappa_I}\left|\int g(y,x)\chi_k(y)\left(\rd\mu(y)-\rd\mu_N(y)\right)\right|
=\sum_{k: |I_k|\geq \kappa_I}\frac{1}{|I_k|}\left|\int g_k(y,x)\left(\rd\mu(y)-\rd\mu_N(y)\right)\right|\\
\lesssim &\sum_{k: |I_k|\geq \kappa_I}\frac{\ln N M}{N|I_k|}
\lesssim \frac{\ln N M}{N\kappa_I}\asymp \frac{\ln N M}{N}\left(\frac{N}{L}\right)^{2/3}.
\end{split}\end{align}
It follows by combining \eqref{e:intg2}, \eqref{e:intg3} and \eqref{e:intg4}, and noticing $\beta-x\asymp \kappa_I\asymp (L/N)^{2/3}$, we have
\begin{align}\label{e:defrWedge1}
\left|\frac{1}{\pi}\sqrt{\frac{x-\alpha}{\beta-x}}\frac{N}{L}\int_{b}^{\infty} g(y,x)(\rd \mu(y)-\rd\mu_N(y))\right|\lesssim \frac{\ln N M}{L}\frac{N}{L}\sqrt{x-\alpha}.
\end{align}
For the second term on the righthand side of \eqref{e:defrWedge}, we have $g(y,x)\asymp \sqrt{y-\beta}/\kappa_I^{3/2}$. It follows from the same argument as for \eqref{e:bbetabound}, 
\begin{align}\label{e:defrWedge2}
\left|\frac{1}{\pi}\sqrt{\frac{x-\alpha}{\beta-x}}\int_\beta^b g(y,x)\left(\rho_W( y)-\frac{N}{L}\rho_V(y)\right)\rd y\right|\lesssim \left(\frac{\ln N M}{L}\right)^{3/2}\frac{N}{L}\sqrt{x-\alpha}.
\end{align}

For the last two terms on the righthand side of \eqref{e:defrWedge}, exact one of them is nonzero. More precisely, if $A\leq \alpha$, the last term vanishes; if $\alpha\leq A$, the third term vanishes. We first exam the first case. If $A\leq \alpha$, we can estimate the third term in \eqref{e:defrWedge},
\begin{align}\label{e:defrWedge3}
\frac{1}{\pi} \sqrt{\frac{x-\alpha}{\beta-x}}\frac{N}{L}\int_{A}^{\alpha}g(y,x)\rho_V(y)\rd y
\lesssim\frac{N}{L}\sqrt{x-\alpha} \int_{A}^{\alpha}\frac{\sqrt{\alpha-A}}{(x-y)\sqrt{\alpha-y}}\rd y,
\end{align}
where we used $\beta-x, \beta-y\asymp (L/N)^{2/3}$,  and $\rho_V(y)\asymp\sqrt{[y-A]_+}\leq \sqrt{\alpha-A}$. For the integral in \eqref{e:defrWedge3}, there are two cases,
\begin{align}\label{e:defrWedge4}
\int_{A}^{\alpha}\frac{1}{(x-y)\sqrt{\alpha-y}}\rd y\lesssim 
\left\{
\begin{array}{cc}
\frac{1}{\sqrt{x-\alpha}},  &0\leq x-\alpha\leq \alpha-A,\\
\frac{\sqrt{\alpha-A}}{x-\alpha}, & x-\alpha\geq \alpha-A.
\end{array}
\right.
\end{align}
\eqref{e:erroraA} follows from \eqref{e:defrWedge}, \eqref{e:defrWedge1}, \eqref{e:defrWedge2}, \eqref{e:defrWedge3} and \eqref{e:defrWedge4}.

If $\alpha\leq A$, we estimate the last term in \eqref{e:defrWedge}. There are two cases, if $x\geq A$, the last term on the righthand side of \eqref{e:defrWedge} is a normal integral, 
\begin{align*}
\frac{1}{\pi}\sqrt{\frac{x-\alpha}{\beta-x}}\frac{ N}{L}\int_{\alpha}^A\frac{\sqrt{\beta-y}}{(y-x)\sqrt{y-\alpha}}F_V'(y)\rd y
\lesssim \frac{N}{L}\sqrt{x-\alpha}\int_{\alpha}^A\frac{\sqrt{A-\alpha}}{(x-y)\sqrt{y-\alpha}}\rd y,
\end{align*}
where we used $\beta-x, \beta-y\asymp (L/N)^{2/3}$, $F_V'(y)\asymp \sqrt{[A-y]_+}\leq \sqrt{A-\alpha}$. It follows \eqref{e:defrWedge4}, 
\begin{align}\label{e:defrWedge4.5}
\frac{1}{\pi}\sqrt{\frac{x-\alpha}{\beta-x}}\frac{ N}{L}\int_{\alpha}^A\frac{\sqrt{\beta-y}}{(y-x)\sqrt{y-\alpha}}F_V'(y)\rd y\lesssim 
\left\{
\begin{array}{cc}
\frac{N}{L}\sqrt{A-\alpha},  &0\leq x-A\leq A-\alpha,\\
\frac{N}{L}\frac{A-\alpha}{\sqrt{x-\alpha}}, & x-A\geq A-\alpha.
\end{array}
\right.
\end{align}
%
%
If $\alpha\leq x\leq A$, we denote 
\begin{align*}
F_V'(y)=-\sqrt{A-y}S_A(y),\quad t(y)=\sqrt{\frac{\beta-y}{y-\alpha}}F_V'(y).
\end{align*}
where $S_A(y)$ is an analytic function as defined in the discussion below \eqref{e:RQ3}. Then we have
\begin{align}\label{e:defrWedge5}
P.V.\int_{\alpha}^A\frac{\sqrt{\beta-y}}{(y-x)\sqrt{y-\alpha}}F_V'(y)\rd y=\int_{\alpha}^A \frac{t(y)-t(x)}{y-x}\rd y+t(x)P.V.\int_{\alpha}^A\frac{1}{y-x}\rd y.
\end{align}
For the second term in \eqref{e:defrWedge5},
\begin{align}\label{e:defrWedge6}
t(x)P.V.\int_{\alpha}^A\frac{1}{y-x}\rd y
\asymp \left(\frac{L}{N}\right)^{1/3}\sqrt{\frac{A-x}{x-\alpha}}\ln \frac{A-\alpha}{\dist(x, \{\alpha, A\})}
\lesssim 
\left(\frac{L}{N}\right)^{1/3}\sqrt{\frac{A-\alpha}{x-\alpha}}\ln \frac{A-\alpha}{x-\alpha},
\end{align}
where we used $\beta-x\asymp (L/N)^{2/3}$. For the first term in \eqref{e:defrWedge5}, we rewrite it as the sum of two terms
\begin{align}\label{e:defrWedge7}
-\sqrt{\frac{(\beta-x)(A-x)}{x-\alpha}}\int_\alpha^A\frac{S_A(y)-S_A(x)}{y-x}\rd y-S_A(x)\int_{\alpha}^A\frac{\sqrt{\frac{(\beta-y)(A-y)}{y-\alpha}}-\sqrt{\frac{(\beta-x)(A-x)}{x-\alpha}}}{y-x}\rd y.
\end{align}
For the first term in \eqref{e:defrWedge7},
\begin{align}\label{e:defrWedge8}
-\sqrt{\frac{(\beta-x)(A-x)}{x-\alpha}}\int_\alpha^A\frac{S_A(y)-S_A(x)}{y-x}\rd y
\lesssim \left(\frac{L}{N}\right)^{1/3}\sqrt{\frac{A-x}{x-\alpha}}
\lesssim  \left(\frac{L}{N}\right)^{1/3}\sqrt{\frac{A-\alpha}{x-\alpha}},
\end{align}
where we used that $S_A$ is an analytic function, so it is Lipschitz. For the second term in \eqref{e:defrWedge7}, the integrand is bounded by 
\begin{align}\label{e:defrWedge9}
\left(\frac{N}{L}\right)^{1/3}\left(\sqrt{\frac{A-y}{y-\alpha}}+\sqrt{\frac{A-x}{x-\alpha}}\right)^{-1}
+\left(\frac{L}{N}\right)^{1/3}\frac{A-\alpha}{(x-\alpha)\sqrt{(A-y)(y-\alpha)}+(y-\alpha)\sqrt{(A-x)(x-\alpha)}},
\end{align}
where we used $\beta-\alpha\asymp \beta-x\asymp \beta-y\asymp (L/N)^{2/3}$. For the first term in \eqref{e:defrWedge9},
\begin{align}\label{e:defrWedge10}
\left(\frac{N}{L}\right)^{1/3 }\int_\alpha^A \left(\sqrt{\frac{A-y}{y-\alpha}}+\sqrt{\frac{A-x}{x-\alpha}}\right)^{-1}\rd y\lesssim \left(\frac{N}{L}\right)^{1/3 }\sqrt{(A-\alpha)(x-\alpha)}\lesssim \left(\frac{L}{N}\right)^{1/3 }\sqrt{\frac{A-\alpha}{x-\alpha}},
\end{align}
where in the last inequality, we used $x-\alpha\leq A-\alpha \lesssim (L/N)^{2/3}$ from Proposition \ref{prop:alphabetaloc}.
For the second term in \eqref{e:defrWedge9}, if $x\geq (\alpha+A)/2$, then $x-\alpha\asymp A-\alpha$, and 
\begin{align}\begin{split}\label{e:defrWedge11}
&\left(\frac{L}{N}\right)^{1/3}\int_{\alpha}^A\frac{A-\alpha}{(x-\alpha)\sqrt{(A-y)(y-\alpha)}+(y-\alpha)\sqrt{(A-x)(x-\alpha)}}\rd y\\
\leq& \left(\frac{L}{N}\right)^{1/3}\int_{\alpha}^A\frac{A-\alpha}{(x-\alpha)\sqrt{(A-y)(y-\alpha)}}\rd y
\lesssim  \left(\frac{L}{N}\right)^{1/3}.
\end{split}\end{align}
If $x\leq (\alpha+A)/2$, then $A-x\asymp A-\alpha$, and 
\begin{align}\begin{split}\label{e:defrWege12}
&\left(\frac{L}{N}\right)^{1/3}\int_{\alpha}^A\frac{A-\alpha}{(x-\alpha)\sqrt{(A-y)(y-\alpha)}+(y-\alpha)\sqrt{(A-x)(x-\alpha)}}\rd y\\
\lesssim&\frac{A-\alpha}{\sqrt{x-\alpha}} \left(\frac{L}{N}\right)^{1/3}\int_{\alpha}^A\frac{\rd y}{\sqrt{y-\alpha}\left(\sqrt{(x-\alpha)(A-y)}+\sqrt{(y-\alpha)(A-\alpha)}\right)}\\
\lesssim  &\frac{A-\alpha}{\sqrt{x-\alpha}} \left(\frac{L}{N}\right)^{1/3}\left(\int_{\alpha}^{(\alpha+A)/2}\frac{\rd y}{\sqrt{y-\alpha}\sqrt{A-\alpha}\left(\sqrt{x-\alpha}+\sqrt{y-\alpha}\right)}
+\int_{(\alpha+A)/2}^A\frac{\rd y}{(A-\alpha)^{3/2}}
\right)\\
\lesssim &\sqrt{\frac{A-\alpha}{x-\alpha}}\left(\frac{L}{N}\right)^{1/3}\ln \frac{A-\alpha}{x-\alpha}.
\end{split}\end{align}
It follows from combining the above estimates all together, for $\alpha \leq x\leq A$, 
\begin{align}\label{e:defrWedge13}
\frac{1}{\pi}\sqrt{\frac{x-\alpha}{\beta-x}}\frac{ N}{L}\int_{\alpha}^A\frac{\sqrt{\beta-y}}{(y-x)\sqrt{y-\alpha}}F_V'(y)\rd y\lesssim 
\frac{N}{L}\sqrt{A-\alpha}\ln \frac{A-\alpha}{x-\alpha}.
\end{align}
\eqref{e:erroraA} follows from \eqref{e:defrWedge}, \eqref{e:defrWedge1}, \eqref{e:defrWedge2}, \eqref{e:defrWedge4.5} and \eqref{e:defrWedge13}.This finishes the proof.
\end{proof}

\section{Level Repulsion and Proof of Proposition \ref{p:levelrep}}\label{a:levelrep}

\begin{proof}[Proof of Proposition \ref{p:levelrep}]
The continuous case, i.e. \eqref{e:cont}, is the the same statement as in \cite[Lemma D.1]{MR3253704}. We present the proof for \eqref{e:dis} when $k=L$, which is a discrete analogue of the proof of \cite[Lemma D.1]{MR3253704}. The general case when $k\in \qq{1,L}$ follows by further conditioning on $\ell_{k+1}, \ell_{k+1}, \cdots, \ell_{L}$.

We set $\ell_+\deq\ell_{L+1}-\theta$ and $\ell_-\deq\ell_+-a$ with $a\deq N^{\fa+1/3}L^{-1/3}$. By our assumption, $(\ell_{L+1}, \ell_{L+2}, \cdots, \ell_{N})\in \cR_L^{*}(\fa, r)$ is a ``good" boundary condition as in Definition \ref{def:good2}.

We decompose the configurational space according to the number of the particles in $[\ell_-, \ell_+]$, which we denote by $n$. For any large integer $D>0$ (we allow that $D=\infty$), we define the partition function 
\begin{align*}
Z_D\deq \sum_{n=0}^L \sum_{\ell_1<\ell_2<\cdots < \ell_{L-n}<\ell_-,\atop \ell_-\leq \ell_{L-n+1}<\cdots < \ell_L\leq \ell_+-a/2D} Z_L^{\loc}\bP_L^{\loc}(\ell_1, \ell_2, \cdots, \ell_L).
\end{align*}
We remark that $Z_{\infty}=Z_{L}^{\loc}$.
In the following we prove that if $a\gg DL$ and $D\gg L^2$, then there exists some constant $C$ such that 
\begin{align}\label{e:ratio}
\frac{Z_D}{Z_\infty}\geq \left(1-\frac{C}{D}\right)^{CL^2}.
\end{align}
Then it follows by taking $D=\lfloor N^{\fa}/2s\rfloor$, that if $s\ll L^{-2}$, we have
\begin{align*}
\bP_L^{\loc}(|\ell_{L+1}-\ell_L|\geq sN^{1/3}L^{-1/3})
\geq \bP_L^{\loc}(|\ell_{L+1}-\ell_L|\geq a/2D)
=\frac{Z_D}{Z_\infty}\geq \left(1-Cs L^2\right),
\end{align*}
and this finishes the proof. 

For the proof of \eqref{e:ratio}, similarly to the proof of \cite[Lemma D.1]{MR3253704}, we fix $n$ and perform a discrete change of variable which sends $\ell_{L-n+1}, \ell_{L-n+2}, \cdots, \ell_{L}\in [\ell_-, \ell_+-a/2D]$
to $\tilde\ell_{L-n+1}, \tilde\ell_{L-n+2}, \cdots, \tilde\ell_{L}\in [\ell_-, \ell_+]$
. Let $\ell_-^n=\inf\{\ell\in a(N)+ (L-n+1)\theta+\bZ_{\geq 0}: \ell\geq \ell_-\}$, then we have
\begin{align*}\begin{split}
y_1\deq \ell_{L-n+1}-\ell_-^n, \quad y_2\deq \ell_{L-n+2}-\ell_{L-n+1}-\theta,\quad \cdots, \quad y_n\deq \ell_{L}-\ell_{L-1}-\theta,\\
\tilde y_1\deq \tilde\ell_{L-n+1}-\ell_-^n, \quad \tilde y_2\deq \tilde\ell_{L-n+2}-\tilde\ell_{L-n+1}-\theta,\quad \cdots, \quad y_n\deq \tilde\ell_{L}-\tilde\ell_{L-1}-\theta,
\end{split}\end{align*}
where $y_1, y_2,\cdots, y_n, \tilde y_1, \tilde y_2,\cdots, \tilde y_n\in \bZ_{\geq 0}$. We define the transition kernel
\begin{align*}
\bK(\{\tilde\ell_i\}_{L-n+1}^L, \{\ell_i\}_{L-n+1}^L)
=\prod_{i=1}^{n}\bK_0(\tilde y_i, y_i),
\end{align*}
where 
\begin{align*}
\bK_0(\tilde y, y)=\left\{
\begin{array}{cc}
\delta_{\tilde y, y} & \text{if $y=m$, where $0\leq m\leq D-1$,}\\
\frac{m}{D-1}\delta_{\tilde y-k,  y} +\frac{D-1-m}{D-1} \delta_{\tilde y-(k-1),y}& \text{if $\tilde y=kD+m$, where $1\leq k $ and $0\leq m\leq D-1$.}
\end{array}
\right.
\end{align*}
One can easily check:
\begin{claim}\label{c:ltol}
If $\tilde\ell_{L-n+1}, \tilde\ell_{L-n+2}, \cdots, \tilde\ell_{L}\in [\ell_-, \ell_+]$ and $\bK(\{\tilde\ell_i\}_{L-n+1}^L, \{\ell_i\}_{L-n+1}^L)\neq 0$, then there exists some large constants $C$, such that 
\begin{enumerate}
\item \label{a1}$\ell_{L-n+1}, \ell_{L-n+2}, \cdots, \ell_{L}\in [\ell_-, \ell_+-a/2D]$, provided that $a\geq CDL$;
\item \label{a2}$\sum_{\tilde\ell_{L-n+1}, \tilde\ell_{L-n+2}, \cdots, \tilde\ell_{L}\in [\ell_-, \ell_+]}\bK(\{\tilde\ell_i\}_{L-n+1}^L, \{\ell_i\}_{L-n+1}^L)\leq (1+C/D)^n$;
\item \label{a3}For $L-n+1\leq i\leq L$ and $j>L$, we have $\ell_j-\tilde\ell_i\leq \ell_j-\ell_i$;
\item \label{a4}For $L-n+1\leq i \leq L$, $|\tilde \ell_i-\ell_i|\leq C|\ell_i-\ell_{-}^n|/D\leq Ca/D$;
\item \label{a5}For $L-n+1\leq i<j \leq L$, $|\tilde \ell_j-\tilde \ell_i|\leq  (1+C/D)| \ell_j-\ell_i|$;
\item \label{a6}For $1\leq i<L-n+1\leq j \leq L$, $|\tilde \ell_j- \ell_i|\leq (1+C/D)| \ell_j-\ell_i|$.
\end{enumerate}
\end{claim}
Using \eqref{a3},\eqref{a4}, \eqref{a5} and \eqref{a6} in Claim \ref{c:ltol}, following the same argument as in the proof of \cite[Lemma D.1]{MR3253704}, we  can  compare the densities  and find that
\begin{align}\begin{split}\label{qw}
\bP_L^{\loc}(\ell_1, \ell_2, \cdots,\ell_{L-n},\tilde \ell_{L-n+1}, \cdots, \tilde \ell_L)
\leq& (1+C/D)^{C(n^2+n(L-n))} e^{Cn N^{2\fa}/D}\bP_L^{\loc}(\ell_1, \ell_2, \cdots,\ell_L)\\
\leq &(1+C/D)^{CL^2} \bP_L^{\loc}(\ell_1, \ell_2, \cdots,\ell_L),
\end{split}\end{align}
provided that $D\gg N^{2\fa}$.

For simplification of notations, we denote $\{\ell_i\}=\{\ell_{L-n+1}, \ell_{L-n+2},\cdots, \ell_{L}\}$, and $\{\tilde\ell_i\}=\{\tilde \ell_{L-n+1}, \tilde\ell_{L-n+2},\cdots, \tilde\ell_{L}\}$.
All the discussions above together lead to the following estimate for $Z_{\infty}$,
\begin{align}\begin{split}\label{e:changeofvbound}
&\sum_{\ell_1<\ell_2<\cdots < \ell_{L-n}<\ell_-,\atop \ell_-\leq \tilde\ell_{L-n+1}<\cdots < \tilde \ell_L\leq \ell_+} \bP_L^{\loc}(\ell_1, \ell_2, \cdots,\ell_{L-n},\tilde \ell_{L-n+1}, \cdots, \tilde \ell_L)\\
= &
\sum_{\ell_1<\ell_2<\cdots < \ell_{L-n}<\ell_-,\atop \ell_-\leq \tilde\ell_{L-n+1}<\cdots < \tilde \ell_L\leq \ell_+} \bP_L^{\loc}(\ell_1, \ell_2, \cdots,\ell_{L-n},\tilde \ell_{L-n+1}, \cdots, \tilde \ell_L)\sum_{\{\ell_{i}\}}\bK(\{\tilde \ell_i\},\{\ell_i\})\\
\leq &
(1+C/D)^{CL^2} \sum_{\ell_1<\ell_2<\cdots < \ell_{L-n}<\ell_-,\atop \ell_-\leq \tilde\ell_{L-n+1}<\cdots < \tilde \ell_L\leq \ell_+} \sum_{\{\ell_{i}\}}\bP_L^{\loc}(\ell_1, \ell_2, \cdots,\ell_L)\bK(\{\tilde \ell_i\},\{\ell_i\})\\
\leq &
(1+C/D)^{CL^2} \sum_{\ell_1<\ell_2<\cdots < \ell_{L-n}<\ell_-,\atop \ell_-\leq \ell_{L-n+1}<\cdots < \ell_L\leq \ell_+-a/2D} \bP_L^{\loc}(\ell_1, \ell_2, \cdots,\ell_L)\sum_{\ell_-\leq \tilde\ell_{L-n+1}<\cdots<\tilde\ell_{L}\leq \ell_+}\bK(\{\tilde \ell_i\},\{\ell_i\})\\
\leq &
(1+C/D)^{CL^2+n} \sum_{\ell_1<\ell_2<\cdots < \ell_{L-n}<\ell_-,\atop \ell_-\leq \ell_{L-n+1}<\cdots < \ell_L\leq \ell_+-a/2D} \bP_L^{\loc}(\ell_1, \ell_2, \cdots,\ell_L),
\end{split}
\end{align}
where in the first equality we used that $\bK$ is a transition kernel, in the third line we used \eqref{qw}, and in the last line we used \eqref{a2} in Claim \ref{c:ltol}. \eqref{e:ratio} follows by summing over $n$ from $0$ to $L$ in \eqref{e:changeofvbound}, provided that
$a\gg DL$, $D\gg L^2$, $D\gg N^{2\fa}$ and $s\ll L^{-2}$. With a little algebra, they are satisfied if $L^{4/3}N^{-1/3}\ll s\ll \min\{L^{-2},N^{-\fa}\}$.

The estimate \eqref{e:dis2} follows the previous proof combining with the rigidity estimates assumption, i.e. with probability $1-\exp(-c(\ln N)^2)$, with respect to $\bP_L^{\loc}$, the number of particles on the interval $[\ell_-, \ell_+]$ is bounded by $CN^\fa$. 

\end{proof}

\bibliography{References}{}
\bibliographystyle{plain}

\end{document}